\documentclass[10pt, reqno]{amsart}
\usepackage[margin=1in]{geometry}
\usepackage{amsmath,amssymb}
\usepackage{xcolor} 
\usepackage{amsthm}
\usepackage{amsfonts}
\usepackage{subcaption}
\usepackage{graphicx}
\usepackage{enumerate}

\usepackage[dvpis]{epsfig}
\usepackage[colorlinks=true]{hyperref}
 \hypersetup{urlcolor=blue, citecolor=red}
\usepackage{mdframed}
\newmdenv[
  linecolor=black,
  linewidth=1pt,
  skipabove=10pt,
  skipbelow=10pt,
  innertopmargin=6pt,
  innerbottommargin=6pt,
  innerleftmargin=10pt,
  innerrightmargin=10pt,
  backgroundcolor=white,
  nobreak=false 
]{mybox}

\makeatletter
\newcommand*{\rom}[1]{\expandafter\@slowromancap\romannumeral #1@}
\makeatother

\newtheorem{theorem}{Theorem}[section]

\newtheorem{lemma}[theorem]{Lemma}
\newtheorem{proposition}[theorem]{Proposition}

\theoremstyle{definition}
\newtheorem{definition}[theorem]{Definition}

\theoremstyle{remark}

\newtheorem{remark}[theorem]{Remark}

\numberwithin{equation}{section}

\newcommand{\U}{\mathsf{u}} 

\def\XXint#1#2#3{{\setbox0=\hbox{$#1{#2#3}{\int}$ }
\vcenter{\hbox{$#2#3$ }}\kern-.6\wd0}}

\newcommand{\dx}{\textnormal{d}x}

\newcommand{\dy}{\textnormal{d}y}

\newcommand{\dv}{\textnormal{d}v}

\newcommand{\dw}{\textnormal{d}w}

\newcommand{\ds}{\textnormal{d}s}

\newcommand{\dxi}{\textnormal{d}\xi}
\newcommand{\dtau}{\textnormal{d}\tau}

\newcommand{\ddt}{\frac{\textnormal{d}}{\textnormal{d}t}}
\newcommand{\dds}{\frac{\textnormal{d}}{\textnormal{d}s}}

\newcommand{\bbN}{\mathbb N}

\newcommand{\calC}{\mathcal C}

\newcommand{\calE}{\mathcal E}
\newcommand{\calF}{\mathcal F}

\newcommand{\calH}{\mathcal H}

\newcommand{\calP}{\mathcal P}

\newcommand{\calW}{\mathcal W}

\newcommand{\calZ}{\mathcal Z}

\newcommand{\N}{\mathbb{N}}
\newcommand{\R}{\mathbb{R}}

\newcommand{\T}{\mathbb{T}}

\newcommand{\bp}{\begin{pmatrix}}
\newcommand{\ep}{\end{pmatrix}}

\newcommand{\p}{\partial}

\DeclareMathOperator*{\esssup}{ess\,sup}

\newcommand{\weakto}{\rightharpoonup}

\usepackage{tikz}
\usepackage{tikz-cd}
\usetikzlibrary{matrix}

\begin{document}

\title[Micro-macro and macro-macro limits for controlled leader-follower systems]{Micro-macro and macro-macro limits for controlled leader-follower systems}

\author[Albi]{Giacomo Albi}
\address[Giacomo Albi]{\newline Department of Computer Science\newline
Verona University, Verona 37134, Italy}
\email{giacomo.albi@univr.it}

\author[Choi]{Young-Pil Choi}
\address[Young-Pil Choi]{\newline Department of Mathematics\newline
Yonsei University, Seoul 03722, Republic of Korea}
\email{ypchoi@yonsei.ac.kr}

\author[Piu]{Matteo Piu}
\address[Matteo Piu]{\newline Department of Computer Science\newline
Verona University, Verona 37134, Italy}
\email{matteo.piu@univr.it}

\author[Song]{Sihyun Song}
\address[Sihyun Song]{\newline Department of Mathematics\newline
Yonsei University, Seoul 03722, Republic of Korea}
\email{ssong@yonsei.ac.kr}

\date{\today}


\begin{abstract} 
We study a leader-follower system of interacting particles subject to feedback control and derive its mean-field limits through a two-step passage: first to a micro-macro system coupling leader particles with a follower fluid, and then to a fully continuum macro-macro system. For each limiting procedure, we establish quantitative stability and convergence estimates based on modulated energy methods and Wasserstein distances. These results provide a rigorous foundation for the hierarchical reduction of controlled multi-agent systems. Numerical simulations are presented, including examples with interaction potentials beyond the analytical class considered, to demonstrate the dynamics and support the theoretical results.
\end{abstract}

\maketitle
\tableofcontents

%
%
%
%
%
%
%
%
%
%
%
%
%
%
\section{Introduction}
Interacting agent populations have emerged as powerful tools to capture the dynamics of systems composed of hierarchical classes of individuals in a wide range of applications, such as biology~\cite{BTDAC97,CKFL05,Shen08}, epidemiology~\cite{AWMSAHH20,DPTZ21}, and collective decision processes such as opinion formation~\cite{APZ14,CHL12,DMPW09}. A particularly relevant framework involves \emph{leader-follower interactions}, where agents differ in their ability to influence others and in the dynamics that drive each population. 

A central objective in multi-agent systems is to steer a group of agents toward desired collective behavior by designing appropriate control mechanisms. This modeling paradigm is widely used in contexts where guidance, coordination, or influence are central aspects such as: in crowd control \cite{AB19,AFS21,BLQRS23,Bosse13,CPT14} where selected agents guide collective movement to avoid congestion or ensure safety; in traffic management, where autonomous or informed vehicles act as mobile regulators \cite{Stern18,TZ19}; in social dynamics, such as opinion formation or information dissemination \cite{APZ14,DFWZ24}; and in biological systems, where differentiated roles emerge naturally, for example in swarm behavior or cellular coordination \cite{AF24,BEP21,CLMT25,BHLY24}. In these scenarios, the leader-follower distinction approach proves particularly advantageous, especially in the presence of limited resources or selective interventions \cite{AAMS22, BFK15, CFPT13}, while the design of optimal control action becomes extremely challenging especially when the number of agents involved is very large,
see for example \cite{BBCK18, CFPT13}. For these reasons, both from a modeling and computational point of view, it is important to develop multiscale representations of leader-follower dynamics. On the one hand, such approaches aim to reduce the complexity of high-dimensional systems by approximating collective behavior on coarser scales; on the other hand, they allow for hybrid formulations in which some components, such as key leaders or localized interactions, are better described at the microscopic level, while the rest of the population is captured more efficiently through mesoscopic or macroscopic models \cite{ABCK15,CCH14,CCP17}.  Furthermore, the accurate coupling of different modeling scales is essential not only for theoretical robustness but also for the design of effective and manageable control strategies in large-scale systems. In particular, macroscopic or mean-field approximations play a crucial role in enabling the analysis and optimization of control actions, especially when direct intervention on a large ensemble of microscopic agents becomes computationally prohibitive. 
For further developments in this direction, see~\cite{AAMS22,ACFK17,BPTTR20,FPR14}.

Here, within the framework of leader-follower systems, we focus on leader agents influenced by external controls that aim to balance two competing goals: (i) approaching a preassigned destination, and (ii) maintaining cohesion with the follower population, represented by the empirical center of mass. Specifically, we will consider the following controlled dynamics of the interacting agent system:
\begin{equation} \label{eq:main-with-u}
\begin{split}
    &\ddt x_i = v_i, \quad i=1,\dots, N, \quad t > 0,\\
    &\ddt v_i = \U_i-v_i + \frac{1}{N}\sum_{j=1}^N \nabla W_L (x_j-x_i),\\
    &\ddt y_i = w_i, \quad i=1,\dots, M, \quad t > 0, \\
    &\ddt w_i = \frac{1}{M}\sum_{j=1}^M \nabla W_F(y_j-y_i) + \frac{1}{N}\sum_{j=1}^N \nabla W_C(x_j - y_i) + \frac{1}{N}\sum_{j=1}^N \phi(x_j - y_i) (v_j - w_i),
\end{split}
\end{equation}
where $\{(x_i, v_i)\}_{i=1}^N\in\Omega\times \R^d$ and $\{(y_i, w_i)\}_{i=1}^M\in\Omega\times \R^d$ represent the positions and velocities of the leaders and followers, respectively, with $\Omega = \T^d$ or $\R^d$. The potentials $W_L$, $W_F$, and $W_C$ encode the intra-population and inter-population interactions among leaders, followers, and between the two groups. In addition, $\phi: \Omega \to \R_+$ represents a spatially dependent communication weight that modulates the velocity alignment between leaders and followers.

The control $\U(t)=(\U_1(t),\ldots,\U_N(t))$ acts as a forcing term driving the leaders' population according to the minimization of the following functional,
\begin{equation}\label{eq:main-cost}
\min_{\U(\cdot)\in\mathcal{U}} \frac{1}{N}\sum_{i=1}^N\left(\mathcal L(x_i(T),\langle x \rangle_{\varrho_F^M}(T))+\gamma\int_0^T |\U_i(\tau)|^2\,\dtau\right),
\end{equation}
where $\mathcal U$ is the set of admissible controls $\U$ (e.g. $\U\in L^2(0,T)$), and $\gamma > 0$ penalizes the control energy.  The terminal cost $\mathcal{L}$ is designed to enforce both the convergence to the target and the alignment with the group:
\[
\mathcal L(x_i(t),\langle x \rangle_{\varrho_F^M}(t))=\alpha|x_i(t) - \langle x \rangle_{\varrho_F^M}|^2+\beta|x_i(t)-x_{d_i}|^2.
\]
where $\alpha,\beta \in [0,1]$ are weighting parameters such that $\alpha+\beta = 1$, and $\langle x \rangle_{\varrho_F^M}$ denotes the empirical center of mass of the followers:
\[
\langle x \rangle_{\varrho_F^M} := \int_\Omega x \varrho_F^M (\dx) = \frac1M\sum_{j=1}^M y_j, \quad \mbox{with } \varrho_F^M = \frac1M \sum_{i=1}^M \delta_{y_i}.
\]
However, the minimization of the full time-dependent functional \eqref{eq:main-cost} is computationally demanding, particularly in the presence of nonlinear interactions and large agent populations, see for example \cite{BBCK18}. This motivates the search for simplified strategies that approximate optimal control laws while remaining tractable for large-scale systems. In particular, by minimizing the cost functional \eqref{eq:main-cost} via a greedy approach over a sequence of sequential instantaneous prediction horizon, i.e. replacing the time horizon in \eqref{eq:main-cost}  with $[t,t+h]$ with a $h\ll 1$, one can derive a feedback control of the form:
\begin{equation}\label{eq:u_feedback}
 \U_i(t) = -\alpha(x_i(t) - \langle x \rangle_{\varrho_F^M}(t)) - (1-\alpha)(x_i(t)-x_{d_i}).
 \end{equation}
We refer to Appendix~\ref{appendix:greedycontrol} for a detailed derivation of the feedback control starting from the cost functional~\eqref{eq:main-cost}, and throughout the manuscript, the role of the control will be investigated solely in the form expressed in \eqref{eq:u_feedback}.

 Hence, in this context, the goal of this work is twofold: first, to analyze the influence of the control-driven interaction terms that naturally arising from the instantaneous approximation of \eqref{eq:main-cost}; and second, to rigorously establish a mean-field hierarchy of reduced models, including particle-fluid and fluid-fluid systems, that faithfully approximate the dynamics of the original controlled particle system in successive asymptotic regimes.
 To better clarify the hierarchy proposed, first we will derive from the fully discrete particle system the particle-fluid coupled model, as the number of followers $M \to \infty$; and subsequently to the fully macroscopic description as the number of leaders $N \to \infty$. 
 This hierarchical approach not only provides a systematic framework for model reduction, but also offers insight into the emergent behavior induced by decentralized control strategies.

 The paper is organized as follows. In Section \ref{sec:main_results}, we discuss our main results concerning the hierarchy of mean-field limits for the leader-follower system. In Section \ref{sec:mimi-mima}, we rigorously derive the first-level mean-field limit, establishing the convergence from the fully discrete leader-follower particle system to the intermediate particle-fluid model, where the follower population is described by a continuum. This includes quantitative estimates for both the leader and follower sectors and a stability analysis based on a modulated energy framework. Section \ref{sec:mima-mama} is devoted to the second-level limit, where we rigorously derive the macroscopic fluid-fluid system as the mean-field limit of the intermediate micro-macro model. In Section~\ref{sec_numer}, we present a series of numerical tests designed to validate the theoretical convergence results, investigate the impact of leader–follower coupling, and illustrate complex emergent behaviors such as finite-time blow-up and shock formation. These examples also include configurations with interaction potentials beyond the analytical assumptions considered in the theory. Finally, in Appendix \ref{app_existence}, we establish local-in-time existence and uniqueness of regular solutions for the derived continuum models, thereby providing a well-posed foundation for the limit systems.

\section{Mean-field limits of leader-follower system}\label{sec:main_results}
In this section, we first present our main results concerning the mean-field limit of the interacting leader-follower system, considering the microscopic leader-follower system
\begin{mybox}
\begin{equation}\label{eq:mimi}
\begin{split}
    \ddt x_i^{N,M} &= v_i^{N,M}, \quad i = 1,\dots, N, \quad t > 0,\\[1mm]
    \ddt v_i^{N,M} &= - (1 - \alpha) (x_i^{N,M} - x_{d_i}^{N,M}) - \alpha (x_i^{N,M} - \langle x \rangle_{\varrho_F^M}) - v_i^{N,M} \cr
    &\quad + \frac{1}{N}\sum_{j=1}^N \nabla W_L (x_j^{N,M}-x_i^{N,M}) ,\\
    \ddt y_i^{N,M} &= w_i^{N,M},\quad i = 1,\dots, M, \quad t > 0,\\[1mm]
    \ddt w_i^{N,M} &= \frac{1}{M}\sum_{j=1}^M \nabla W_F(y_j^{N,M}-y_i^{N,M}) + \frac{1}{N}\sum_{j=1}^N \nabla W_C(x_j^{N,M} - y_i^{N,M}) \\
    &\quad + \frac{1}{N}\sum_{j=1}^N \phi(x_j^{N,M} - y_i^{N,M}) (v_j^{N,M} - w_i^{N,M}), 
\end{split}
\end{equation}
\end{mybox}
where we considered the feedback control term \eqref{eq:u_feedback} as driving force for the leader dynamics in system \eqref{eq:main-with-u}.

Throughout the paper, we impose the following assumptions on the interaction potentials and the communication weight:
\begin{itemize}
\item[{\bf (A1)}] The potentials are symmetric, i.e.,
\[
W(x) = W(-x) \quad \mbox{for all $x \in \Omega$ and $W \in \{W_L, W_F, W_C\}$}
\]
and satisfy the boundedness condition
\[
 -\inf_{x\in \Omega} W_L(x) < +\infty, \quad -\inf_{x\in \Omega} W_F(x) < +\infty .
\]
\item[{\bf (A2)}] The interaction potentials are regular:
\[
W_L, W_F, W_C \in C^2(\Omega) \quad \mbox{and} \quad \nabla W_L, \nabla W_F, \nabla W_C \in \calW^{1,\infty}(\Omega) .
\]
\item[{\bf (A3)}] The communication weight satisfies
\[
\phi \in C^1 \cap \calW^{1,\infty}(\Omega).
\]
\end{itemize} 
These assumptions ensure the global-in-time existence and uniqueness of classical solutions to the coupled particle system \eqref{eq:mimi}.

We postpone the discussion on the existence of solutions to the intermediate (micro-macro) and limiting (macro-macro) systems to Appendix \ref{app_existence}, and here focus solely on the derivation of mean-field limits. It is worth noting that the assumptions required for proving convergence in the mean-field limit are, in fact, weaker than those typically needed to establish well-posedness for the limiting PDE systems.

Before stating our main results, we introduce several notational conventions and function spaces that will be used throughout the analysis. For $1 \leq p < \infty$ and $q \in \R$, we define a weighted space $L^p_q(\Omega)$ equipped with norm:
\[
\|f\|_{L^p_q}:= \left(\int_{\Omega} (1+|x|^2)^{\frac q2} |f(x)|^p\,\dx\right)^\frac1p, 
\]
and for $p = \infty$,
\[
\|f\|_{L^\infty_q}:= \esssup_{x \in \Omega} (1+|x|^2)^{\frac q2} |f(x)|.
\]
We also define the corresponding Sobolev space $\calW^{1,p}_q(\Omega)$ to consist of functions $f$ such that both $f$ and $\nabla f$ belong to $L^p_q(\Omega)$. Note that when $\Omega = \mathbb{T}^d$, these weights are unnecessary. In addition, for a Banach space $X$ and a compactly supported function $g: \Omega \to \mathbb{R}$, we use the notation
\[
\|f\|_{X_g} = \sup_{\xi \in {\rm supp}(g)} \|f(\cdot, \xi)\|_{X}.
\]
In the special case where $g = \delta_{x_d}$ for some fixed point $x_d \in \Omega$, corresponding to $x_{d_i} = x_d$ for all $i$, the function $f$ is effectively independent of the variable $\xi$ and only evaluated at $\xi = x_d$. As a result, the norm reduces to
\[
\|f\|_{X_g} = \|f(\cdot)\|_X \quad \mbox{with } f(x) := f(x, x_d).
\]
Finally, we denote by $\pi^x:\R^d\times\R^d\to\R^d$, $\pi(x,v)=x$ the projection into the $x$-component, and for a measure $\nu\in\calP(\R^m)$, $f {}_\# \nu\in\calP(\R^k)$ is the push-forward of $\nu$ by a measurable map $f :\R^m \to \R^k$, $m,k\in\N$.

In what follows, we will abuse notation by not distinguishing between measures and their densities when no confusion arises.

Our first main result concerns the mean-field limit from the microscopic system \eqref{eq:mimi} to the particle-fluid system \eqref{eq:mima}, in the regime where the number of followers tends to infinity ($M \to \infty$), while the number of leaders $N \in \mathbb{N}$ remains fixed:
\begin{mybox}
\begin{equation}\label{eq:mima}
\begin{split}
    &\ddt \bar{x}_i^N = \bar{v}_i^N, \quad i = 1,\dots, N, \quad t > 0,\\[1mm]
    &\ddt \bar{v}_i^N = -(1-\alpha)\left(\bar{x}_i^N - \bar x_{d_i}^N \right) -\alpha (\bar{x}_i^N - \langle x \rangle_{\rho_F^N} ) - \bar{v}_i^N - \nabla W_L *   \bar\varrho_{L}^N(\bar{x}_i^N),\\[1mm]
    &\partial_t \rho_F^N + \nabla \cdot (\rho_F^N u_F^N) = 0, \quad x\in \Omega, \quad t > 0,\\[1mm]
    &\partial_t (\rho_F^N u_F^N) + \nabla \cdot (\rho_F^N u_F^N\otimes u_F^N) = -\rho_F^N\, \nabla W_F * \rho_F^N - \rho_F^N\, \nabla W_C *  \bar\varrho_{L}^N \\
    &\hspace{6cm} + \rho_F^N \iint_{\Omega \times\R^d } \phi(x-y)\,(w - u_F^N(x))\,\bar\mu_L^N(\dy\dw),
\end{split}
\end{equation} 
\end{mybox}
where
\[
\langle x \rangle_{\rho_F^N} = \int_{\Omega} x\rho_F^N(x)\,\dx, \quad \bar\varrho_{L}^N = \frac1N \sum_{i=1}^N \delta_{\bar x_i^N},  \quad \mbox{and}  \quad \bar\mu_L^N = \frac1N \sum_{i=1}^N \delta_{(\bar x_i^N, \bar v_i^N)}.
\]

We now rigorously justify the first step in the mean-field hierarchy by passing from the fully microscopic leader-follower particle system \eqref{eq:mimi} to the intermediate particle-fluid system \eqref{eq:mima}, in the limit as the number of followers $M \to \infty$ with $N$ fixed. The result below establishes convergence of the empirical measures under appropriate regularity and initial alignment assumptions.

\begin{theorem}[Micro-Micro to Micro-Macro Limit]\label{thm:mimi-mima}
Let $T>0$. Let $ \{(x^{N,M}_i,v^{N,M}_i) \}_{i=1}^N$ and $ \{(y^{N,M}_i,w^{N,M}_i) \}_{i=1}^M$ be classical solutions to the particle system \eqref{eq:mimi} on $[0,T]$. Let $(\{(\bar{x}^N_i, \bar{v}^N_i) \}_{i=1}^N, \rho_F^N, u_F^N)$ denote the unique classical solution to the particle-fluid system \eqref{eq:mima} satisfying 
\[
\rho_F^N \in L^\infty(0,T; L^1_2(\Omega)), \quad u_F^N \in L^\infty(0,T;\calW^{1,\infty}(\Omega)).
\] 
Suppose the assumptions {\bf (A1)}--{\bf (A3)} hold, and the initial data satisfy
\begin{align}\label{ass_mimi-mima}
\begin{aligned}
&d_1^2\left(\pi^x_\#\mu_F^{N,M}(0), \rho_F^N(0) \right) + \iint_{\Omega \times \R^d} |w - u_F^N(0,y)|^2 \mu_F^{N,M}(0,\dy \dw)  \cr
&\quad + \iiiint_{(\Omega \times \R^d)^2} (|x-\bar x|^2 + |v-\bar v|^2) \pi_L^{N,M}(0,\dx \dv {\rm d}\bar x {\rm d}\bar v) \to 0 \quad \mbox{as } M \to \infty, 
\end{aligned}
\end{align}
where $d_1$ stands for the first-order Wasserstein distance, $\mu_F^{N,M}$ and $\pi_L^{N,M}$ denote the empirical measure associated to \eqref{eq:mimi} and coupling, respectively:
\[
\mu_F^{N,M}(t) := \frac1M\sum_{i=1}^M \delta_{((y_i^{N,M}(t), w_i^{N,M}(t))}, \quad  \pi_L^{N,M}(t) := \frac1N \sum_{i=1}^N \delta_{\left((x_i^{N,M}(t), v_i^{N,M}(t)), (\bar x_i^{N}(t), \bar v_i^{N}(t))\right)}.
\]
Then the following convergences hold as $M \to \infty$:
\begin{equation} \label{eq: thm micmic conv}
\begin{split}
    &\sup_{0\le t \le T} \iiiint_{(\Omega \times \R^d)^2} (|x-\bar x|^2 + |v-\bar v|^2) \pi_L^{N,M}(t,\dx \dv {\rm d}\bar x {\rm d}\bar v) \to 0,\\
    &\int_{\R^d}\mu_F^{N,M}( \dw) = \frac{1}{M}\sum_{i=1}^M \delta_{y_i^{N,M}}  \weakto \rho_F^N \quad \text{in } L^\infty(0,T;(C_b(\Omega))^*),\\
    &\int_{\R^d}w\mu_F^{N,M}(\dw) = \frac{1}{M}\sum_{i=1}^M w_i^{N,M} \delta_{y_i^{N,M}} \weakto \rho_F^N u_F^N \quad \text{in }  L^\infty(0,T;(  \calW^{1,\infty} (\Omega))^*),\\
    &\int_{\R^d}w \otimes w \mu_F^{N,M}(\dw) = \frac{1}{M}\sum_{i=1}^M w_i^{N,M} \otimes  w_i^{N,M}\delta_{y_i^{N,M}} \weakto \rho_F^N u_F^N \otimes u_F^N \quad \text{in }  L^\infty(0,T;(  \calW^{1,\infty}  (\Omega) )^*),    \\
       & \mu_F^{N,M} = \frac{1}{M}\sum_{i=1}^M \delta_{(y_i^{N,M}, w_i^{N,M})} \weakto \rho_F^N \delta_{u_F^N}\quad \text{in }  L^\infty(0,T;(  \calW^{1,\infty}     (\Omega \times \R^d))^*).
    \end{split}
\end{equation}
Moreover, we have the following quantitative estimate for $t \in [0,T]$:
\begin{align}\label{eq: thm micmic micmac quant}
\begin{aligned}
&d_1^2\left(\pi^x_\#\mu_F^{N,M}(t), \rho_F^N(t) \right) + \iint_{\Omega \times \R^d} |w - u_F^N(t,y)|^2 \mu_F^{N,M}(t,\dy \dw)  \cr
&\quad + \iiiint_{(\Omega \times \R^d)^2} (|x-\bar x|^2 + |v-\bar v|^2) \pi_L^{N,M}(t,\dx \dv {\rm d}\bar x {\rm d}\bar v) \cr
&\qquad \leq  Cd_1^2\left(\pi^x_\#\mu_F^{N,M}(0), \rho_F^N(0) \right)+ C\iint_{\Omega \times \R^d} |w - u_F^N(0,y)|^2 \mu_F^{N,M}(0,\dy \dw)  \cr
&\quad \qquad +C\iiiint_{(\Omega \times \R^d)^2} (|x-\bar x|^2 + |v-\bar v|^2) \pi_L^{N,M}(0,\dx \dv {\rm d}\bar x {\rm d}\bar v),
\end{aligned}
\end{align}
where $C>0$ is a constant independent of $N$ and $M$.
\end{theorem} 

\begin{remark}\label{rmk_mom}
Here we clarify the a priori assumptions on $(\rho_F^N, u_F^N)$ in Theorem \ref{thm:mimi-mima}. The condition $\rho_F^N \in L^\infty(0,T; L^1_2(\Omega))$ ensures that the follower center of mass is well-defined. Moreover, we can easily verify
\begin{align*}
        \frac12\ddt \int_{\Omega} |x|^2 \rho_F^N\,\dx 
        &= \int_{\Omega} x \cdot \rho_F^N u_F^N \,\dx  \\
        &\le \frac12\int_{\Omega} |x|^2 \rho_F^N\,\dx + \frac12\int_{\Omega} \rho_F^N |u_F^N|^2 \,\dx  \\
        &\leq  \frac12\int_{\Omega} |x|^2 \rho_F^N\,\dx + \frac12\|u_F^N\|_{L^\infty}^2\|\rho_F^N\|_{L^1},
\end{align*}
so that an application of Gr\"onwall's lemma yields the propagation in time of $\|\rho_F^N(t)\|_{L^1_2}$.

Regarding the assumption on $u_F^N$, note that its boundedness is not required to establish the quantitative estimate \eqref{eq: thm micmic micmac quant} itself, but it is necessary to justify the convergence statements in \eqref{eq: thm micmic conv}.
\end{remark}

We next present the mean-field limit from the intermediate (micro-macro) system \eqref{eq:mima} to the following macro-macro system as the number of leaders tends to infinity, i.e., $N \to \infty$:
\begin{mybox}
\begin{align}\label{eq:mama}
\begin{aligned}
    &\p_t \bar \rho_L + \nabla \cdot (\bar \rho_L \bar u_L) = 0, \quad (x,\xi)\in \Omega\times\textnormal{supp}(g), \quad t>0,\\
    &\p_t (\bar \rho_L \bar u_L) + \nabla \cdot (\bar \rho_L \bar u_L \otimes \bar u_L) = -(1-\alpha) \bar \rho_L (x-\xi) - \alpha \bar \rho_L (x - \langle x \rangle_{\bar \rho_F}) - \bar \rho_L \bar u_L\\
    &\qquad\qquad\qquad\qquad\qquad\qquad\qquad - \bar\rho_L \iint_{\Omega \times \Omega} \nabla W_L(x-y)\bar\rho_L(y,\xi_*)\, g(\dxi_*) \dy,\\
    & \p_t \bar\rho_F + \nabla \cdot (\bar\rho_F \bar u_F) = 0,  \quad x\in \Omega, \quad t > 0,\\
    &\p_t (\bar\rho_F \bar u_F) + \nabla \cdot (\bar\rho_F \bar u_F \otimes \bar u_F ) = - \bar\rho_F \nabla W_F * \bar \rho_F - \bar\rho_F \nabla W_C * (\pi^x_{\#}\bar\rho_L) \\
    &\qquad \qquad \qquad \qquad \qquad \qquad \qquad \qquad + \bar\rho_F \iint_{\Omega\times\Omega} \phi(x-y) (\bar u_L (y,\xi) - \bar u_F(x) ) \bar \rho_L(y,\xi) \,   g(\dxi) \dy .
\end{aligned}
\end{align} 
\end{mybox}
 As before, we will always write $\nabla = \nabla_x$ to refer to differentiation with respect to the spatial variable, whenever there is no risk of confusion. In the above, $\bar \rho_L(t,x,\xi)$ denotes the density of leaders located at position $x \in \Omega$ and associated with target point $\xi \in \Omega$ at time $t > 0$. The variable $\xi$ is distributed according to a fixed probability measure $g \in \mathcal{P}(\Omega)$, which encodes the distribution of desired destinations. The macroscopic velocity field $\bar u_L(t,x,\xi)$ represents the average velocity of leaders at $(x,\xi)$ and is derived as the limit of empirical velocities from the intermediate system. Similarly, $(\bar \rho_F, \bar u_F)$ describe the macroscopic density and velocity of the follower population.

Regarding this result, we shall assume without loss of generality that
\begin{equation*}
    \iint_{\Omega\times\Omega} \bar\rho_L(t,x,\xi) \,g(\dxi)\dx = 1
\end{equation*}
for all $t\in [0,T]$. In addition, we assume that the target points $x_{d_i}$ for the leaders (particles) in \eqref{eq:mima} are well-prepared in the following sense:
\begin{equation}\label{eq: g compat 2}
\begin{split}
    &\forall N\in \bbN, \quad  1 \le i \le N, \qquad x_{d_i} \in \textnormal{supp}(g); \\
    &d_{1}\left( \frac{1}{N} \sum_{i=1}^N \delta_{\bar x_i^N(0)} \otimes \delta_{x_{d_i}} \; , \; \bar\rho_L (0) (\dx \otimes g(\dxi)) \right) \xrightarrow[N\to\infty]{} 0.
\end{split}
\end{equation}
This ensures that the micro-macro and macro-macro systems are compatible. This assumption is essential to control the propagation of the target information in the macroscopic limit.

Finally, we introduce the following additional regularity condition on the communication weight function.
\begin{itemize}
\item[{\bf (A4)}] The weight function $\phi$ satisfies the weighted regularity condition:
\[
\phi \in \calW^{1,\infty}_1(\Omega),
\]
i.e. $|x|\phi, |x||\nabla \phi| \in L^\infty(\Omega)$.
\end{itemize} 

 Under the above assumptions, we are now in a position to rigorously justify the convergence from the intermediate particle-fluid system to the macroscopic two-fluid system. The next result establishes the mean-field limit as the number of leaders $N$ tends to infinity, showing that the empirical measures associated with the leader particles converge to the macroscopic density $\bar \rho_L$, and that the corresponding velocity fields also converge appropriately. The result includes both qualitative convergence and a quantitative estimate controlled by the initial data and the discrepancy between the empirical distribution of target positions and the limit measure $g$.

\begin{theorem}[Micro-Macro to Macro-Macro Limit]\label{thm:mima-mama}
    Let $T>0$. Let $(\{(\bar{x}^N_i, \bar{v}^N_i) \}_{i=1}^N, \rho_F^N, u_F^N)$ be a classical solution to the micro-macro system \eqref{eq:mima} on $[0,T]$, and let $(\bar \rho_L, \bar u_L, \bar\rho_F, \bar u_F)$ be the unique classical solution to the macro-macro system \eqref{eq:mama} satisfying 
    \[
    \bar \rho_F \in L^\infty(0,T; L^1_2(\Omega)), \quad \bar u_F \in L^\infty(0,T;\calW^{1,\infty}(\Omega)), \quad \bar u_L \in L^\infty(0,T; L^\infty_{g}(B)) \mbox{ for all } B \subset \subset \Omega, 
    \]
    and 
    \begin{equation} \label{eq: uL assump}
    \nabla_x \bar u_L \in L^\infty(0,T;L^\infty_g(\Omega)), \quad \nabla_\xi \bar u_L \in L^\infty(0,T;L^\infty_g(\Omega)).
    \end{equation}
Suppose the assumptions {\bf (A1)}--{\bf (A4)} hold, and that the compatibility condition \eqref{eq: g compat 2} is fulfilled. We assume also that the initial data satisfy
    \begin{align*}
&    d_1^2(\bar\rho^N_F(0), \bar\rho_F(0)) + \int_{\Omega} \rho^N_F(0) | u^N_F(0) - \bar u_F(0)|^2\, \dx \cr
&\quad + \iiint_{\Omega \times \R^d \times \Omega} |v - \bar u_L(0,x,\xi)|^2 \bar\varsigma_{L}^N(0,\dx\dv\dxi) \to 0 \quad \mbox{as } N \to \infty, 
    \end{align*}
    where $\bar\varsigma_L^N$ denotes the full empirical measure
\begin{equation*}
    \bar\varsigma_L^N(t,\dx\dv\dxi) := \frac{1}{N}\sum_{i=1}^N \delta_{\bar x_i(t)}(\dx) \otimes \delta_{\bar v_i(t)}(\dv) \otimes \delta_{x_{d_i}}(\dxi).
\end{equation*} 
    Then the following convergences hold as $N \to \infty$:
\[
    \begin{split}
        &\frac{1}{N}\sum_{i=1}^N \delta_{\bar x_i^N} \weakto \int_{\Omega} \bar \rho_L \, g(\dxi)  \quad \text{in }  L^\infty(0,T;(C_b(\Omega))^*)), \\ 
        &\frac{1}{N}\sum_{i=1}^N \bar v_i^N \delta_{\bar x_i^N} \weakto  \int_{\Omega} \bar \rho_L \bar u_L \, g(\dxi) \quad \text{in }L^\infty(0,T;(C_c^1(\Omega))^*)  ,  \\
        &\frac{1}{N}\sum_{i=1}^N \bar v_i^N \otimes v_i^N \delta_{\bar x_i^N} \weakto  \int_{\Omega} \bar \rho_L \bar u_L \otimes \bar u_L \,g(\dxi)  \quad \text{in }  L^\infty(0,T; (C_c^1(\Omega))^*)),  \\ 
                &\frac{1}{N}\sum_{i=1}^N  \delta_{(\bar x_i^N, \bar v_i^N)} \weakto  \int_{\Omega} \bar\rho_L \delta_{\bar u_L} g(\dxi) \quad \text{in }L^\infty(0,T; (C^1 \cap W^{1,\infty}(\Omega\times\R^d))^*) ,  \\
        &\bar \rho_F^N  \weakto \bar\rho_F  \quad \text{in }  L^\infty(0,T;(C_b(\Omega))^*),\\
        &\bar\rho_F^N \bar u_F^N  \weakto \bar\rho_F  \bar u_F, \quad \bar\rho_F^N \bar u_F^N \otimes \bar u_F^N  \weakto \bar\rho_F  \bar u_F \otimes \bar u_F    \quad \text{in }  L^\infty(0,T; (C_c^1(\Omega))^* ).
    \end{split}
\]
Moreover, we have the following quantitative estimate for $t \in [0,T]$:
        \begin{align}\label{est_miMA}
        \begin{aligned}
&    d_1^2(\bar\rho^N_F(t), \bar\rho_F(t)) + \int_{\Omega} \rho^N_F(t) | u^N_F(t) - \bar u_F(t)|^2\, \dx \cr
&\quad +  d_1^2 \left( \frac{1}{N}\sum_{i=1}^N \delta_{\bar x_i^N(t)} \otimes \delta_{x_{d_i}}   , \; \bar\rho_L(t) (\dx \otimes g(\dxi)) \right)  \\
&\quad + \iiint_{\Omega \times \R^d \times \Omega} |v - \bar u_L(t,x,\xi)|^2 \bar\varsigma_{L}^N(t,\dx\dv\dxi)  \cr
&\qquad  \leq C d_1^2(\bar\rho^N_F(0), \bar\rho_F(0)) + C \int_{\Omega} \bar \rho^N_F(0) | \bar u^N_F(0) - \bar u_F(0)|^2\, \dx \cr
&\qquad \quad + C  d_1^2 \left( \frac{1}{N}\sum_{i=1}^N \delta_{\bar x_i^N(0)} \otimes \delta_{x_{d_i}}  \; , \; \bar\rho_L(0)(\dx \otimes g(\dxi)) \right)  \\
&\qquad \quad + C  \iiint_{\Omega \times \R^d \times \Omega} |v - \bar u_L(0,x,\xi)|^2 \bar\varsigma_{L}^N(t,\dx\dv\dxi)  ,
\end{aligned}
    \end{align}
where the constant $C > 0$ is independent of $N$.
\end{theorem}

\begin{remark} An argument nearly identical to that in Remark \ref{rmk_mom} shows that $\|\bar \rho_F(t)\|_{L^1_2}$ propagates in time.
\end{remark}

\begin{remark}\label{rem: en}

The second condition in \eqref{eq: uL assump} requires that the macroscopic leader velocity $\bar u_L$ is uniformly Lipschitz continuous with respect to the target variable $\xi$. This assumption is not artificial. In Appendix \ref{app_ffs}, we provide a local-in-time construction of classical solutions to the macro--macro system under quite general assumptions on the measure $g$, for which such regularity properties naturally arise.

In contrast to the micro-micro to micro-macro limit (Theorem \ref{thm:mimi-mima}), the analysis of the micro-macro to macro-macro limit (Theorem \ref{thm:mima-mama}) requires particular care in estimating the leader velocity field $\bar u_L$. A central difficulty stems from the presence of the unbounded linear forcing term $-x$ in the leader momentum equation of the macro-macro system.

More precisely, $\bar u_L$ satisfies
\[
\partial_t \bar u_L  + \bar u_L \cdot \nabla \bar u_L = - x + \text{(other terms)},
\]
so that in the whole-space case $\Omega = \R^d$, one cannot generally expect $\bar u_L(t,\cdot,\xi)$ to belong to $L^p(\R^d)$ for any $1 \leq p \leq \infty$ because of the unbounded linear drift. This immediately shows that a global $L^\infty$ bound for $\bar u_L$ is unrealistic. However, our proof requires sufficient regularity of $\bar u_L$ to perform quantitative error estimates in the modulated energy approach. To address this, we carefully track the effect of the $-x$ term and design our estimates so that only {\it local-in-space} boundedness of $\bar u_L$ is needed. Indeed, one can rigorously show that local-in-space bounds of the form:
\[
\bar u_L(\cdot,\cdot,\xi) \in L^\infty(0,T; L^\infty_{\rm loc}(\Omega))
\]
can be propagated in time despite the unbounded forcing. This is the strongest bound that can be reasonably expected in unbounded domains.

Achieving the required error estimates under this weaker, local condition is a key technical contribution of our argument. It ensures that, even without global integrability, the convergence result remains valid. For clarity and to guarantee well-posedness of the limiting system, the rigorous local-in-time existence and uniqueness theory presented in Appendix \ref{app_ffs} is developed under periodic boundary conditions $\Omega = \T^d$, where such global issues are avoided.
\end{remark}

\begin{remark}
When the domain is periodic, i.e., $\Omega = \T^d$, the assumptions on the communication weight $\phi$ can be simplified. For instance, the condition $x \phi(x) \in \mathcal{W}^{1,\infty}(\T^d)$ is no longer needed and can be replaced by the standard assumption $\phi \in \mathcal{W}^{1,\infty}(\T^d)$.
\end{remark}

The study of mean-field limits for interacting particle systems has been a central theme in mathematical physics and PDE theory, with significant advances in the derivation of kinetic and fluid-type equations from microscopic dynamics. Depending on the regularity of the interaction potentials and communication kernels, different methods have been developed to rigorously justify the passage from discrete to continuum descriptions. Classical results for Vlasov-type kinetic equations and aggregation models under Lipschitz or bounded interaction fields can be found in \cite{BH, CCHS19, CFRT10, CFTV, CHL17, HL09, HT08, JW16, JW17, JW18, LP17, Spohn}, while more recent work addresses the challenges posed by singular kernels and weaker regularity, both in kinetic and aggregation-diffusion regimes; see for instance \cite{CCH14, CFI25, Deu16, HJ07, HJ15, S20}. In addition, there have been focused efforts on deriving pressureless Euler-type systems from particle models, particularly those governed by nonlocal attraction-repulsion or alignment forces. Notably, the modulated energy framework developed in \cite{S20} provides the mean-field limit of interacting particle systems with singular potentials, such as the Coulomb or super-Coulombic Riesz potentials. This methodology inspired subsequent developments, including the work in \cite{CC21}, where the authors rigorously derive a multi-dimensional pressureless Euler system with nonlocal interactions as the mean-field limit of a second-order swarming model.

In this spirit, the derivation of the micro-macro and macro-macro systems from the particle model can be understood as a two-scale mean-field limit. In this hierarchy, the follower population transitions from a discrete particle description to continuum-level dynamics governed by kinetic or fluid-type equations, while the leader population retains its particle character in the micro-macro limit and is further approximated by fluid dynamics in the macro-macro limit.

A related methodological framework was developed in \cite{piu1,piu2}, where hybrid micro-macroscopic limits were rigorously derived for multilane traffic models, coupling continuous microscopic dynamics with discrete lane-changing events. This approach highlights an interesting analogy between interacting populations in multi-agent systems and vehicles moving across multiple lanes, where the discrete-continuous interplay plays a key role in the emergent macroscopic behavior.

Our proof strategy builds upon the modulated kinetic energy framework introduced in \cite{CC21, S20}, suitably adapted to the present setting involving two interacting populations. The main analytical tool is a discrete modulated kinetic energy functional designed to measure the discrepancy between the particle-level system and its macroscopic counterpart. For the particle-to-fluid limit of the follower population, we define the functional
\begin{align} \label{Disc Mod Energy}
\begin{aligned}
    \calE(\calZ_F^{N,M}| Z_F^N)(t) &:= \frac{1}{2}\iint_{\Omega\times\R^d} | u_F^N(t,y) - w|^2 \mu_F^{N,M}(t,\dy\dw)\cr
    &= \frac{1}{2M}\sum_{i=1}^M | u_F^N(t,y_i^{N,M}(t)) - w_i^{N,M}(t)|^2,
\end{aligned}
\end{align}
where 
\[
\calZ^{N,M}_L = \{(y_i^{N,M}, w_i^{N,M})\}_{i=1}^M \quad \mbox{and} \quad  Z_F^N = \bp \rho_F^N \\ \rho_F^N u_F^N \ep
\]
are regular solutions to the system \eqref{eq:mimi} and \eqref{eq:mima}, respectively. However, due to the presence of nonlocal interaction terms in both position and velocity, the modulated energy functional alone is insufficient to close the estimates. To address this, we also incorporate the first-order Wasserstein distance and bounded-Lipschitz distance. Although these are equivalent when $\Omega = \T^d$, we use both $d_{\rm BL}$ and  $d_1$ for generality, taking into account that $d_{\rm BL} \leq d_1$ in general. The use of $d_1$ is particularly important for estimating the discrepancy in centers of mass. In the first stage of the mean-field hierarchy, the leaders remain at the discrete level, and we measure the discrepancy using the $\ell^2$-distance. 

In the second stage, the leader population transitions from a particle-level description to a continuum one that incorporates a target variable $\xi \in \Omega$. This variable reflects heterogeneity in the leader targets and is distributed according to a fixed probability measure $g(\xi)$. While each discrete leader is associated with a fixed target $x_{d_i} \in \text{supp}(g)$, the continuum leader density and velocity fields, $\bar\rho_L(t,x,\xi)$ and $\bar u_L(t,x,\xi)$, retain explicit $\xi$-dependence. To overcome this discrepancy, we introduce a discrete modulated kinetic energy in which the continuum leaders are also targeted to the $x_{d_i}$:
\begin{equation} \label{Averaged Disc Mod Energy}
\begin{split}
    \calF(\bar\calZ_L^N| \bar Z_L)(t) &:= \frac{1}{2}\iiint_{\Omega\times\R^d \times \Omega} |v - \bar u_L(t,x,\xi)|^2 \bar\varsigma_L^N(t,\dx\dv\dxi)  \cr
    &=  \frac{1}{2N}\sum_{i=1}^N \int_{\Omega} \left|\bar v_i(t) - \bar u_L(t,\bar x_i(t),x_{d_i})\right|^2,
\end{split}
\end{equation} 
where 
\[
\bar \calZ^N_L = \{(\bar x_i^N, \bar v_i^N)\}_{i=1}^N \quad \mbox{and} \quad   \bar Z_L = \bp \bar \rho_L \\ \bar \rho_L \bar u_L \ep
\]
are regular solutions to the system \eqref{eq:mima} and \eqref{eq:mama}, respectively. This expression quantifies the velocity discrepancy and accounts for the fact that it is physically natural to compare micro and macro states which are aimed at the same target point. However, we must additionally take into consideration the fact that the continuum leaders are actually distributed along the measure $g$, and not the discrete target points $x_{d_i}$. The error which arises from this subtlety is controlled by using the fact that the discrete target points are compatible with $g$, in the sense of \eqref{eq: g compat 2}. The approach is, in a broad sense, similar to what was done in the earlier stage, but for this result we do not employ the bounded Lipschitz distance when quantifying the difference between the particle and continuum descriptions. Only the Wasserstein distance is additionally utilized. This is due to a technical difficulty which arises from the fact that we cannot expect $\bar u_L$ to be globally bounded in the spatial variable, see Remark \ref{rem: en} above.    

For the convergence of the follower population at the fluid level, we adopt the classical modulated energy:
\[
\frac12 \int_{\Omega} \rho^N_F(t) | u^N_F(t) - \bar u_F(t)|^2\,\dx.
\] 
To propagate convergence forward in time, we derive an evolution inequality satisfied by this modulated energy functional. The key step involves differentiating the functional with respect to time and carefully estimating the resulting terms. This requires handling commutator-type errors, bounding the nonlocal interaction terms, and utilizing the regularity of the limiting velocity fields. These estimates allow us to apply Gr\"onwall's inequality and conclude the quantitative convergence.

%
%
%
%
%
%
%
%
%
%
%
%
%
%
\section{From the particle system to the micro-macro limit}\label{sec:mimi-mima}
In this section, we provide a detailed proof of Theorem \ref{thm:mimi-mima}, which concerns the mean-field limit from the interacting leader-follower particle system \eqref{eq:mimi} to the coupled particle-fluid system \eqref{eq:mima}. This step corresponds to the first level of the mean-field hierarchy, where the follower population becomes a continuum while the leaders remain discrete.

Since the number of leaders $N$ is fixed throughout this section, for notational simplicity we omit the superscripts and the dependence on $N$. We begin by recalling the form of the limiting micro-macro system:
\begin{align} \label{eq: micmac lim}
\begin{aligned}
    &\ddt \bar{x}_i = \bar{v}_i, \quad i = 1,\dots, N, \quad t > 0,\\[1mm]
    &\ddt \bar{v}_i = -(1-\alpha)\left(\bar{x}_i - \bar x_{d_i} \right) -\alpha (\bar{x}_i - \langle x \rangle_{\rho_F} ) - \bar{v}_i - \nabla W_L *  \bar\varrho_{L}^N(\bar{x}_i),\\[1mm]
    &\partial_t \rho_F + \nabla \cdot (\rho_F u_F) = 0, \quad x\in \Omega, \quad t > 0,\\[1mm]
    &\partial_t (\rho_F u_F) + \nabla \cdot (\rho_F u_F\otimes u_F) = -\rho_F\, \nabla W_F * \rho_F - \rho_F\, \nabla W_C *   \bar\varrho_{L}^N \\
    &\hspace{5cm} + \rho_F \iint_{\Omega\times\R^d} \phi(x-y)\,(w - u_F(x))\,\bar\mu_L^N(\dy\dw),
\end{aligned}
\end{align}
where
\[
\langle x \rangle_{\rho_F} = \int_{\Omega} x\rho_F(x)\,\dx, \quad \bar\varrho_{L}^N = \frac1N \sum_{i=1}^N \delta_{\bar x_i},  \quad \mbox{and}  \quad \bar\mu_L^N = \frac1N \sum_{i=1}^N \delta_{(\bar x_i, \bar v_i)}.
\]

The proof proceeds in two main steps. First, we derive uniform estimates and convergence bounds for the leaders, whose dynamics are governed by the discrete particle system. Second, we establish quantitative stability estimates for the followers by comparing the empirical measure associated with the discrete system to the corresponding continuum limit, using a modulated energy framework. The key difficulty lies in coupling the two populations and controlling the interaction error terms, which we handle carefully through a combination of $\ell^2$ estimates, Wasserstein distance bounds, and Gr\"onwall-type arguments.

Throughout this section, we assume that the assumptions {\bf (A1)}--{\bf (A3)} hold. For notational convenience, we introduce the following notation, which will be used consistently in the subsequent analysis:
 \begin{equation*}
 \mathcal{Z}^M_L := \{(x_i^M, v_i^M)\}_{i=1}^N, \quad \mathcal{Z}^M_F := \{(y_i^M, w_i^M)\}_{i=1}^M,
\quad
  \bar{\mathcal{Z}}_L := \{(\bar x_i, \bar v_i)\}_{i=1}^N, \quad Z_F := \bp \rho_F \\ \rho_F u_F \ep.
 \end{equation*}
%
%
%
%
%
%
%
%
%
%
%
%
%
%
\subsection{Stability estimates for the leader sector}
We begin this part by deriving uniform moment bounds for the microscopic particle system \eqref{eq:mimi}, focusing in particular on the behavior of the leaders. These estimates play a key role in establishing convergence in the subsequent mean-field analysis. To control the discrete modulated energy between the particle system and the limiting model, we require uniform-in-time bounds on the $\ell^2$-norms of the system trajectories. We state the following auxiliary lemma.
\begin{lemma} \label{lem: l2 bdd}
Let $T>0$, and let $(\calZ^M_L, \calZ^M_F)$ be a classical solution to \eqref{eq:mimi} on the time interval $[0,T]$. Then there exists a constant $C>0$, independent of $N$ and $M$, such that
    \begin{align*}
        \frac{1}{N}\sum_{i=1}^N (|x_i|^2 + |v_i|^2) + \frac{1}{M}\sum_{i=1}^M (|y_i|^2+|w_i|^2) \le C.
    \end{align*}
\end{lemma}
\begin{proof}
We begin with the energy estimate for the leader particles. A direct computation yields
    \begin{align*}
        &\ddt \left(\frac{1}{2N}\sum_{i=1}^N (|x_i|^2 + |v_i|^2) \right)\\
        &\quad = \frac{1}{N}\sum_{i=1}^N v_i \cdot \left((1-\alpha) x_{d_i} - v_i + \alpha  \langle x \rangle_{\varrho_F^M} + \frac{1}{N}\sum_{j=1}^N \nabla W_L(x_j-x_i) \right) \\
        &\quad = \frac{1}{N}\sum_{i=1}^N \Big( (1-\alpha) x_{d_i}\cdot v_i + \alpha v_i \cdot \langle x \rangle_{\varrho_F^M} \Big) + \frac{1}{2N^2}\sum_{i,j=1}^N (v_i-v_j)\cdot \nabla W_L(x_j-x_i) - \frac{1}{N}\sum_{i=1}^N |v_i|^2.
    \end{align*}
Here the last line is owing to the odd symmetry of $\nabla W_L$. Applying Young's and Jensen's inequalities, we obtain
    \begin{equation}\label{eq: x,v l2}
    \begin{split}
        &\ddt \left(\frac{1}{2N}\sum_{i=1}^N (|x_i|^2+|v_i|^2) + \frac{1}{2N^2}\sum_{i,j=1}^N W_L(x_j-x_i) + m_L \right) + \frac{1}{N}\sum_{i=1}^N |v_i|^2 \\
        &\quad \le \frac{1-\alpha}{N}\sum_{i=1}^N |x_{d_i}|^2 + \frac{1}{N}\sum_{i=1}^N |v_i|^2 + \alpha |\langle y \rangle|^2 \\
        &\quad \le C_\alpha + \frac{1}{N}\sum_{i=1}^N |v_i|^2 + \frac{\alpha}{M}\sum_{j=1}^M |y_j|^2, 
        \end{split}
    \end{equation}
    where (recalling Assumption \textbf{(A1)})
    \[
    m_L := - \frac12\inf_{x \in \Omega} W_L(x).
    \]
Proceeding similarly for the follower dynamics, we get
    \begin{align*}
        &\ddt \left(\frac{1}{2M}\sum_{i=1}^M (|y_i|^2 + |w_i|^2)  + \frac{1}{2M^2}\sum_{i,j=1}^M W_F(y_j-y_i) + m_F \right)\\
        &\quad = \frac{1}{M}\sum_{i=1}^M y_i\cdot w_i + \frac{1}{M}\sum_{i=1}^M w_i \cdot \left( \frac{1}{N}\sum_{j=1}^N \nabla W_C(x_j-y_i) + \frac{1}{N}\sum_{j=1}^N \phi(x_j-y_i)(v_j-w_i) \right) \\
        &\quad \le \frac{1}{2M}\sum_{i=1}^M (|y_i|^2 + |w_i|^2) + \frac{1}{MN}\sum_{i=1}^M \sum_{j=1}^N \Big( C|w_i| |x_j-y_i| + \|\phi\|_{L^\infty}|v_j-w_i| \Big)\\
        &\quad \lesssim \frac{1}{2M}\sum_{i=1}^M (|y_i|^2 + |w_i|^2 ) + \frac{1}{2M}\sum_{j=1}^N (|x_j|^2 + |v_j|^2), 
    \end{align*}
    where
    \[
    m_F := - \frac12\inf_{x \in \Omega} W_F(x).
    \]
    Adding this with \eqref{eq: x,v l2} and then applying Gr\"onwall's lemma completes the proof.
\end{proof}

We now estimate the $\ell^2$-distance between the leaders in the microscopic system and those in the intermediate particle-fluid model. Define
\begin{align*}
    &L_X(t) := \left(\frac{1}{N}\sum_{i=1}^N |x_i-\bar x_i|^2\right)^{1/2}, \quad L_V(t) := \left(\frac{1}{N}\sum_{i=1}^N |v_i - \bar v_i|^2 \right)^{1/2}.
\end{align*}
\begin{lemma} \label{lem: ddt L}Let $(\calZ_L^M, \calZ_F^M)$ and $(\bar \calZ_L, Z_F)$ be classical solutions to the system \eqref{eq:mimi} and \eqref{eq: micmac lim} on the time interval $[0,T]$, respectively. Then there exists a constant $C>0$, independent of $N$ and $M$, such that 
    \begin{align*}
        &\ddt \{L_X(t)\}^2 \le 2L_X(t) L_V(t),\\
        &\ddt \{L_V(t)\}^2 \le  C\left( \{L_X(t)\}^2 + \{L_V(t)\}^2\right) + d_{1}^2(\varrho_F^M(t),\rho_F(t)).
    \end{align*}
\end{lemma}
\begin{proof}
By the Cauchy--Schwarz inequality, we get
    \begin{align*}
        \ddt \{L_X(t)\}^2 &= \frac{2}{N}\sum_{i=1}^N (x_i-\bar x_i)\cdot (v_i- \bar v_i) \le 2 L_X(t) L_V(t).
    \end{align*}
A similar computation and use of Young's inequality gives
    \begin{align*}
       \ddt \{L_V(t)\}^2 
        &= \frac{2}{N}\sum_{i=1}^N (v_i-\bar v_i) \cdot \Bigg( (\bar x_i - x_i) + (\bar v_i - v_i) \\
        &\qquad \qquad \qquad \qquad \qquad \quad +\frac{1}{N}\sum_{j=1}^N \Big(\nabla W_F(x_j - x_i) - \nabla W_F(\bar x_j - \bar x_i)\Big) + \alpha \left(\langle x \rangle_{\rho_F} - \langle x \rangle_{\varrho_F^M} \right) \Bigg) \\
        &\le \{L_X(t)\}^2 - \{L_V(t)\}^2 + \frac{2\|\nabla W_F\|_{\rm Lip}}{N^2} \sum_{i,j=1}^N \Big(|x_j - \bar x_j| + |x_i - \bar x_i|\Big) |v_i - \bar v_i| \\
        &\quad + \alpha \{L_V(t)\}^2 + \alpha   |\langle x \rangle_{\rho_F}- \langle x \rangle_{\varrho_F^M}|^2 \\
        &\le C\left( \{L_X(t)\}^2 + \{L_V(t)\}^2 \right) + \alpha   |\langle x \rangle_{\rho_F}- \langle x \rangle_{\varrho_F^M}|^2.
    \end{align*}
The last term on the right-hand side of the above is estimated as
\[
    \begin{split}
        |\langle x \rangle_{\rho_F}- \langle x \rangle_{\varrho_F^M}| &= \left|\int_{\Omega} x (\rho_F \dx - \varrho_F^M (\dx) ) \right| \le d_{1}(\rho_F,\varrho_F^M).
    \end{split}
\]
This completes the proof.
\end{proof}

%
%
%
%
%
%

\subsection{Stability estimates for follower sector}
We now turn to the quantitative estimates for the followers. This analysis complements the estimates for the leaders and is essential to establishing the convergence of the full system. In particular, we derive an evolution inequality for the modulated energy between the discrete follower population and its continuum limit, which plays a crucial role in the mean-field approximation.

To begin, we recall the governing equations for the follower density and velocity field, assuming $\rho_F > 0$:
\begin{align*}
&    \p_t \rho_F + \nabla\cdot (\rho_F  u_F) = 0,\\
&    \p_t  u_F + (u_F \cdot \nabla)u_F =  - \nabla W_F *  \rho_F - \nabla W_C *  \bar \varrho_L^N + H(\bar \mu_L^N)(\cdot, u_F(t,\cdot)),
\end{align*}
where
\[
H(\bar\mu_L^N)(y,w) := \iint_{\Omega\times\R^d} \phi(x-y) (v-w)\bar\mu_L^N(\dx\dv).
\]
From this formulation, we can compute the time derivative of $u_F$ along the trajectory of each follower:
\begin{equation} \label{eq: pt u}
\begin{split}
   \p_t (u_F(t,y_i(t)))  
    &= \p_t u_F(t,y_i(t)) + (w_i(t) \cdot \nabla_y)  u_F(t,y_i(t)) \\
    &=  - (\nabla_y u_F(t,y_i(t)))^T  u_F(t,y_i(t)) - \nabla W_F *  \rho_F(t,y_i(t)) - \nabla W_C *  \bar \varrho_L^N  (y_i(t)) \\
    &\qquad + H(\bar\mu_L^N)(y_i(t), u_F(t,y_i(t)))  + (w_i(t) \cdot \nabla_y ) u_F(t,y_i(t)).
\end{split}
\end{equation}

We now derive the following estimate for the time evolution of the discrete modulated energy between the particle and continuum follower states. 

\begin{lemma} \label{lem: disc efct est}Let $T>0$, and let $(\calZ_L^M, \calZ_F^M)$ and $(\bar \calZ_L, Z_F)$ be classical solutions to the system \eqref{eq:mimi} and \eqref{eq: micmac lim} on the time interval $[0,T]$, respectively. Then, the discrete modulated energy functional in \eqref{Disc Mod Energy} satisfies the estimate
\[
    \begin{split}
        \ddt \calE(\calZ_F^M| Z_F)(t) &\le C \calE(\calZ_F^M| Z_F)(t) + C\Big( d_1(\varrho_F^M(t),\bar\rho_F(t)) + L_X(t) \Big) \sqrt{\calE(\calZ_F^M| Z_F)(t)} \\
        &\qquad + C L_V(t) \sqrt{\calE(\calZ_F^M| Z_F)(t)} + C L_X(t) \sqrt{\calE(\calZ_F^M| Z_F)(t)}
    \end{split}
\]
    for constants $C>0$ independent of $N$ and $M$, where
    \[
    \varrho_F^M = \frac1M\sum_{i=1}^M \delta_{y_i}.
    \]
\end{lemma}
\begin{proof}
A direct computation, utilizing the calculations in \eqref{eq: pt u}, show
\begin{align*}
   \ddt \calE(\calZ_F^M| Z_F)  
    &= \frac{1}{M}\sum_{i=1}^M [u_F(t,y_i(t))-w_i(t)]\cdot\left( (w_i(t) - u_F(t,y_i(t)) ) \cdot \nabla_y u_F(t,y_i(t))  \right)\\
    &\quad + \frac{1}{M}\sum_{i=1}^M [u_F(t,y_i(t)) - w_i(t)] \cdot \left(\nabla W_F * \varrho_F^M(y_i(t)) - \nabla W_F * \rho_F(y_i(t)) \right) \\
    &\quad + \frac{1}{M}\sum_{i=1}^M [u_F(t,y_i(t)) - w_i(t)] \cdot \left(\nabla W_C * \varrho_L^N(y_i(t)) - \nabla W_C * \bar \varrho_L^N(y_i(t)) \right) \\
    &\quad + \frac{1}{M}\sum_{i=1}^M [u_F(t,y_i(t)) - w_i(t)]\cdot \left(\frac{1}{N}\sum_{j=1}^N \Big[H(\bar{\mu}_L^N)(y_i(t),u_F(t,y_i(t))) - \phi(x_j-y_i)(v_j-w_i) \Big] \right)\\
    &=: I_1 + I_2 + I_3 + I_4. 
\end{align*}
We now estimate each term $I_k$ separately.

$\bullet$ Estimate of $I_1$: We simply have that
\begin{align*}
    I_1 \le \|\nabla u_F\|_{L^\infty} \frac{1}{M}\sum_{i=1}^M |u_F(t,y_i(t)) - w_i(t)|^2 = 2\|\nabla u_F\|_{L^\infty} \calE(\calZ_F^M| Z_F).
\end{align*}

$\bullet$ Estimate of $I_2$: By the Cauchy--Schwarz inequality,
\begin{align*}
    I_2 \le \|\nabla W_F\|_{\calW^{1,\infty}} d_1(\varrho_F^M(t), \rho_F(t)) \sqrt{\calE(\calZ_F^M| Z_F)} .
\end{align*}

$\bullet$ Estimate of $I_3$: By definition of the empirical measure, we obtain
\begin{align*}
    &|\nabla W_C *  \varrho_L^N (y_i) - \nabla W_C *  \bar \varrho_L^N(y_i)| \cr
    &\quad \le \left|\int_{\Omega} \nabla W_C(y_i-y) ( \varrho_L^N -  \bar\varrho_L^N)(\dx) \right| \le \frac1N\sum_{j=1}^N \left|\nabla W_C(y_i - x_j) - \nabla W_C(y_i - \bar x_j)\right| \le \|\nabla W_C\|_{\calW^{1,\infty}} L_X(t).
\end{align*}
This gives us
\begin{align*}
    I_3 \le \|\nabla W_C\|_{\calW^{1,\infty}} L_X(t) \sqrt{\calE(\calZ_F^M| Z_F)}.
\end{align*}

$\bullet$ Estimate of $I_4$: We note that
\begin{align*}
    H(\bar\mu_L^N)(y_i(t),u_F(t,y_i(t))) &:= \iint_{\Omega\times\R^d} \phi(x-y_i(t))(v-u_F(t,y_i(t))) \bar\mu_L^N(\dx\dv) \\
    &= \frac{1}{N}\sum_{j=1}^N \phi(\bar x_j(t) - y_i(t)) (\bar{v}_j(t) - u_F(t,y_i(t)) ). 
\end{align*}
Therefore, the term in parentheses arising in $I_4$ can be written and estimated with Jensen's inequality as
\begin{align*}
    &\frac{1}{N}\sum_{j=1}^N \Big[ \phi(\bar x_j - y_i) (\bar{v}_j - u_F(t,y_i) ) - \phi(x_j-y_i)(v_j-w_i) \Big]\\
    &\quad = \frac{1}{N}\sum_{j=1}^N \Big[ \phi(\bar x_j - y_i) (w_i - u_F(t,y_i)) + \phi(\bar x_j - y_i)(\bar v_j - w_i)  - \phi(x_j-y_i)(v_j-w_i) \Big] \\
    &\quad = \frac{1}{N} \sum_{j=1}^N \Big[ \phi(\bar x_j - y_i) (w_i - u_F(t,y_i)) + \phi(\bar x_j - y_i)(\bar v_j-v_j)  + \left[\phi(\bar x_j - y_i) - \phi(x_j-y_i) \right] (v_j - w_i) \Big]\\
    &\quad \le   \|\phi\|_{L^\infty} |w_i - u_F(t,y_i)|  + \|\phi\|_{L^\infty} L_V(t) + \frac{\|\phi\|_{\rm Lip}}{N}\sum_{j=1}^N |\bar x_j - x_j| \, |v_j- w_i|. 
\end{align*}
Utilizing this and Lemma \ref{lem: l2 bdd}, we find  
\begin{align*}
    I_4 &\le \|\phi\|_{L^\infty}\calE(\calZ_F^M|Z_F) + \|\phi\|_{L^\infty} L_V(t) \sqrt{\calE(\calZ_F^M| Z_F)}   + C\|\phi\|_{\rm Lip}L_X(t) \sqrt{\calE(\calZ_F^M|Z_F)}.
\end{align*}
Collecting all estimates for the $I_i$, we complete the proof.
\end{proof}

Up to this point, Lemmas \ref{lem: ddt L}--\ref{lem: disc efct est} yield the following evolution inequality,
\begin{align*}
    &\ddt \Big(\calE(\calZ_F^M| Z_F) + \{L_X(t)\}^2 + \{L_V(t)\}^2 \Big)  \le C \calE(\calZ_F^M| Z_F) + C d_{1}^2(\varrho_F^M(t),\rho_F(t)) + C\left(\{L_X(t)\}^2 + \{L_V(t)\}^2 \right).
\end{align*}
An application of Gr\"onwall's lemma then provides a bound on the cumulative error up to time $t \in [0,T]$:
\begin{equation}\label{eq: prerestore}
\begin{split}
    &\calE(\calZ_F^M| Z_F)(t) + \{L_X(t)\}^2 + \{L_V(t)\}^2  \\
    &\quad \le C \left(\calE(\calZ_F^M| Z_F)(0) + \{L_X(0)\}^2 + \{L_V(0)\}^2\right) + C \int_0^t d_{1}^2(\varrho_F^M(s),\rho_F(s)) \,\ds. 
\end{split}
\end{equation}

To close this estimate, it remains to bound the term $d_{1}^2(\varrho_F^M(t),\rho_F(t))$ in terms of the modulated energy. This is provided by the following auxiliary result from \cite{CC21} (see also \cite[Proposition 3.1]{Ch21}), whose proof can be adapted to our setting.
\begin{lemma}[Proposition 2.2, \cite{CC21}] \label{lem: d1 sigma} Let $T>0$, and let $\calZ_F^M$ and $ Z_F$ be classical solutions to the system \eqref{eq:mimi} and \eqref{eq: micmac lim} on the time interval $[0,T]$, respectively. Then, there exists a constant $C>0$, dependent only on $\|\nabla u_F\|_{L^\infty}$ and $T$, such that
    \begin{align*}
        d_{1}^2(\varrho_F^M(t),\rho_F(t)) \le C d_{1}^2(\varrho_F^M(0),\rho_F(0)) + C\int_0^t \calE(\calZ_F^M| Z_F)(s) \, \ds. 
    \end{align*}
\end{lemma}

\begin{remark}
    Although Lemma \ref{lem: d1 sigma} was originally stated in terms of the bounded Lipschitz distance in \cite{CC21}, a careful inspection of its proof shows that the same reasoning applies for the squared $d_1$ distance as stated above.
\end{remark}

%
%
%
%
%
%
\subsection{Proof of Theorem \ref{thm:mimi-mima}: Convergence toward the micro-macro limit system}\label{sec:23}
We are now ready to conclude the proof of Theorem \ref{thm:mimi-mima}, which provides the quantitative mean-field limit from the fully microscopic system to the intermediate particle-fluid model. Building upon the stability estimates derived in the preceding lemmas, we combine the evolution inequality for the modulated energy with a complementary estimate for the Wasserstein distance to close the convergence argument.

As a direct application of Lemma \ref{lem: d1 sigma}, we revisit the inequality \eqref{eq: prerestore} and append the $d_1^2$-distance term to the left-hand side. This yields the bound
\begin{align*}
    &\calE(\calZ_F^M| Z_F)(t) + \{L_X(t)\}^2 + \{L_V(t)\}^2 + d_{1}^2(\varrho_F^M(t),\rho_F(t))  \\
    &\quad \le C \left(\calE(\calZ_F^M| Z_F)(0) + \{L_X(0)\}^2 + \{L_V(0)\}^2\right) \\
    &\qquad + C \int_0^t d_1^2(\varrho_F^M(s),\rho_F(s)) \,\ds + Cd_1^2(\varrho_F^M(0) ,\rho_F(0)) + C\int_0^t \calE(\calZ_F^M| Z_F)(s)\,\ds .
\end{align*}
Applying Gr\"onwall's lemma to the above expression, we obtain the uniform-in-time quantitative estimate
\begin{align*}
    &\calE(\calZ_F^M| Z_F)(t) + \{L_X(t)\}^2 + \{L_V(t)\}^2 + d_1^2(\varrho_F^M(t),\rho_F(t))  \\
    &\quad \le e^{CT}\Big(\calE(\calZ_F^M|Z_F)(0) + \{L_X(0)\}^2 + \{L_V(0)\}^2 + d_1^2(\varrho_F^M(0),\rho_F(0)) \Big)
\end{align*}
for all $t \in [0,T]$, thereby proving the desired stability estimate stated in Theorem \ref{thm:mimi-mima}.

To complete the proof, we now deduce the convergence of the empirical measure $\varrho_F^M$ and the momentum measure $\mu_F^{M}$ in the dual of continuous test functions. Assume that the initial data satisfies the convergence condition \eqref{ass_mimi-mima}, that is,
\begin{align*}
    \calE(\calZ_F^M|Z_F)(0) \to 0, \quad L_X(0) \to 0, \quad L_V(0) \to 0, \quad \text{and} \quad d_1(\varrho_F^M(0), \rho_F(0)) \to 0,
\end{align*}
as $M \to \infty$. Then, from the estimate above, we immediately obtain
\begin{align*}
    \sup_{0\le t \le T} d_1\left(\varrho_F^M(t), \rho_F(t)\right) \to 0 \quad \text{as } M \to \infty,
\end{align*}
which, in turn, implies convergence in the dual of $C_b(\Omega)$:
\begin{align*}
    \frac{1}{M}\sum_{i=1}^M \delta_{y_i^{M}} \weakto \rho_F \quad \text{in } L^\infty(0,T; (C_b(\Omega))^*).
\end{align*}

For the remaining convergence results, we rely on \cite[Lemma 2.1]{CC21} and the inequality $d_{\rm BL} \leq d_1$ to control the discrepancy between kinetic and macroscopic quantities. Specifically, we obtain
\[
d_{\rm BL}^2\left(\int_{\R^d}w\mu_F^{M}(\dw),\, \rho_F u_F \right) \leq C\iint_{\Omega \times \R^d} |w - u_F(y)|^2 \mu_F^{M}(\dy \dw) + Cd_{1}^2(\varrho_F^M,\rho_F), 
\]
\[
\begin{split}
&d_{\rm BL}^2\left(\int_{\R^d}w\otimes w \mu_F^{M}(\dw),\, \rho_F u_F \otimes u_F \right) \cr
&\quad \leq C\left(\iint_{\Omega \times \R^d} |w - u_F(y)|^2 \mu_F^{M}(\dy \dw)\right)^2 + \iint_{\Omega \times \R^d} |w - u_F(y)|^2 \mu_F^{M}(\dy \dw) +  Cd_{1}^2(\varrho_F^M,\rho_F), 
\end{split}
\]
and
\[
d_{\rm BL}^2 (\mu_F^M, \rho_F \delta_{u_F}) \leq C\iint_{\Omega \times \R^d} |w - u_F(y)|^2 \mu_F^{M}(\dy \dw) + Cd_{1}^2(\varrho_F^M,\rho_F)
\]
for some $C>0$ independent of $N$ and $M$. The right-hand sides are quantitatively controlled and vanish as $M \to \infty$, in view of our earlier stability estimates. This concludes the desired convergence statements.

%
%
%
%
%
%
%
%
%
%
%
%
\section{From micro-macro description to fully continuum dynamics}\label{sec:mima-mama}

In this section, we investigate the limiting behavior of the micro-macro system \eqref{eq:mima} as the number of leader agents $N$ tends to infinity. The goal is to derive the corresponding macroscopic description of both the leader and follower populations, referred to as the macro-macro system. This limiting system captures the emergent dynamics when the leader agents are no longer modeled as discrete particles, but rather as a continuum described by a density function.

The passage from micro-macro to macro-macro consists of replacing the empirical measures of the leaders with their mean-field limits, leading to a fully continuum model. This reduction is essential for both analytical tractability and computational feasibility in large-scale systems, and it also allows for a deeper understanding of the macroscopic interactions between the leader and follower populations.

We rigorously derive the limiting equations and establish quantitative estimates for the convergence. To this end, we analyze the evolution of modulated energy functionals between the discrete and continuum systems, and provide Wasserstein-type estimates on the corresponding empirical measures. We also introduce suitable compatibility conditions and structural assumptions to ensure well-posedness of the limit and consistency between the different layers of the mean-field hierarchy.

The macro-macro system obtained in the limit reads:
\begin{equation} \label{eq: macmac lim}
\begin{split}
    &\p_t \bar \rho_L + \nabla_x\cdot (\bar \rho_L \bar u_L) = 0, \quad (x,\xi)\in \Omega\times\textnormal{supp}(g), \quad t>0,\\
    &\p_t (\bar \rho_L \bar u_L) + \nabla_x \cdot (\bar \rho_L \bar u_L \otimes \bar u_L) = -(1-\alpha) \bar \rho_L (x-\xi) - \alpha \bar \rho_L (x - \langle x \rangle_{\bar \rho_F}) - \bar \rho_L \bar u_L\\
    &\qquad\qquad\qquad\qquad\qquad\qquad\qquad - \bar \rho_L \iint_{\Omega \times \Omega} \nabla W_L(x-y)\bar \rho_L(t,y,\xi_*)\,g(\textnormal{d}\xi_*) \dy ,\\
    & \p_t \bar \rho_F + \nabla_x \cdot (\bar \rho_F \bar u_F) = 0,  \quad x\in \Omega, \quad t > 0,\\
    &\p_t (\bar \rho_F \bar u_F) + \nabla_x \cdot (\bar \rho_F \bar u_F \otimes \bar u_F ) = - \bar \rho_F \nabla W_F *_x \bar \rho_F - \bar \rho_F \iint_{\Omega \times \Omega} \nabla W_C(x-y) \bar \rho_L(t,y,\xi_*)g(\textnormal{d}\xi_*) \dy \\
    &\qquad \qquad \qquad \qquad \qquad \qquad \qquad + \bar \rho_F \iint_{\Omega \times \Omega} \phi(x-y) (\bar u_L (t,y,\xi_*) - \bar u_F(t,x) ) \bar \rho_L(t,y,\xi_*)g(\textnormal{d}\xi_*) \dy.
\end{split}
\end{equation}

Throughout this section, we assume as usual that the assumptions {\bf (A1)}--{\bf (A4)} hold. For notational simplicity, we introduce the following notation throughout this section:
 \[
  \bar{\mathcal{Z}}_L^N := \{(\bar x_i, \bar v_i)\}_{i=1}^N, \quad \bar Z_L := \bp \bar \rho_L \\ \bar \rho_L \bar u_L \ep,  \quad Z_F := \bp \rho_F \\ \rho_F u_F \ep, \quad \bar Z_F := \bp \bar \rho_F \\ \bar \rho_F \bar u_F \ep.
 \]

\subsection{Stability estimates for the leader sector}

We now derive quantitative estimates for the leaders, which play a central role in the passage from the micro-macro system \eqref{eq:mima} to its macroscopic limit. As in the previous sections, our main tool is the modulated energy method, allowing us to compare the discrete particle description with the continuum model and to track how discrepancies evolve over time. This subsection mirrors the analysis performed for the followers, and complements the overall convergence framework by controlling the leader dynamics in the mean-field regime.

We begin by observing that on the support $\bar \rho_L>0$, the macroscopic velocity field satisfies
\begin{align*}
    \p_t \bar u_L + (\bar u_L\cdot \nabla) \bar u_L = -(1-\alpha)(x-\xi) - \alpha(x-\langle x \rangle_{\bar \rho_F}) - \bar u_L - \iint_{\Omega \times \Omega} \nabla W_L(x-y)\bar \rho_L(t,y,\xi_*)g(\textnormal{d}\xi_*)\dy,
\end{align*}
and thus along the particle trajectory  
\begin{align*}
    \p_t \bar u_L(t,\bar x_i(t),x_{d_i})  &= (\nabla_x \bar u_L(t,\bar x_i(t),x_{d_i}) ) (\bar v_i(t) - \bar u_L(t,\bar x_i(t),x_{d_i}))  -(1-\alpha)(\bar x_i(t)- x_{d_i}) - \alpha (\bar x_i(t) - \langle x \rangle_{\bar \rho_F}) \\
    &\quad - \bar u_L(t,\bar x_i(t),x_{d_i})  - \iint_{\Omega \times \Omega} \nabla W_L (\bar x_i(t) - y) \bar \rho_L(t,y,\xi_*) g(\textnormal{d}\xi_*)\, \dy  ,
\end{align*} 
from which we obtain the following lemma.
\begin{lemma} \label{lem: macmac disc ener} Let $T>0$, and let $(\bar \calZ_L^N, Z_F)$ and $(\bar Z_L, \bar Z_F)$ be classical solutions to the system \eqref{eq: micmac lim} and \eqref{eq: macmac lim} on the time interval $[0,T]$, respectively. Then, the averaged discrete modulated energy functional between the leaders \eqref{Averaged Disc Mod Energy} satisfies the estimate
    \begin{align*}
        \calF(\bar \calZ_L^N|\bar Z_L)(t) &\le \calF(\bar \calZ_L^N|\bar Z_L)(0) + C \int_0^t \calF(\bar \calZ_L^N | \bar Z_L)(s)\,\ds  + C\int_0^t d_{1}^2(\rho_F(s), \bar \rho_F(s)) \,\ds \\
        &\quad + C  d_{BL}^2\left( \frac{1}{N}\sum_{j=1}^N \delta_{\bar x_j(t)} \otimes \delta_{x_{d_i}}\; , \; \bar\rho_L(t)  (\dx \otimes g(\dxi)) \right) , 
    \end{align*}
    with the constants $C>0$ dependent only on $\|\nabla_x \bar u_L\|_{L^\infty(0,T;L^\infty_g)}$, $\|\nabla W_L\|_{\calW^{1,\infty}}$, and $T$.
\end{lemma}
\begin{proof}
The proof proceeds by direct differentiation of the energy functional and term-by-term estimation of the resulting contributions. In particular, we obtain three main terms arising from the chain rule:
    \begin{align*}
        &\ddt \calF(\bar \calZ_L^N|\bar Z_L) \\
        &\quad = \frac{\alpha}{N}\sum_{i=1}^N (v_i-\bar u_L(\bar x_i,x_{d_i})) \cdot (\langle x \rangle_{\rho_F} - \langle x \rangle_{\bar \rho_F}) \\
        &\qquad -\frac{1}{N}\sum_{i=1}^N (v_i - \bar u_L(\bar x_i,x_{d_i})) \cdot ((v_i - \bar u_L(\bar x_i,x_{d_i}))\cdot \nabla_x) \bar u_L(\bar x_i,x_{d_i})) \\
        &\qquad + \frac{1}{N}\sum_{i=1}^N \Bigg((v_i - \bar u_L(\bar x_i,x_{d_i}))  \left(-  \nabla W_L * \bar\varrho_L^N(\bar x_i) +  \iint_{\Omega \times \Omega} \nabla W_L(\bar x_i - y)  \bar \rho_L(y,\xi_*)g(\textnormal{d}\xi_*)\dy \right) \Bigg) \\
        &\quad =: I_1 + I_2 + I_3.
    \end{align*} 
    $\bullet$ Estimate of $I_1$: Using the Cauchy--Schwarz inequality and $\alpha \le 1$, we show
    \begin{align*}
        I_1 \le \calF(\bar \calZ_L^N|\bar Z_L) + |\langle x \rangle_{\rho_F} - \langle x \rangle_{\bar \rho_F}|^2.
    \end{align*}
    We remark that the second term is estimated easily with Kantorovich duality as
\begin{align*}
    |\langle x \rangle_{\rho_F} - \langle x \rangle_{\bar \rho_F}|^2 &= \left|\int_{\Omega} x(\rho_F(x) - \bar \rho_F(x))\dx\right|^2 \le d_1^2( \rho_F, \bar \rho_F).
\end{align*}
    $\bullet$ Estimate of $I_2$: We easily find
    \begin{align*}
        I_2 \le \sup_{0\le t \le T} \|\nabla_x \bar u_L(t)\|_{L^\infty_g} \calF(\bar \calZ_L^N|\bar Z_L).
    \end{align*}
    $\bullet$ Estimate of $I_3$: The nonlocal interaction is controlled using the bounded Lipschitz distance as
    \begin{align*}
        &\left|\nabla W_L * \bar\varrho_L^N(\bar x_i) -  \iint_{\Omega \times \Omega} \nabla W_L(\bar x_i - y)  \bar \rho_L(y,\xi_*)g(\textnormal{d}\xi_*)\dy\right|\\
        &\quad = \left| \iint_{\Omega\times\Omega} \nabla W_L(\bar x_i - y) \left(\frac{1}{N}\sum_{j=1}^N \delta_{\bar x_j(t)}  \otimes \delta_{x_{d_i}}  - \bar\rho_L(t) (\dy \otimes g(\dxi)) \right) \right|\\
        &\quad \le \|\nabla W_L\|_{\calW^{1,\infty}} d_{1} \left( \frac{1}{N}\sum_{j=1}^N \delta_{\bar x_j(t)}   \otimes \delta_{x_{d_i}}    , \; \bar\rho_L(t) (\dy \otimes g(\dxi)) \right).
    \end{align*}
    
    This, the Cauchy--Schwarz inequality, and the arithmetic-geometric inequality yield
    \begin{align*}
        I_3 &\le \calF(\bar \calZ_L^N|\bar Z_L) + C d_{BL}^2\left( \frac{1}{N}\sum_{j=1}^N \delta_{\bar x_j(t)} \otimes \delta_{x_{d_i}}  , \; \bar\rho_L(t) (\dy \otimes g(\dxi)) \right) .
    \end{align*}
    Combining all terms and applying Gr\"onwall's lemma yields the desired bound.
\end{proof}

We also record a useful auxiliary estimate, adapted from \cite[Proposition 2.2]{CC21}. Although the proof follows more or less the same ideas of \cite{CC21}, we include the details below for completeness, due to the rather distinctive-looking assertion of the statement.

\begin{lemma} Let $T>0$, and let $\bar\calZ_L^N$ and $\bar Z_L$ be classical solutions to the system \eqref{eq: micmac lim} and \eqref{eq: macmac lim} on the time interval $[0,T]$, respectively. Then, we have
\begin{equation*}
\begin{split}
    &d_{1}^2\left(\frac{1}{N}\sum_{j=1}^N \delta_{\bar x_j(t)}   \otimes \delta_{x_{d_j}}  , \; \bar\rho_L(t)(\dx \otimes g(\dxi)) \right) \\
    &\quad \le C \int_0^t \calF(\bar \calZ_L^N | \bar \calZ_L)(s) \,\ds + C d_{1}^2 \left( \frac{1}{N}\sum_{j=1}^N \delta_{\bar x_j(0)} \otimes \delta_{x_{d_i}} , \; \bar\rho_L(0)(\dx \otimes g(\dxi)) \right)
    \end{split}
\end{equation*}
for a constant $C>0$ depending only on $T$, $\|\nabla_x \bar u_L\|_{L^\infty_{t,x,\xi}}$, and $\|\nabla_\xi\bar u_L\|_{L^\infty_{t,x,\xi}}$.
\end{lemma}

\begin{proof}
    For each fixed $(x,\xi)\in \Omega\times\Omega$, denote by $\eta(\cdot;0,x,\xi)$ the flow which solves
    \begin{equation*}
        \ddt \eta(t;0,x,\xi) = \bar u_L(t; \eta(t;0,x,\xi), \xi), \quad \eta(0;0,x,\xi) = x.
    \end{equation*}
    We remark that since the field $x\mapsto \bar u_L(t,x,\xi)$ is Lipschitz continuous, the existence and uniqueness of the flow $\eta$ follows thanks to the classical Picard--Lindel\"of theorem.
    
    For any Lipschitz continuous function $\psi\in C^1(\Omega\times\Omega)$, we compute
    \begin{align*}
        &\left|\iint_{\Omega\times\Omega} \psi(y,\xi) \left( \frac{1}{N}\sum_{j=1}^N \delta_{\bar x_j(t)} \otimes \delta_{x_{d_i}} - \bar\rho_L(t)(\dy \otimes g(\dxi)) \right) \right| \\
        &\quad = \left| \frac{1}{N}\sum_{j=1}^N \psi(\bar x_j(t) , x_{d_i}) - \iint_{\Omega\times \Omega} \psi(\eta(t;0,y,\xi),\xi) \bar\rho_L(0,y,\xi) \,g(\dxi)\dy \right| \\
        &\quad  \le \left|\frac{1}{N}\sum_{j=1}^N \psi(\bar x_j(t), x_{d_i}) - \frac{1}{N}\sum_{j=1}^N \psi(\eta(t;0,\bar x_j(0),x_{d_i}),x_{d_i})  \right| \\
        &\qquad + \left|\frac{1}{N}\sum_{j=1}^N \psi(\eta(t;0,\bar x_j(0),x_{d_i}),x_{d_i}) - \iint_{\Omega\times\Omega} \psi(\eta(t;0,y,\xi),\xi) \bar\rho_L(0,y,\xi) \,g(\dxi)\dy  \right| \\
        &\quad  =: J_1 + J_2.
    \end{align*}
    In order to estimate $J_1$, we first observe that
    \begin{equation} \label{eq: char gron}
    \begin{split}
        &|\bar x_j(t) - \eta(t;0,\bar x_j(0),x_{d_i})| \\
        &\quad \le \int_0^t |\bar v_j(s) - \bar u_L(s,\eta(s;0,\bar x_j(0),x_{d_i})) | \,\ds \\
        &\quad \le \int_0^t |\bar v_j(s) - \bar u_L(s, \bar x_j(s), x_{d_i}) | \,\ds
        + \int_0^t |\bar u_L (s, \bar x_j(s), x_{d_i}) - \bar u_L(s,\eta(s;0,\bar x_j(0),x_{d_i})) | \,\ds \\
        &\quad \le \int_0^t |\bar v_j(s) - \bar u_L(s,\bar x_j(s), x_{d_i})|\,\ds + \sup_{s\in [0,t]} \|\nabla_x \bar u_L(s)\|_{L^\infty_{x,\xi}} \int_0^t |\bar x_j(s) - \eta(s;0,\bar x_j(0),x_{d_i})|\,\ds.
    \end{split}
    \end{equation}
    Gr\"onwall's inequality therefore provides
    \begin{equation*}
        |\bar x_j(t) - \eta(t;0,\bar x_j(0),x_{d_i})| \le C_T \int_0^t |\bar v_j(s) - \bar u_L(s,\bar x_j(s), x_{d_i}) |\,\ds.
    \end{equation*}
    It follows that
    \begin{align*}
        J_1 &\le \|\nabla_x \psi\|_{L^\infty_{x,\xi}} \frac{1}{N}\sum_{j=1}^N |\bar x_j(t) - \eta(t;0,\bar x_j(0),x_{d_i}) | \\
        &\le C_T \|\nabla_x \psi\|_{L^\infty_{x,\xi}} \frac{1}{N}\sum_{j=1}^N \int_0^t |\bar v_j(s) - \bar u_L(s,\bar x_j(s),x_{d_i}), x_{d_i} ) | \,\ds \\
        &\le C_T \|\nabla_x \psi\|_{L^\infty_{x,\xi}} \int_0^t \sqrt{\calF(\bar\calZ_L^N|\bar\calZ)(s)} \,\ds.
    \end{align*}
    On the other hand, we note that $J_2$ can be easily computed as
    \begin{align*}
        J_2 &= \left| \iint_{\Omega\times\Omega} \psi(\eta(t;0,x,\xi),\xi) \left(\frac{1}{N}\sum_{j=1}^N \delta_{\bar x_j(0)} \otimes \delta_{x_{d_i}} - \bar\rho_L(0) (\dy \otimes g(\dxi))   \right) \right|.
    \end{align*}
    It is therefore enough to prove that $(x,\xi)\mapsto \psi(\eta(t;0,x,\xi),\xi)$ is Lipschitz continuous. We observe that, for any $x_1,x_2,\xi\in \Omega$:
    \begin{align*}
        |\psi(\eta(t;0,x_2,\xi),\xi) - \psi(\eta(t;0,x_1,\xi),\xi)| &\le \|\psi\|_{\calW^{1,\infty}} |\eta(t;0,x_2,\xi) - \eta(t;0,x_1,\xi)| \\
        &\le C_T \|\psi\|_{\calW^{1,\infty}} \|\nabla_x \bar u_L\|_{L^\infty_{t,x,\xi}} |x_2 - x_1|,
    \end{align*}
    the last inequality following by giving a Gr\"onwall argument in exactly the same way as in \eqref{eq: char gron}. For any $x,\xi_1,\xi_2\in \Omega$, we then have
    \begin{align*}
        |\psi(\eta(t;0,x,\xi_2),\xi_2) - \psi(\eta(t;0,x,\xi_1),\xi_1)| &\le |\psi(\eta(t;0,x,\xi_2),\xi_2) - \psi(\eta(t;0,x,\xi_2),\xi_1) | \\
        &\quad + |\psi(\eta(t;0,x,\xi_2),\xi_1) - \psi(\eta(t;0,x,\xi_1),\xi_1) | \\
        &\le \|\nabla_{x,\xi} \psi\|_{L^\infty_{x,\xi}} \left( |\xi_2 - \xi_1| + |\eta(t;0,x,\xi_2) - \eta(t;0,x,\xi_1)| \right) \\
        &\le \|\nabla_{x,\xi}\psi\|_{L^\infty_{x,\xi}} \left(|\xi_2 - \xi_1| + \int_0^t |\bar u_L(s;0,x,\xi_2) - \bar u_L(s;0,x,\xi_1)| \,\ds  \right) \\
        &\le \|\nabla_{x,\xi}\psi\|_{L^\infty_{x,\xi}} (1+t \|\bar u_L\|_{L^\infty_{t,x} \calW^{1,\infty}_\xi } ) |\xi_2 - \xi_1| .
    \end{align*}
    We deduce that there exists $C_T > 0$ such that
    \begin{equation*}
        \textnormal{$(x,\xi)\mapsto \psi(\eta(t;0,x,\xi),\xi)$ is Lipschitz continuous with Lipschitz constant $\le C_T \|\nabla_{x,\xi}\psi\|_{L^\infty_{x,\xi}}$.}
    \end{equation*}
    Hence, $J_2$ satisfies
    \begin{align*}
        |J_2| &\le C_T \|\nabla_{x,\xi}\psi\|_{L^\infty_{x,\xi}} d_{1}\left( \frac{1}{N}\sum_{j=1}^N \delta_{\bar x_j(0)}  \otimes \delta_{x_{d_i}}   , \; \bar\rho_L(0) (\dy \otimes g(\dxi))  \right) 
    \end{align*}
    and this completes the proof.
\end{proof}

%
%
%
%
%
%

\subsection{Stability estimates for follower sector}
We now complement the previous convergence analysis by establishing quantitative estimates for the follower dynamics. Specifically, we aim to control the distance between two follower distributions, $\rho_F$ and $\bar \rho_F$, together with their associated velocity fields $u_F$ and $\bar u_F$. These estimates are crucial for closing the mean-field hierarchy and confirming the consistency between the coupled micro-macro and macro-macro formulations.

We begin with a Wasserstein-type estimate for the density difference, which follows in the same spirit of Lemma \ref{lem: d1 sigma}.

\begin{lemma} Let $T>0$, and let $Z_F$ and $\bar Z_F$ be classical solutions to the system \eqref{eq: micmac lim} and \eqref{eq: macmac lim} on the time interval $[0,T]$, respectively. Then, there exists a constant $C>0$, depending only on $\|\rho_F\|_{L^1}$,  $\|\nabla \bar u_F\|_{L^\infty}$, and $T$, such that 
    \begin{align*}
        &d_{1}(\rho_F(t), \bar \rho_F(t)) \le C d_{1}(\rho_F(0), \bar \rho_F(0)) + C \left(\int_0^t \int_{\Omega} \rho_F |u_F - \bar u_F|^2\,\dx\ds\right)^{1/2}.
    \end{align*}
\end{lemma}
 
 Next, we estimate the macroscopic modulated kinetic energy, which is the $L^2$-norm of the velocity difference between $u_F$ and $\bar u_F$ weighted by the density $\rho_F$. The proof follows a standard energy estimate argument, which involves integrating the difference of the momentum equations and controlling all resulting nonlinear terms.
 
\begin{lemma} Let $T>0$, and let $(\bar \calZ_L^N, Z_F)$ and $(\bar Z_L, \bar Z_F)$ be classical solutions to the system \eqref{eq: micmac lim} and \eqref{eq: macmac lim} on the time interval $[0,T]$, respectively.
Assume that
\begin{align*}
    &\phi(x) \in \calW^{1,\infty}_1(\Omega).
\end{align*}
Then the following inequality holds:
    \begin{align*}
        \int_{\Omega} \rho_F |u_F - \bar u_F|^2\dx &\le \int_{\Omega}\rho_F(0)|u_F(0)-\bar u_F(0)|^2\,\dx+ C\int_0^t d_{\rm BL}^2(\rho_F(s), \bar \rho_F(s))\,\ds \\
        &\quad + C \int_0^t \int_{\Omega} \rho_F |u_F - \bar u_F|^2\,\dx\ds  + C \int_0^t  d_1^2 \left(\frac{1}{N}\sum_{i=1}^N \delta_{\bar x_i(s)} \otimes \delta_{x_{d_i}}  , \; \bar\rho_L(s) (\dy \otimes g(\dxi))  \right)  \ds \\
        &\quad + C \int_0^t \calF(\bar \calZ_L^N|\bar Z_L)(s)\,\ds
    \end{align*}
    with the constant $C>0$ independent of $N$.  
\end{lemma}
\begin{proof}
    A direct computation shows that
    \begin{align*}
        &\frac{1}{2} \ddt\int_{\Omega} \rho_F |u_F - \bar u_F|^2\dx \\
        &\quad = \frac{1}{2}\int_{\Omega} -\nabla\cdot (\rho_F u_F) |u_F - \bar u_F|^2\,\dx + \int_{\Omega} \rho_F (u_F - \bar u_F)\cdot ((\nabla \bar u_F) \bar u_F - (\nabla u_F)u_F ) \,\dx \\
        &\qquad + \int_{\Omega} \rho_F (u_F - \bar u_F)\cdot \nabla W_F * (\bar \rho_F - \rho_F) \,\dx \\
        &\qquad +\int_{\Omega} \rho_F (u_F - \bar u_F)\cdot \Bigg[ \iint_{\Omega \times \Omega} \nabla W_C(x-y)\bar \rho_L(t,y,\xi)g(\textnormal{d}\xi)\dy - \nabla W_C * \bar\varrho_L^N) \Bigg] \dx \\
        &\qquad + \int_{\Omega} \rho_F (u_F - \bar u_F)\cdot \Bigg\{\iint_{\Omega\times \R^d}\phi(x-y)(w-u_F(x))\bar\mu_L^N(\dy\dw)  \\
        &\hspace{4.5cm} - \iint_{\Omega \times \Omega} \phi(x-y)(\bar u_L(t,y,\xi) - \bar u_F(t,x)) \bar \rho_L(t,y,\xi)g(\dxi)\dy  \Bigg\} \dx \\
        &\quad =: \sum_{i=1}^5 K_i.
    \end{align*}
    $\bullet$ Estimate of $K_1$ and $K_2$: We telescope the difference and then integrate by parts in order to obtain
    \begin{align*}
        K_2 &= \int_{\Omega} \rho_F (u_F - \bar u_F)\cdot (\nabla \bar u_F)(\bar u_F - u_F)\,\dx + \int_{\Omega} \rho_F (u_F - \bar u_F)\cdot (\nabla \bar u_F - \nabla u_F) u_F \,\dx \\
        &= \int_{\Omega} \rho_F (u_F - \bar u_F)\cdot (\nabla \bar u_F)(\bar u_F - u_F)\,\dx + \frac{1}{2}\int_{\Omega} \rho_F (u_F - \bar u_F)\otimes (u_F - \bar u_F): \nabla u_F \,\dx \\
        &\quad + \frac{1}{2}\int_{\Omega} (\nabla \rho_F \cdot u_F) |u_F - \bar u_F|^2\,\dx.
    \end{align*}
    This computation, along with the chain rule applied to $K_1$, shows
    \begin{align*}
        K_1 + K_2 &= -\frac{1}{2}\int_{\Omega} \rho_F \nabla\cdot u_F |u_F - \bar u_F|^2\,\dx - \int_{\Omega} \rho_F (u_F - \bar u_F)\cdot \nabla \bar u_F (u_F - \bar u_F) \,\dx \\
        &\quad + \frac{1}{2}\int_{\Omega} \rho_F(u_F-\bar u_F) \otimes (u_F - \bar u_F) :\nabla u_F\,\dx \\
        &\le C\|\nabla \bar u_F\|_{L^\infty} \int_{\Omega} \rho_F |u_F - \bar u_F|^2\,\dx.
    \end{align*}
    $\bullet$ Estimate of $K_3$: By the Cauchy--Schwarz inequality and utilizing that $\nabla W_F\in \calW^{1,\infty}(\Omega)$, we obtain
    \begin{align*}
        K_3 &\le \left(\int_{\Omega} \rho_F |u_F - \bar u_F|^2\,\dx \right)^{1/2} \left(\int_{\Omega} \rho_F |\nabla W_F * (\bar \rho_F - \rho_F)|^2\,\dx \right)^{1/2} \\
        &\le \int_{\Omega} \rho_F|u_F - \bar u_F|^2\,\dx + \|\nabla W_F\|_{\calW^{1,\infty}} d_{1}^2(\bar \rho_F(t), \rho_F(t)) \|\rho_F\|_{L^1}.
    \end{align*}
    $\bullet$ Estimate of $K_4$: Notice that $\nabla W_C * \bar\varrho_L^N$ is independent of $\xi$, and may therefore be rewritten as
    \begin{align*}
        \nabla W_C *  \bar\varrho_L^N &= \int_{\Omega} \nabla W_C(x-y) \frac{1}{N}\sum_{i=1}^N \delta_{\bar x_i(t)}(\dy) = \iint_{\Omega\times\Omega} \nabla W_C(x-y) \frac{1}{N}\sum_{i=1}^N \delta_{\bar x_i(t)}(\dy) \otimes \delta_{x_{d_i}}(\dxi).
    \end{align*}
    This shows
    \begin{align*}
        K_4 &= \int_{\Omega} \rho_F (u_F - \bar u_F)\cdot  \iint_{\Omega \times \Omega} \nabla W_C(x-y) \left( \bar\rho_L(y,\xi)\dy g(\dxi) - \frac{1}{N}\sum_{i=1}^N \delta_{\bar x_i(t)}(\dy) \otimes \delta_{x_{d_i}}(\dxi) \right)  \,\dx \\
        &\le \|\nabla W_C\|_{\calW^{1,\infty}}\left( \int_{\Omega} \rho_F |u_F - \bar u_F| \,\dx \right)  d_1 \left( \frac{1}{N}\sum_{i=1}^N \delta_{\bar x_i(t)}\otimes \delta_{x_{d_i}}   , \; \bar\rho_L(t)(\dy \otimes g(\dxi))  \right)  \\
        &\le \|\nabla W_C\|_{\calW^{1,\infty}} \|\rho_F\|_{L^1}^{1/2}\left(\int_{\Omega} \rho_F |u_F - \bar u_F|^2 \,\dx \right)^{1/2}  d_1 \left( \frac{1}{N}\sum_{i=1}^N \delta_{\bar x_i(t)}\otimes \delta_{x_{d_i}}   , \; \bar\rho_L(t)(\dy \otimes g(\dxi))  \right)  \\
        &\le C \int_{\Omega}\rho_F|u_F - \bar u_F|^2\,\dx + C  d_1^2 \left( \frac{1}{N}\sum_{i=1}^N \delta_{\bar x_i(t)}\otimes \delta_{x_{d_i}}  , \; \bar\rho_L(t)(\dy \otimes g(\dxi))  \right). 
    \end{align*}
    $\bullet$ Estimate of $K_5$: We observe first that
    \begin{align*}
      &\iint_{\Omega\times\R^d} \phi(x-y) (w- u_F(x)) \bar\mu_L^N(\dy\dw) \\
      &\quad = \iiint_{\Omega\times\R^d\times\Omega} \phi(x-y) w \, \bar\varsigma_L^N(\dy\dw\dxi) \\
      &\hspace{1.5cm} - u_F(x) \iint_{\Omega\times\Omega} \phi(x-y) \left(\frac1N \sum_{i=1}^N \delta_{\bar x_i(t)}(\dy) \otimes \delta_{x_{d_i}}(\dxi) \right).
    \end{align*} 
    Therefore, we rewrite the term ($=: \bar K_5(x)$) in the curly brackets $\{\}$ of $K_5$ as
    \begin{align*}
        \bar K_5(x)  &= \iint_{\Omega\times\Omega} \phi(x-y) \bar u_L(y,\xi) \left(\int_{\R^d} \bar\varsigma_L^N(\dy\dw\dxi) - \bar\rho_L(y,\xi) g(\dxi) \dy \right) \\
        &\quad + \iiint_{\Omega \times \R^d \times \Omega} \phi(x-y) (w - \bar u_L(y,\xi)) \bar\varsigma_L^N(\dy\dw\dxi) \\
        &\quad + (\bar u_F(x) - u_F(x)) \iint_{\Omega\times\Omega} \phi(x-y) \bar\rho_L(t,y,\xi) g(\dxi) \dy \\
        &\quad + u_F(x) \left(\iint_{\Omega\times\Omega} \phi(x-y) \left[ \bar\rho_L(t,y,\xi) g(\dxi) \dy - \frac1N \sum_{i=1}^N \delta_{\bar x_i(t)}(\dy) \otimes \delta_{x_{d_i}}(\dxi) \right] \right) \\
        &=: \sum_{i=1}^4 \bar{K}_{5,i}.
    \end{align*} 
    We manipulate and estimate $\bar K_{5,1}$ further as such
    \begin{equation}\label{eq: k51}
    \begin{split}
        |\bar K_{5,1}(x)| &= \left|\iint_{\Omega\times\Omega} \phi(x-y) \bar u_L(y,\xi) \left(\frac{1}{N}\sum_{i=1}^N \delta_{\bar x_i(t)}(\dy) \otimes \delta_{x_{d_i}}(\dxi) - \bar\rho_L(y,\xi) g(\dxi) \dy \right) \right| \\
        &\le \left| \iint_{\Omega \times \Omega} \phi(x-y) (\bar u_L(y,\xi) - \bar u_L(x,\xi))   \left(\frac{1}{N}\sum_{i=1}^N \delta_{\bar x_i(t)}(\dy) \otimes \delta_{x_{d_i}}(\dxi) - \bar\rho_L(y,\xi) g(\dxi) \dy \right) \right| \\
        &\quad + \left|\iint_{\Omega\times\Omega} \bar u_L(x,\xi) \phi(x-y) \left(\frac{1}{N}\sum_{i=1}^N \delta_{\bar x_i(t)}(\dy) \otimes \delta_{x_{d_i}}(\dxi) - \bar\rho_L(y,\xi) g(\dxi) \dy \right) \right| \\
        &\le \|\phi(x-\cdot)(\bar u_L(\cdot,\cdot)-\bar u_L(x,\cdot))\|_{ {\rm Lip}_{x,\xi} }  d_{1}\left(\frac{1}{N}\sum_{i=1}^N \delta_{\bar x_i(t)} \otimes \delta_{x_{d_i}}   , \; \bar\rho_L(t) (\dy \otimes g(\dxi))   \right) \\
        &\quad + \|\bar u_L(x,\cdot) \phi(x-\cdot) \|_{{\rm Lip}_{x,\xi} } d_1 \left(\frac{1}{N}\sum_{i=1}^N \delta_{\bar x_i(t)} \otimes \delta_{x_{d_i}}  , \; \bar\rho_L(t) (\dy \otimes g(\dxi))   \right).
    \end{split}
    \end{equation} 
    To handle the first term of the right-hand side let us set $f(y,\xi) = \phi(x-y)(\bar u_L(y,\xi)-\bar u_L(x,\xi))$ for a fixed $x \in \Omega$. Then, we get
    \begin{align*}
    |\nabla_y f(y,\xi)| &\le |\nabla \phi(x-y)| |\bar u_L(y,\xi)-\bar u_L(x,\xi)| + |\phi(x-y)| |\nabla \bar u_L(y,\xi)|\cr
    &\leq \|\nabla \bar u_L\|_{L^\infty_g}|\nabla \phi(x-y)| |x-y| + \|\phi\|_{L^\infty} \|\nabla \bar u_L\|_{L^\infty_g} \cr
    &\leq \|\nabla \bar u_L\|_{L^\infty_g} \left(\| \phi \|_{\calW^{1,\infty}_1 } + \|\phi\|_{L^\infty} \right).
    \end{align*}     
    Also, we can easily check
    \begin{align*}
        |\nabla_\xi f(y,\xi)| &\le \|\phi\|_{L^\infty} (|\nabla_\xi \bar u_L(y,\xi)| + |\nabla_\xi \bar u_L(x,\xi)|)  \le 2 \|\phi\|_{L^\infty} \|\nabla_\xi \bar u_L\|_{L^\infty_{x,\xi} }.
    \end{align*}

    To handle the second term of the right-hand side in \eqref{eq: k51}, we similarly set $\frak{f}(y,\xi) = \bar u_L(x,\xi) \phi(x-y)$ for a fixed $x\in \Omega$. Then
    \begin{align*}
        |\nabla_y \frak{f}(y,\xi)| \le |\bar u_L(x,\xi)| |\nabla \phi(x-y)| \le (1+|x|) \|\phi\|_{\calW^{1,\infty}},
    \end{align*}
    thanks to the estimate
    \begin{align*}
      |\bar u_L(t,x,\xi)| &\le |\bar u_L(t,0,\xi)| + \|\nabla \bar u_L\|_{L^\infty} |x| \\
      &\le \sup_{t\in [0,T]} \sup_{\xi\in\textnormal{supp}(g)}\|\bar u_L(t,\cdot,\xi)\|_{L^\infty(B_1)} + \|\nabla_x \bar u_L\|_{L^\infty_{x,\xi}} |x|.
    \end{align*} 
    Furthermore,
    \begin{equation*}
        |\nabla_\xi \frak{f}(y,\xi)| \le \|\phi\|_{L^\infty} \|\nabla_\xi \bar u_L\|_{L^\infty_{x,\xi}}.
    \end{equation*}
 Altogether, we deduce
    \begin{align*}
        &|\bar K_{5,1}(x)|  \le C (1 + |x|) d_1 \left(\frac{1}{N}\sum_{i=1}^N \delta_{\bar x_i(t)} \otimes \delta_{x_{d_i}}   , \; \bar\rho_L(t) (\dy \otimes g(\dxi))  \right). 
    \end{align*}
    
    The remaining terms of $\bar K_5$ are easily estimated as
    \begin{align*}
        &|\bar K_{5,2}| \le C\|\phi\|_{L^\infty} \sqrt{\calF(\bar \calZ_L^N | \bar Z_L)(t)},\\
        & |\bar K_{5,3}| \le \|\phi\|_{L^\infty} |\bar u_F(x)- u_F(x)|, \\
        &|\bar K_{5,4}|  \le \|\bar u_F\|_{L^\infty} \|\phi\|_{\calW^{1,\infty}}  d_1 \left(\frac{1}{N}\sum_{i=1}^N \delta_{\bar x_i(t)} \otimes \delta_{x_{d_i}}   , \; \bar\rho_L(t) (\dy \otimes g(\dxi)) \right). 
    \end{align*}
   Collecting the estimates, we deduce
    \begin{align*}
      |K_5| 
        &\le  C d_1 \left(\frac{1}{N}\sum_{i=1}^N \delta_{\bar x_i(t)} \otimes \delta_{x_{d_i}}   , \; \bar\rho_L(t) (\dy \otimes g(\dxi))  \right) \int_{\Omega} \rho_F(x) |u_F(x) - \bar u_F(x)| (1+|x|) \,\dx  \\
        &\quad + C \calF(\bar \calZ_L^N | \bar Z_L) + C \int_{\Omega} \rho_F|u_F-\bar u_F|^2\,\dx  + C d_1^2 \left(\frac{1}{N}\sum_{i=1}^N \delta_{\bar x_i(t)} \otimes \delta_{x_{d_i}}  , \; \bar\rho_L(t) (\dy \otimes g(\dxi))  \right) .  
    \end{align*} 
 The assertion of the lemma now follows as soon as we are able to appropriately estimate the first line of the right-hand side. We observe that
    \begin{align*}
        &d_1 \left(\frac{1}{N}\sum_{i=1}^N \delta_{\bar x_i(t)} \otimes \delta_{x_{d_i}} , \; \bar\rho_L(t) (\dy \otimes g(\dxi))  \right) \int_{\Omega} \rho_F(x) |u_F(x) - \bar u_F(x)| (1+|x|)\,\dx \\
        &\quad \le  d_1^2 \left(\frac{1}{N}\sum_{i=1}^N \delta_{\bar x_i(t)} \otimes \delta_{x_{d_i}}  , \; \bar\rho_L(t) (\dy \otimes g(\dxi)) \right) + \left(\int_{\Omega} \rho_F (1+|x|^2)  \,\dx \right) \left(\int_{\Omega} \rho_F |u_F - \bar u_F|^2 \,\dx \right) \\
        &\quad \le d_1^2 \left(\frac{1}{N}\sum_{i=1}^N \delta_{\bar x_i(t)} \otimes \delta_{x_{d_i}}   , \; \bar\rho_L(t) (\dy \otimes g(\dxi))  \right) + C \int_{\Omega} \rho_F |u_F - \bar u_F|^2 \,\dx ,
    \end{align*} 
    and therefore the proof is completed.
\end{proof}

%
%
%
%
%
%

\subsection{Proof of Theorem \ref{thm:mima-mama}: Convergence toward the macro-macro limit system} 

We now complete the convergence analysis for the macro-macro limit by gathering all previous estimates obtained for the leader and follower components. These collectively yield a stability bound for the distance between the empirical measures of the interacting particle system and the limiting macroscopic densities.

In particular, combining the estimates derived in the previous two subsections, we obtain the following inequality:
    \begin{align*}
        &\calF(\bar \calZ_L^N|\bar Z_L)(t) + d_1^2 \left(\frac{1}{N}\sum_{i=1}^N \delta_{\bar x_i(t)} \otimes \delta_{x_{d_i}}   , \; \bar\rho_L(t) (\dy \otimes g(\dxi))  \right)  + d_{1}^2(\bar \rho_F(t), \rho_F(t)) + \int_{\Omega} \rho_F |u_F - \bar u_F|^2\,\dx  \\
        &\quad \le \calF(\bar \calZ_L^N|\bar Z_L)(0) + C  d_1^2 \left(\frac{1}{N}\sum_{i=1}^N \delta_{\bar x_i(0)} \otimes \delta_{x_{d_i}}   , \; \bar\rho_L(0) (\dy \otimes g(\dxi))  \right) + C d_{1}^2(\bar \rho_F(0), \rho_F(0) )\\
        &\qquad + C \int_0^t \calF(\bar \calZ_L^N| \bar Z_L)(s) \,\ds + C \int_0^t d_{1}^2(\bar \rho_F(s),\rho_F(s)) \,\ds \\
        &\qquad + C \int_0^t  d_1^2 \left(\frac{1}{N}\sum_{i=1}^N \delta_{\bar x_i(s)} \otimes \delta_{x_{d_i}}  , \; \bar\rho_L(s)(\dy \otimes g(\dxi))  \right)  \ds + C\int_0^t \int_{\Omega} \rho_F|u_F- \bar u_F|^2\,\dx\ds.
    \end{align*}
The result then follows by applying Gr\"onwall's lemma to absorb the time-integrated terms.

Together, these bounds establish control over the discrepancy between the follower distributions and velocities in the macro-macro system. By Gr\"onwall-type arguments, one can propagate small initial discrepancies forward in time, provided that the leader dynamics are sufficiently regular and the coupling terms remain bounded. This provides a rigorous foundation for the validity of the mean-field closure in the follower sector and justifies the macro-macro approximation.

The structure of the convergences asserted in Theorem \ref{thm:mima-mama} closely parallels that of Theorem \ref{thm:mimi-mima}. We first observe that, clearly,
\begin{equation*}
    d_1^2 \left(\frac{1}{N}\sum_{i=1}^N \delta_{\bar x_i(t)} \otimes \delta_{x_{d_i}}  , \; \bar\rho_L(t) (\dy \otimes g(\dxi))  \right) \xrightarrow[N\to\infty]{} 0.
\end{equation*}
It is then easy to observe that this implies
\[
    d_1^2 \left( \frac{1}{N}\sum_{i=1}^N \delta_{\bar x_i(t)}(\dy) \; , \; \left[\int_{\Omega} \bar\rho_L(y,\xi) g(\dxi)\right]  \dy \right) \xrightarrow[N\to\infty]{} 0.
\]
Indeed, if we take a function $\varphi = \varphi(x) \in C^1(\Omega)$ with $\|\nabla_x \varphi\|_{L^\infty}\le 1$, then
\begin{align*}
    &\int_\Omega \varphi(x) \left(\frac{1}{N}\sum_{i=1}^N \delta_{\bar x_i(t)}(\dx) - \left[\int_{\Omega} \bar\rho_L(x,\xi)g(\dxi) \right] \dx \right) \\
    & \quad = \iint_{\Omega\times\Omega} \varphi(x) \left(\frac{1}{N}\sum_{i=1}^N \delta_{\bar x_i(t)}(\dx) \otimes \delta_{x_{d_i}}(\dxi) - \bar\rho_L(t)(\dx \otimes g(\dxi))  \right) \\
    & \quad \le d_1 \left(\frac{1}{N}\sum_{i=1}^N \delta_{\bar x_i(t)} \otimes \delta_{x_{d_i}}   , \; \bar\rho_L(t) (\dx \otimes g(\dxi)) \right) \xrightarrow[N\to\infty]{} 0.
\end{align*}

Next, we observe that (again with $\varphi \in C^1(\Omega)$, $\|\nabla \varphi\|_{L^\infty} \le 1$)
\begin{align*}
    &\int_{\Omega} \varphi(x) \left(\frac{1}{N}\sum_{i=1}^N \bar v_i(t) \delta_{\bar x_i(t)}(\dx) - \left[\int_{\Omega} \bar\rho_L(t,x,\xi) \bar u_L(t,x,\xi) g(\dxi) \right] \,\dx \right)\\
    & \quad = \int_{\Omega} \varphi(x) \left( \frac{1}{N}\sum_{i=1}^N (\bar v_i(t) - \bar u_L(t,x,x_{d_i})) \delta_{\bar x_i(t)}(\dx) \right)  \\
    &\qquad + \int_{\Omega} \varphi(x) \left(\frac{1}{N}\sum_{i=1}^N \bar u_L(t,x,x_{d_i}) \delta_{\bar x_i(t)}(\dx) - \int_{\Omega} \bar\rho_L(t,x,\xi) \bar u_L(t,x,\xi) g(\dxi) \dx \right) \\
    &\quad =: L_1 + L_2.
\end{align*}
Here, we estimate
\begin{align*}
    |L_1| \le \|\varphi\|_{L^\infty} \sqrt{\calF(\bar\calZ_L^N|\bar\calZ_L)(t)}
\end{align*}
and
\begin{align*}
    L_2 &= \int_{\Omega}\varphi(x) \left(\frac{1}{N}\sum_{i=1}^N \bar u_L(t,x,\xi) \delta_{\bar x_i(t)}(\dx) \otimes \delta_{x_{d_i}}(\dxi) - \int_{\Omega} \bar\rho_L(t,x,\xi) \bar u_L(t,x,\xi) g(\dxi) \dx \right) \\
    &= \int_{\Omega} \varphi(x) \bar u_L(t,x,\xi) \left(\frac{1}{N}\sum_{i=1}^N \delta_{\bar x_i(t)} \otimes \delta_{x_{d_i}} -\bar\rho_L(t) (\dx \otimes g(\dxi))  \right) .
\end{align*}
By virtue of the fact that $\bar u_L$ is only expected to be locally bounded in $x$, we restrict ourselves to those test functions $\varphi$ that are compactly supported. Under this assumption, we have
\begin{align*}
    \|\varphi \bar u_L(t,\cdot,\xi)\|_{\rm Lip} \le \|\varphi\|_{L^\infty} \|\nabla_x \bar u_L \|_{L^\infty} + \|\nabla_x \varphi\|_{L^\infty} \|\bar u_L\|_{L^\infty(\textnormal{supp}(\varphi))}.
\end{align*}
Thanks to the $L^\infty$ bounds on $\nabla \bar u_L$ and the local $L^\infty$ bounds for $\bar u_L$, we deduce
\begin{equation*}
    L_2 \le C  d_1\left(\frac{1}{N}\sum_{i=1}^N \delta_{\bar x_i(t)} \otimes \delta_{x_{d_i}}  , \;  \bar\rho_L(t) (\dx \otimes g(\dxi))  \right) \xrightarrow[N\to\infty]{} 0,
\end{equation*}
where $C>0$ is dependent on $\|\nabla_x \bar u_L\|_{L^\infty}$ and $\|\bar u_L\|_{L^\infty(\textnormal{supp}(\varphi))}$. 

In particular, by the estimates for $L_1,L_2$, we established the slightly weaker analogue
\begin{align*}
    \frac1N \sum_{i=1}^N \bar v_i(t) \delta_{\bar x_i(t)} \weakto \int_{\Omega} \bar\rho_L(t,\cdot,\xi) \bar u_L(t,\cdot,\xi) g(\dxi) \quad \text{in}\quad L^\infty(0,T;(C_c^1(\Omega))^* ).
\end{align*}
By the same arguments, we obtain
\begin{equation*}
    \frac{1}{N}\sum_{i=1}^N \bar v_i \otimes \bar v_i \delta_{\bar x_i} \weakto \int_{\Omega} \bar\rho_L \bar u_L \otimes \bar u_L \,g(\dxi) \quad \text{in}\quad L^\infty(0,T;(C_c^1(\Omega))^*),
\end{equation*}
and the remaining convergence results follow by arguments analogous to those in Section \ref{sec:23}, together with the quantitative stability estimate from \cite[Corollary 2.3]{CCJ21}.  This completes the proof of Theorem \ref{thm:mima-mama}.

%
%

%
%
%
%
%
%
%
%
%
%

%
%
%
%
%

\section{Numerical experiments}\label{sec_numer}
In this section, we present a series of numerical experiments aimed at validating and comparing the micro-micro model~\eqref{eq:mimi}, the micro-macro model~\eqref{eq: micmac lim}, and the macro-macro model~\eqref{eq: macmac lim}, all considered in a one-dimensional spatial setting ($d = 1$).

To demonstrate the range of collective behaviors captured by the proposed interaction framework, we present three representative numerical tests.  
Each test is designed to probe a distinct aspect of the model hierarchy: the first validates theoretical convergence across scales; the second explores how varying leader–follower coupling shapes the group structure; the third investigates the emergence of finite-time blow-up and shock formation under structured leader–target influence.  
Together, these scenarios illustrate how spatial arrangements, interaction topologies, and control parameters affect the formation, splitting, and aggregation of cohesive clusters in leader–follower dynamics.

\subsection{Numerical schemes}
In the following, we describe in detail the numerical strategies adopted to approximate the solutions of the three models under consideration.

The microscopic system~\eqref{eq:mimi} consists of $2(N+M)$ ordinary differential equations, which are numerically integrated using the classical fourth-order explicit Runge–Kutta method. For the macroscopic system~\eqref{eq: macmac lim}, which takes the form of a system of balance laws, we employ a numerical method based on a splitting (or fractional step) approach. Within this framework, the hyperbolic part is discretized using a first-order explicit finite volume scheme in both space and time, with numerical fluxes computed via the Rusanov (local Lax--Friedrichs) method. The source terms are treated separately and integrated using a forward Euler step. In the micro-macro model~\eqref{eq: micmac lim}, the two approaches are coupled: at each time step, the microscopic and macroscopic components are updated simultaneously using the methods described above. To ensure stability of the explicit finite volume schemes, we employ an adaptive time-stepping strategy based on the Courant–Friedrichs–Lewy (CFL) condition. Specifically, the time step $\Delta t^n$ at time level $n$ is chosen according to $\max_x |\lambda|\cdot \Delta t^n \leq \mathrm{CFL} \cdot \Delta x$, where $\lambda$ denotes the eigenvalues of the Jacobian of the flux function, and $\mathrm{CFL} \in (0,1]$ is a user-defined parameter.

In the case of the macro-macro system~\eqref{eq: macmac lim}, the eigenvalues of the Jacobian matrix of the flux function correspond to the characteristic velocities of the system. More precisely, the eigenvalues are given by $\bar{u}_L$ and $\bar{u}_F$, representing the local average velocities of the two interacting populations. These quantities determine the maximum signal propagation speed and thus directly influence the choice of the time step through the CFL condition. For the micro-macro system~\eqref{eq: micmac lim}, the eigenvalues coincide with $\bar{u}_F$. Homogeneous Neumann boundary conditions are imposed on all variables, setting the spatial derivatives at the boundaries to zero.

To simulate the same dynamics across the different scales, we consider initial data that are consistent between the microscopic and macroscopic models. Specifically, once the initial density profiles for the populations of leaders and followers are defined—typically chosen as Gaussian functions and used as initial data for the macroscopic equations—we generate the corresponding particle positions for the microscopic equations using the procedure described below. 

Particle positions are deterministically sampled from the target Gaussian distribution using the inverse transform method, which relies on a uniform discretization of the cumulative distribution function. This approach ensures a smooth and symmetric sampling of the distribution, avoiding statistical noise and particle clustering that may arise from purely random sampling.

In the following tests, we adopt initial density profiles of Gaussian type, characterized by a prescribed mean $\mu$ and a fixed variance $\sigma^2 = 0.25$. The corresponding Gaussian function is defined as
\begin{equation*}
    \mathsf{G}(x;\mu,\sigma) = \frac{1}{\sigma\sqrt{2\pi}} \exp\left\{ -\frac{1}{2} \left( \frac{x - \mu}{\sigma} \right)^2 \right\},
\end{equation*} which we denote by $\mathsf{G}_\mu(x)$ for brevity throughout the paper.

\subsection{Test 1: Successive limit validation: Particle $\to$ Hybrid $\to$ Continuum}
This test numerically validates Theorems~\ref{thm:mimi-mima} and~\ref{thm:mima-mama}, which establish the convergence of solutions between the microscopic particle system~\eqref{eq:mimi} and the micro-macro system~\eqref{eq: micmac lim}, and between the micro-macro system and the fully continuum system~\eqref{eq: macmac lim}, respectively. In particular, we aim to study the quantitative estimates provided by these theorems and to numerically assess the convergence rates as the number of particles in the leader and follower populations varies.

The simulation is conducted in the spatial domain $\Omega=[-8,8]$ up to the final time $T = 15$. At $t = 0$, the leaders are initially arranged as a single cluster with density $\bar \rho_L(0,x,\xi) = \mathsf{G}_{-3}(x)\delta_0(\xi)$ and zero velocity, $\bar u_L(0,x,\xi) = 0$. They experience repulsive self-interactions governed by $W_L'(r) = -r$ and are attracted to a Dirac delta target at the origin with attraction strength $\alpha = 0.5$. The followers also start as a single cluster with $\bar \rho_F(0,x) = \mathsf{G}_5(x)$ and $\bar u_F(0,x) = 0$, interacting repulsively via $W_F'(r) = -r$ and attracted to the leaders through $W_C'(r) = 2r$. A velocity alignment term is included, defined by $\phi(r) = (1 + |r|^2)^{-\beta/2}$ with $\beta = 0.5$.

In the micro-micro model, we fix the number of leaders to $N=100$ and vary the number of followers as $M = 10^k$, with $k = 1, \dots, 4$. For the micro-macro model, we vary the number of leaders as $N = 10^k$, with $k = 1, \dots, 4$. The macroscopic follower equation in the micro-macro and the macro-macro systems is discretized using a spatial grid with step size $\Delta x = 0.005$.

Figure~\ref{fig:figAB} presents the numerical results for all three models. For the micro-micro model, we display the particle trajectories for $N=M=100$; for the micro-macro model, we show the leader trajectories for $N=100$ together with the follower density; and for the macro-macro model, we report the densities of both populations. As observed, the dynamics of both populations are consistent across the three levels of description.

To quantify this agreement, we compute the quantitative estimates from the convergence theorems. Specifically, we evaluate the estimate~\eqref{eq: thm micmic micmac quant} between the particle system and the micro-macro system by defining the functional $\mathcal{F}^M(t) := F_1 + F_2 + F_3$, with 
\[
\begin{alignedat}{2}
    F_1 &:= d_1^2\!\big(\pi^x_\#\mu_F^{N,M}(t),\, \rho_F^N(t)\big), &\quad
    F_2 &:= \iint_{\Omega \times \R^d} |w - u_F^N(t,y)|^2 \mu_F^{N,M}(t,\dy\,\dw), \\
    F_3 &:= \iiiint_{(\Omega \times \R^d)^2} (|x-\bar x|^2 + |v-\bar v|^2) \pi_L^{N,M}(t,\dx\,\dv\,{\rm d}\bar x\,{\rm d}\bar v).
\end{alignedat}
\]

For the estimate~\eqref{est_miMA} between the micro-macro and fully continuum models, we define the functional $\mathcal{G}^N(t) := G_1 + G_2 + G_3 + G_4$, with
\[
\begin{alignedat}{2}
    G_1 &:= \iiint_{\Omega \times \R^d} |v - \bar u_L(t,x,0)|^2 \bar\mu_{L}^N(t,\dx\,\dv), &\quad
    G_2 &:=  d_{1}^2\!\big( \bar\varrho^N_L(t),\, \bar\rho_L(t,\cdot,0)\big),  \\
    G_3 &:= \int_{\Omega} \rho^N_F(t) |u^N_F(t) - \bar u_F(t)|^2\,\dx, &\quad
    G_4 &:= d_1^2(\bar\rho^N_F(t),\, \bar\rho_F(t)).
\end{alignedat}
\]
We point out that the above functional is equivalent to the left-hand side of \eqref{eq:mima}, thanks to the fact that the only target point of the micro and macro models is the origin $\xi = 0$.

Figures~\ref{fig:fiCD3} and~\ref{fig:fiCD1} illustrate that the functionals $\mathcal{F}^M(t)$ and $\mathcal{G}^N(t)$ decrease as $M$ and $N$ increase, respectively, indicating convergence toward the corresponding limit models. In Figures~\ref{fig:figCD4} and~\ref{fig:figCD2}, we report the estimated convergence rates for each term with respect to the reference solutions computed at $M = 10^4$ and $N = 10^4$, respectively. The rates, obtained via log-log regression of $L^2$ norm errors, suggest an overall convergence of order $\mathcal{O}(1/M)$ and $\mathcal{O}(1/N)$, respectively.

In summary, the results of this convergence test provide clear numerical evidence that the multiscale hierarchy correctly recovers the expected limiting behavior as the number of particles increases. The estimated convergence rates align well with the theoretical predictions, confirming the validity of the quantitative estimates derived in the convergence theorems.  
This supports the consistency of the micro-micro, micro-macro, and macro-macro models and illustrates the effectiveness of the proposed numerical framework for studying the propagation of mean-field limits in leader–follower systems.

\begin{figure}[ht]
\centering
\begin{subfigure}[b]{0.32\textwidth} \includegraphics[width=\linewidth]{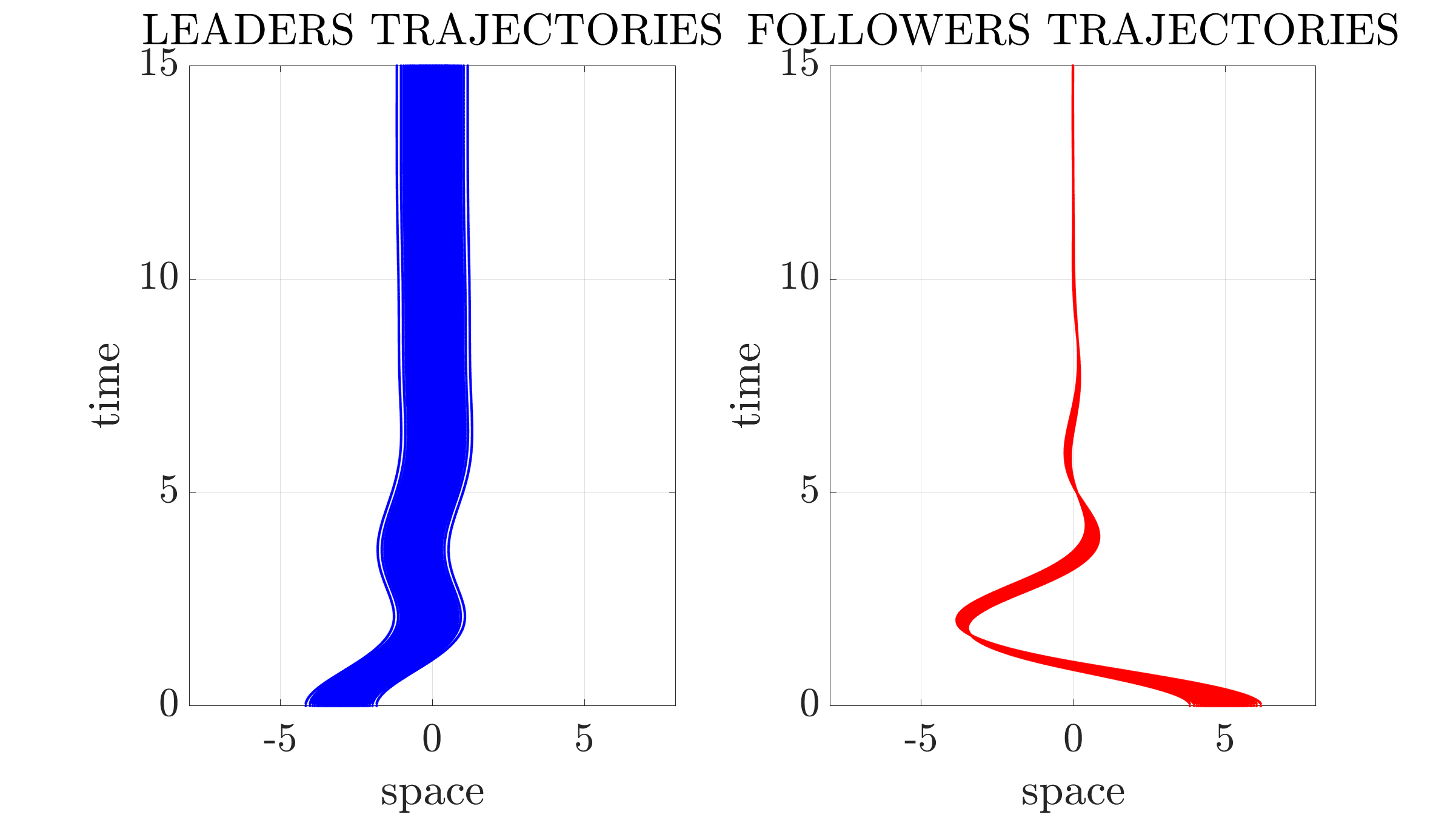}
  \caption{micro-micro with $N=M=100$}
\end{subfigure}
\hfill
\begin{subfigure}[b]{0.32\textwidth}
\includegraphics[width=\linewidth]{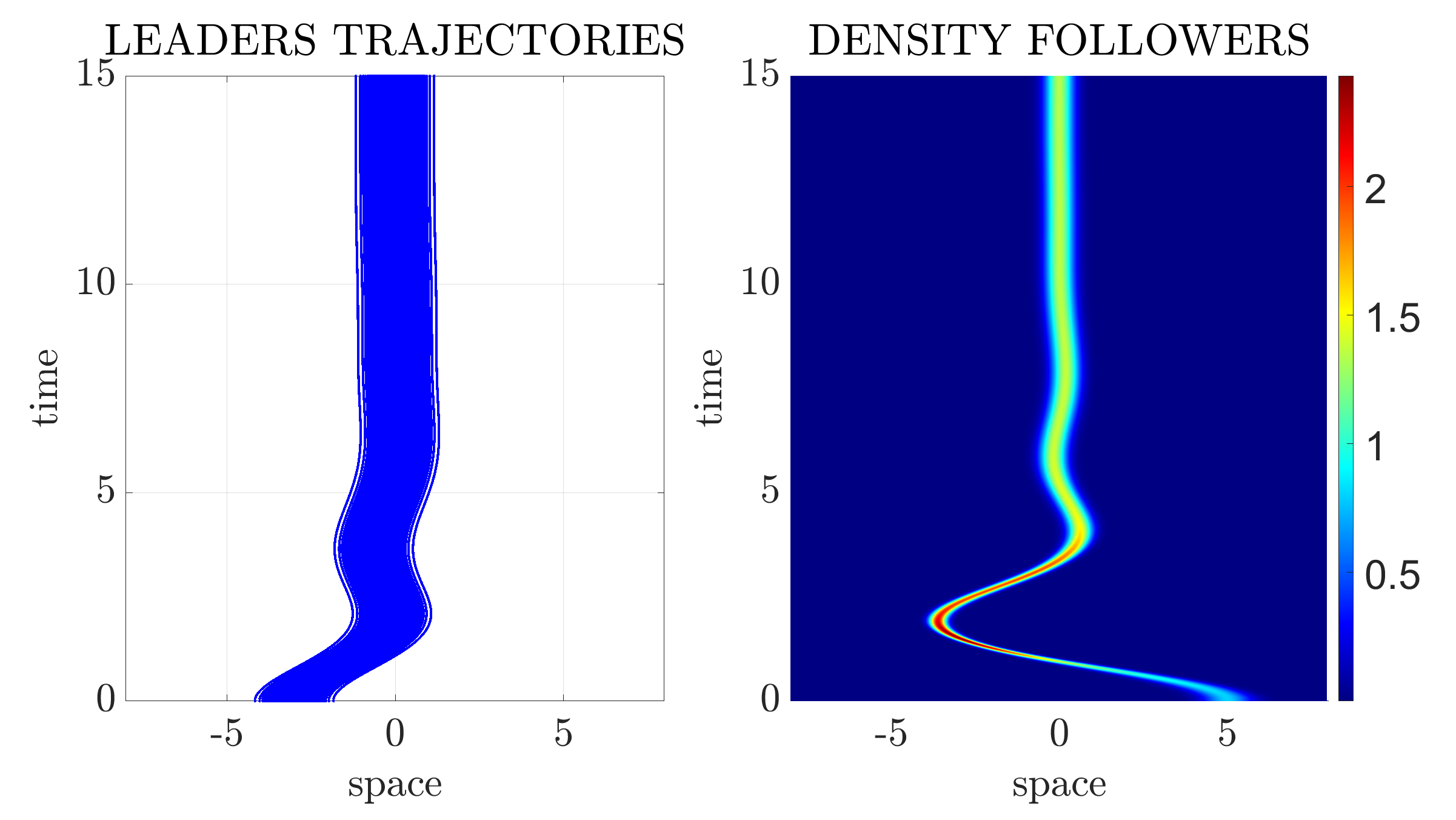}
  \caption{micro-macro with $N=100$}
\end{subfigure}
\hfill
\begin{subfigure}[b]{0.32\textwidth}\includegraphics[width=\linewidth]{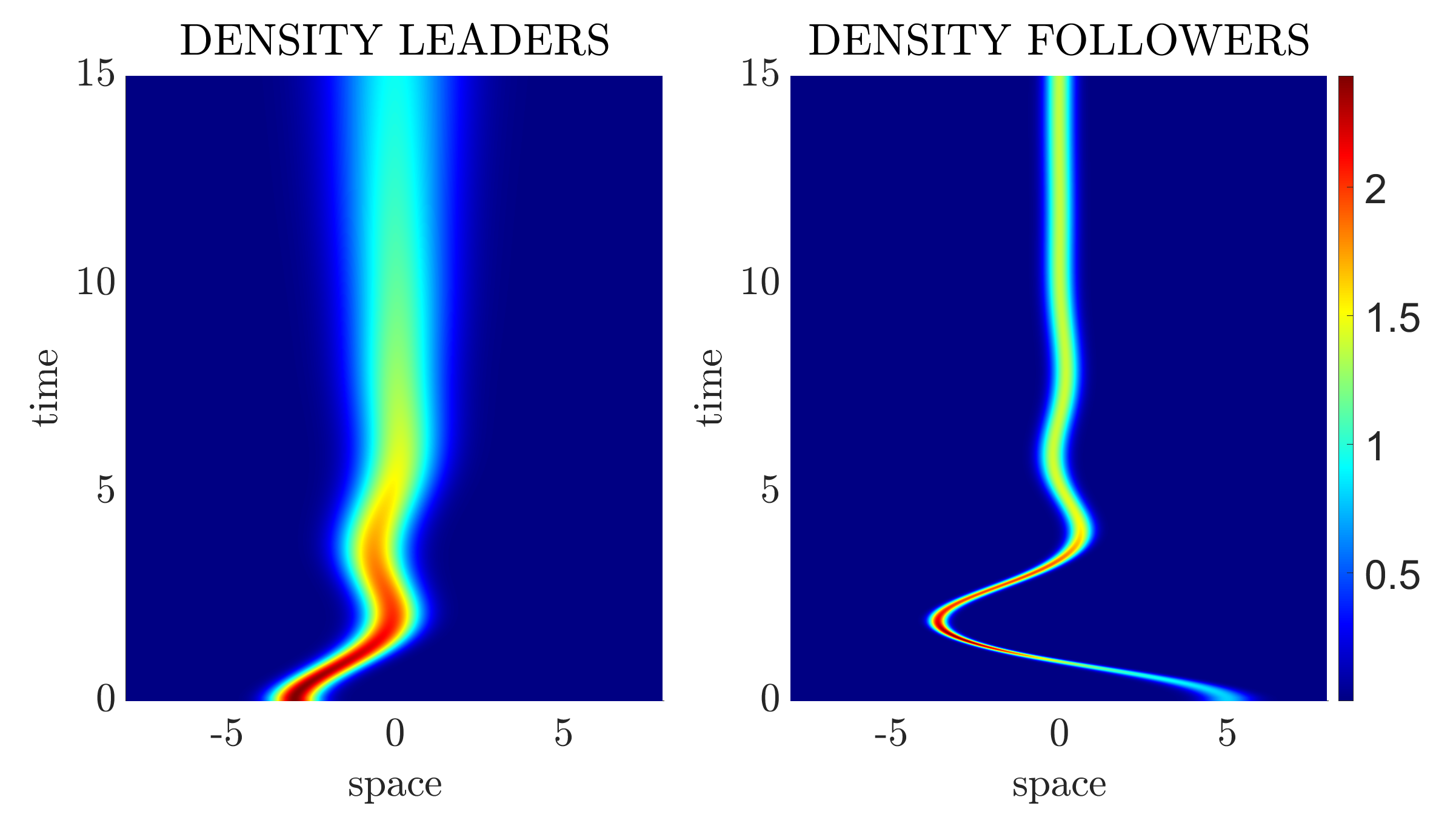}
  \caption{macro-macro}
\end{subfigure}
\caption{Test 1. Comparison of the solutions between the three models.}
\label{fig:figAB}
\end{figure}

\begin{figure}[ht]
\centering
\begin{subfigure}[b]{0.48\textwidth}
  \includegraphics[width=\linewidth]{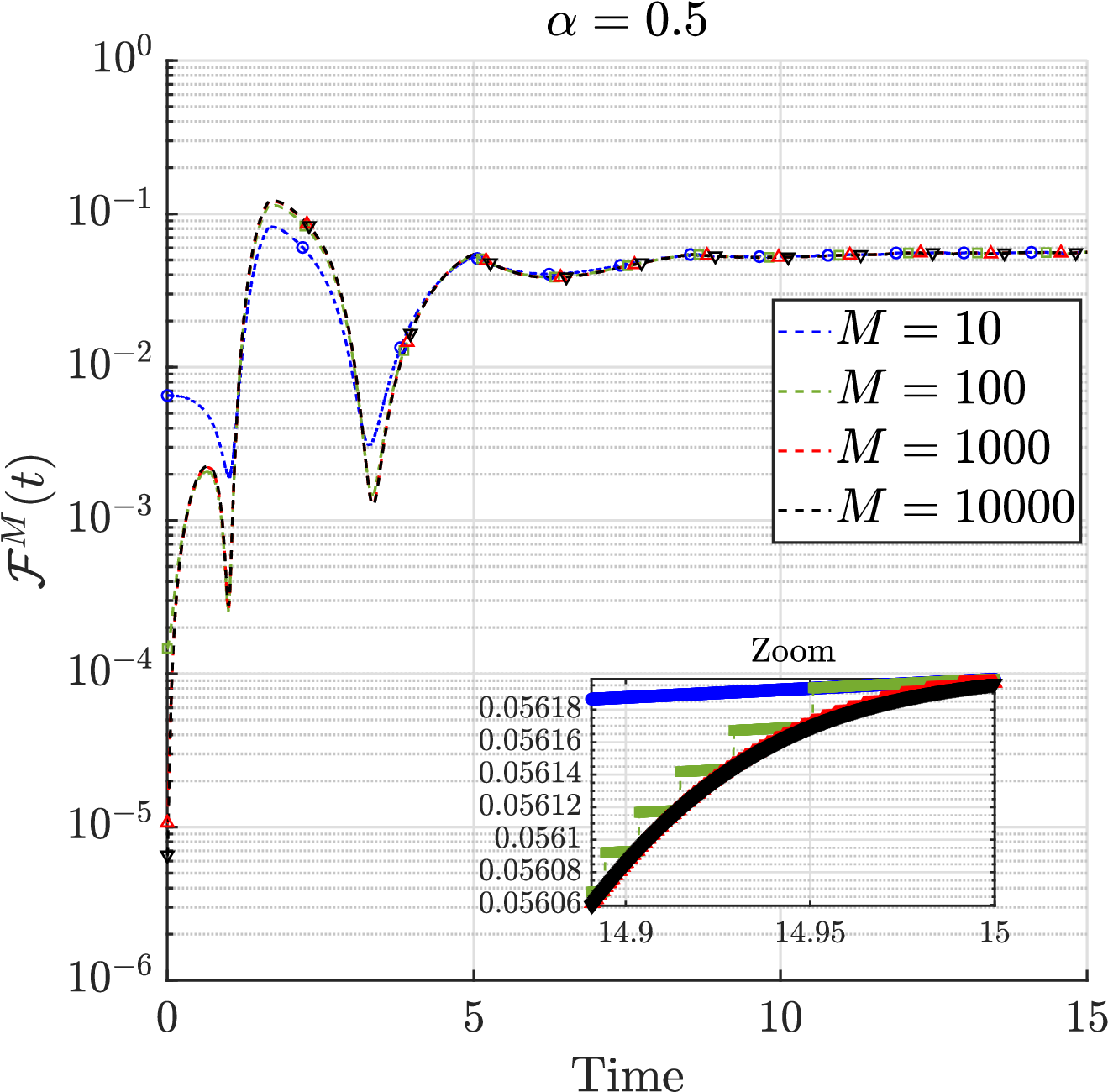}
  \caption{Time evolution of the functional $\mathcal{F}^M(t)$ for different values of $M$.}
  \label{fig:fiCD3}
\end{subfigure}
\hfill
\begin{subfigure}[b]{0.48\textwidth}
  \includegraphics[width=\linewidth]{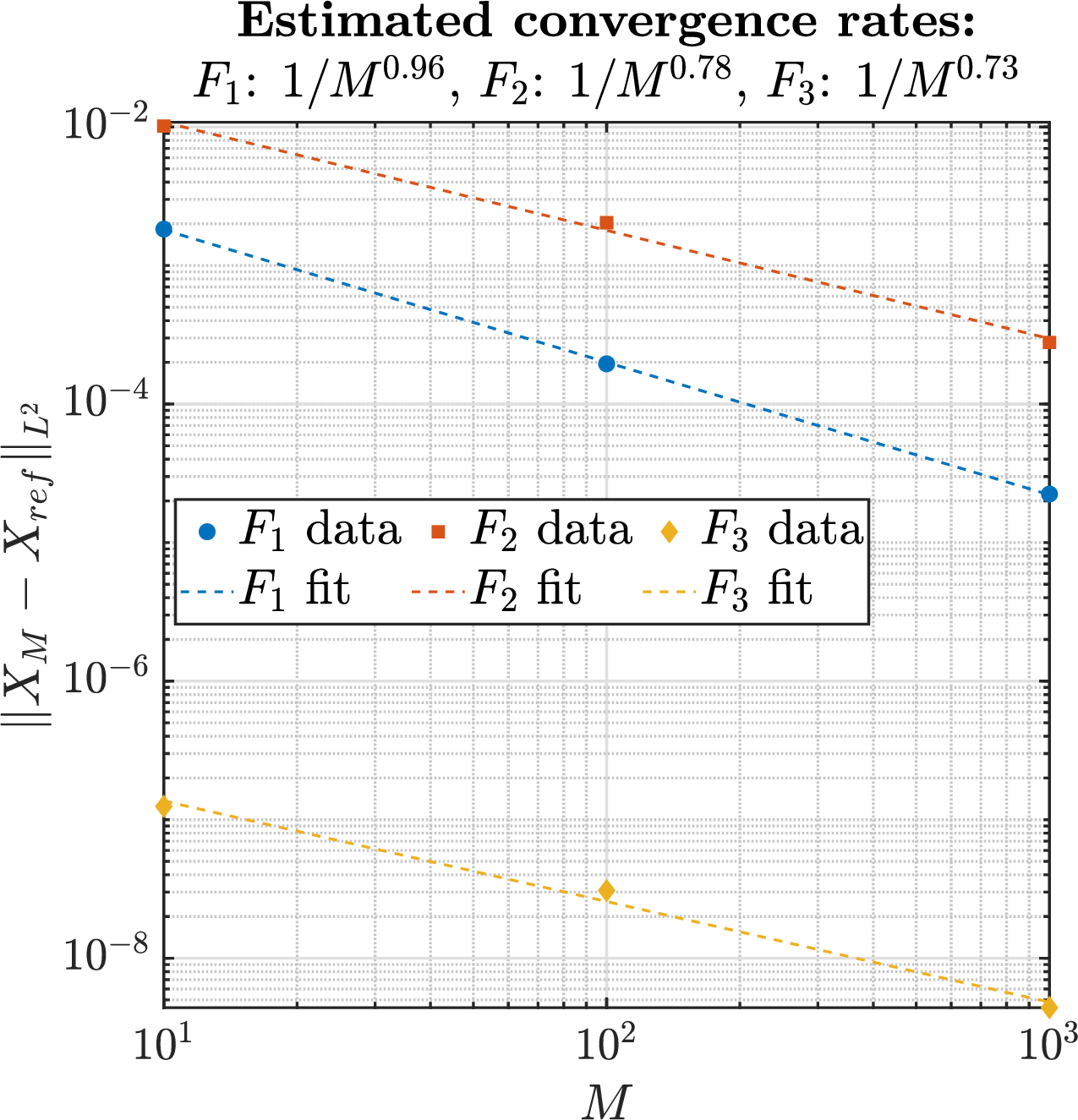}
  \caption{Estimated convergence rates for each component of $\mathcal{F}^M(t)$ as $M$ varies.}
  \label{fig:figCD4}
\end{subfigure}
\caption{Test 1. Particle $\to$ Hybrid convergence as $M$ increases.}
\end{figure}

\begin{figure}[ht]
\centering
\begin{subfigure}[b]{0.48\textwidth}
  \includegraphics[width=\linewidth]{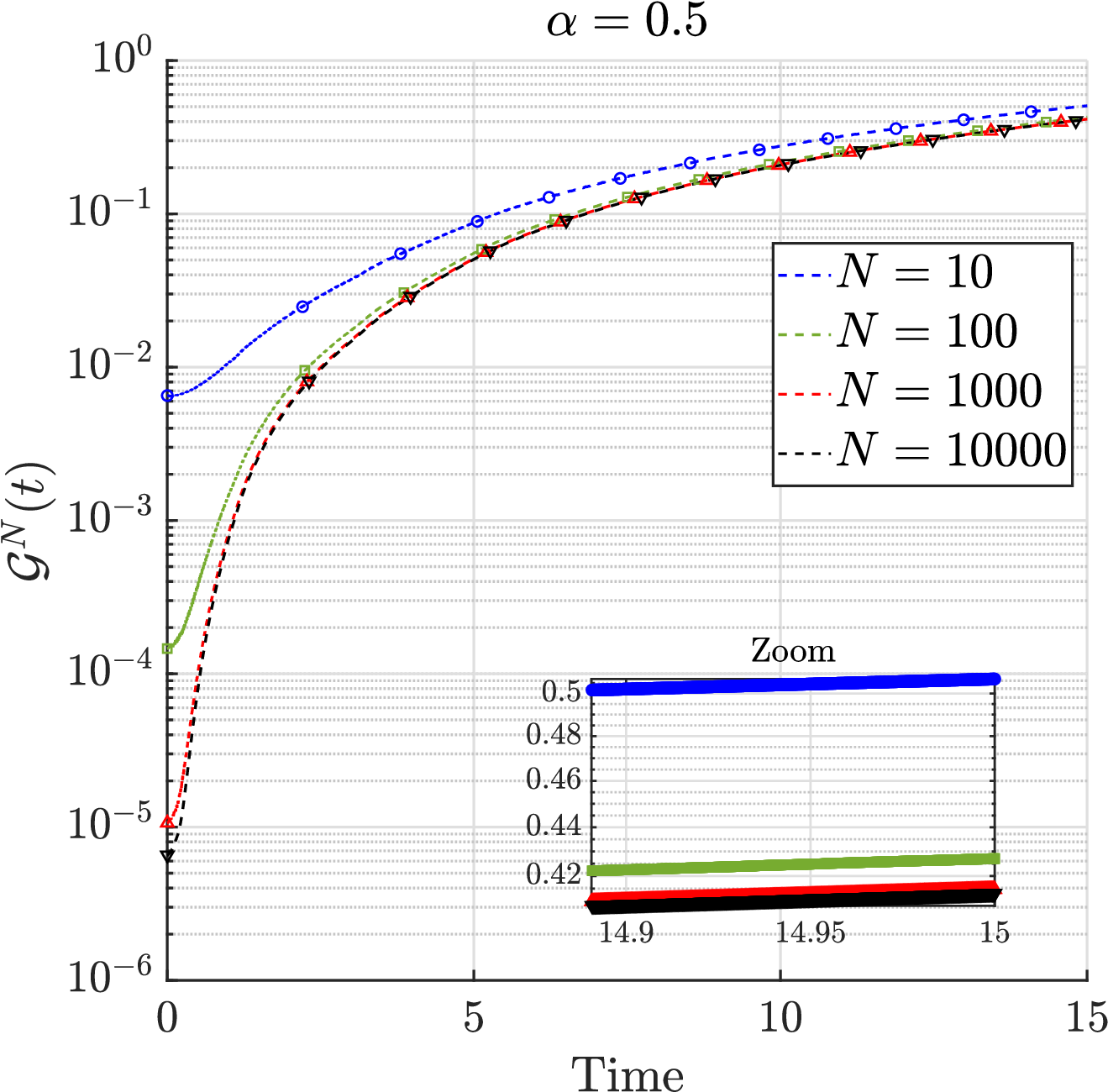}
  \caption{Time evolution of the functional $\mathcal{G}^N(t)$ for different values of $N$.}
  \label{fig:fiCD1}
\end{subfigure}
\hfill
\begin{subfigure}[b]{0.48\textwidth}
  \includegraphics[width=\linewidth]{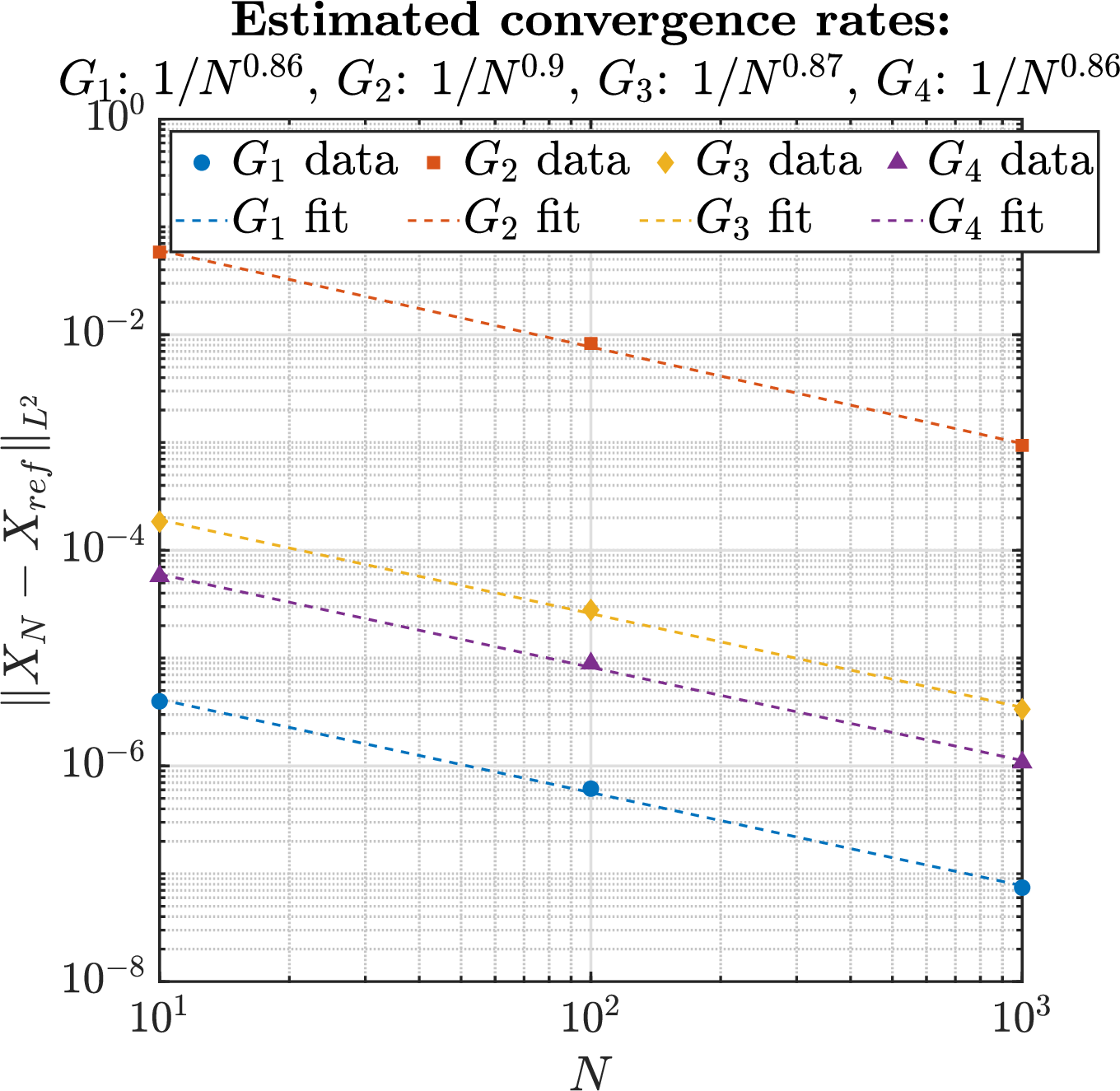}
  \caption{Estimated convergence rates for each component of $\mathcal{G}^N(t)$ as $N$ varies.}
  \label{fig:figCD2}
\end{subfigure}
\caption{Hybrid $\to$ Continuum convergence as $N$ increases.}
\end{figure}

\subsection{Test 2: Model comparison across control regimes} In this test, we investigate how the micro-micro, micro-macro, and macro-macro models reproduce the same qualitative dynamics under varying values of the control parameter $\alpha$. This comparison highlights the consistency of the model hierarchy and the role of the coupling strength in shaping the collective behavior.

The spatial domain is $[-15, 15]$, and the final simulation time is $T = 10$. Initially, the leader population is split into two clusters:
\[
\bar \rho_L(0,x,\xi) = 0.5 \left( \mathsf{G}_{-3}(x)\delta_{-5}(\xi) + \mathsf{G}_3(x)\delta_{5}(\xi) \right), \quad \bar u_L(0,x,\xi) = 0,
\]
with two Dirac delta targets positioned  at $\xi = \pm 5$. The repulsive interaction among leaders is given by
\[
W'_L(r) = - \frac{2\operatorname{sign}(r)}{|r| + \varepsilon^2}  \mathbf{1}_{\{ \varepsilon < |r| \leq 2.5 \}},
\]
where $\varepsilon \ll 1$ is a small regularization parameter to avoid singularities. The follower population starts as a single Gaussian:
\[
\bar \rho_F(0,x) = \mathsf{G}_0(x), \quad \bar u_F(0,x) = -\chi_{[-2,0]}(x) + \chi_{[0,2]}(x).
\]
Followers interact repulsively via $W'_F(r) = -\operatorname{sign}(r)$ and are attracted to leaders through:
\(
W'_C(r) = r  \mathbf{1}_{\{ |r| \leq 4 \}}.
\)
The velocity alignment mechanism is modeled by $\phi(r) = (1 + |r|^2)^{-\beta/2}$ with $\beta = 0.005$. For the microscopic simulations, we use $N = M = 150$ particles, while macroscopic simulations employ a spatial grid with step size $\Delta x = 0.005$. We examine three different values of the attraction strength $\alpha$: 0, 0.5, and 1, to illustrate how the collective behavior changes under weaker or stronger leader-follower coupling.

In Fig.~\ref{fig:test2_1}, we illustrate the trajectories in the microscopic case and the corresponding density profiles in the macroscopic cases for all three models, as the control parameter $\alpha$ varies. Overall, the dynamics of both leader and follower populations remain qualitatively consistent across the different scales of description.

For $\alpha = 0$, the two leader clusters each converge towards their respective target positions, while the follower population splits into two subgroups, each attracted to the nearest leaders. When $\alpha = 0.5$, the leaders persist in two separate groups with distinct targets but move closer together due to the attraction towards the followers’ center of mass. Consequently, the followers divide into three subgroups: two following the leaders and a central one remaining stationary, equally influenced by both leader clusters. Finally, for $\alpha = 1$, the leaders are predominantly driven by the global center of mass of the followers, which pulls them towards the origin. In turn, the followers, attracted to the leaders, also concentrate at the origin.

In Fig.~\ref{fig:test2_2}, we provide a visual comparison of the momentum, defined as the product of density and velocity. In the macroscopic models, momentum is directly evolved, whereas in the microscopic simulations it is reconstructed by partitioning the spatial domain into uniform bins: at each time step, the density in a bin corresponds to the number of particles it contains, and the momentum is computed as the density multiplied by the average velocity in that bin. This comparison reveals a strong qualitative agreement among the three model formulations.

In conclusion, this test demonstrates that the three models generate qualitatively consistent collective dynamics under different levels of leader–follower coupling.  
The observed transitions—from well-separated groups to full aggregation—highlight the interplay between target attraction and follower alignment, and confirm that the multiscale framework reliably captures these structural changes at all levels of description.  
These findings emphasize the importance of the coupling parameter in controlling group cohesion and suggest its potential as a design variable for applications in collective behavior and crowd management.

\begin{figure}[ht]
\centering
\begin{subfigure}[b]{0.3\textwidth}
\includegraphics[width=\linewidth]{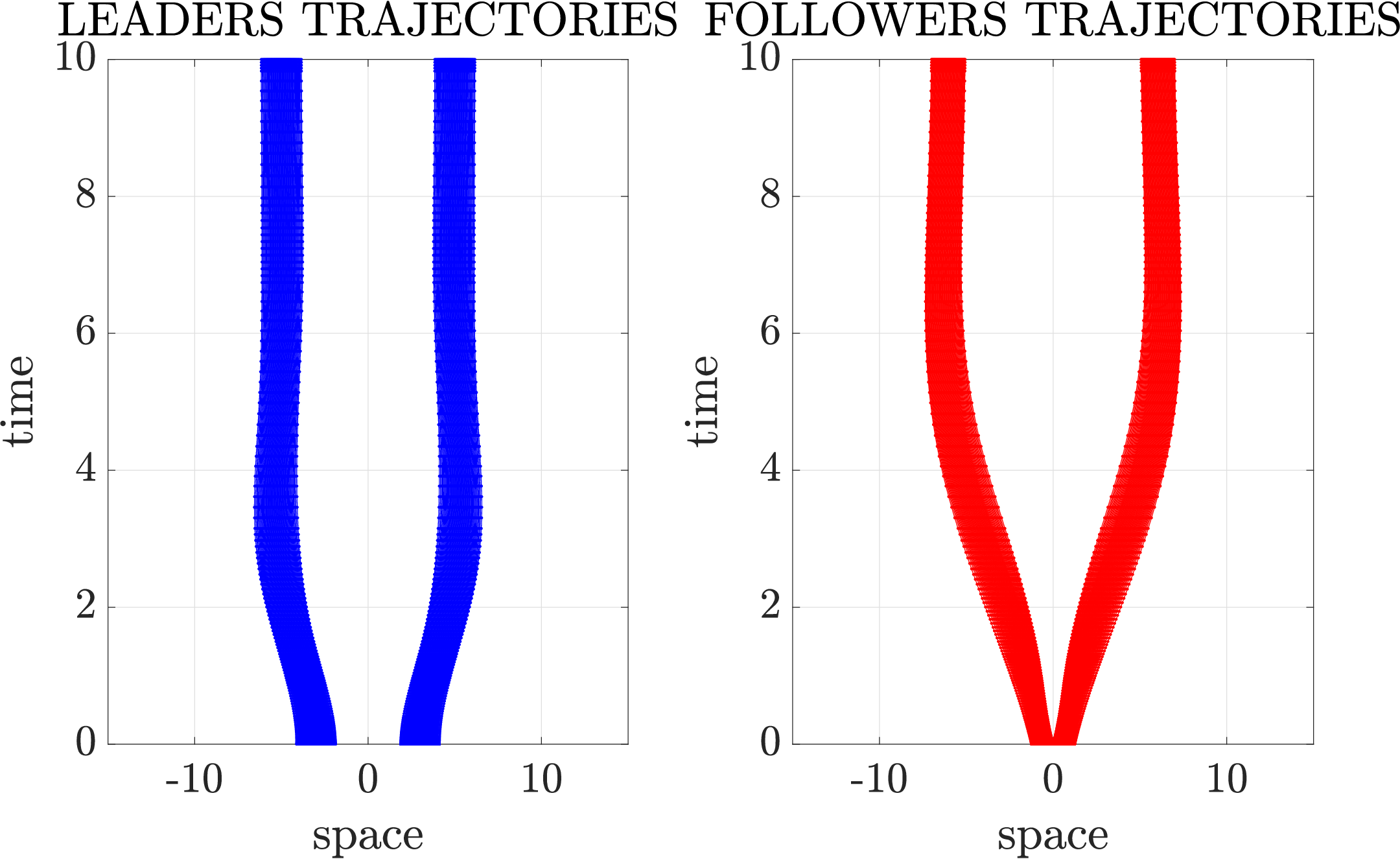}
\caption{micro-micro, $\alpha = 0$}
\end{subfigure}
\hfill
\begin{subfigure}[b]{0.32\textwidth}
\includegraphics[width=\linewidth]{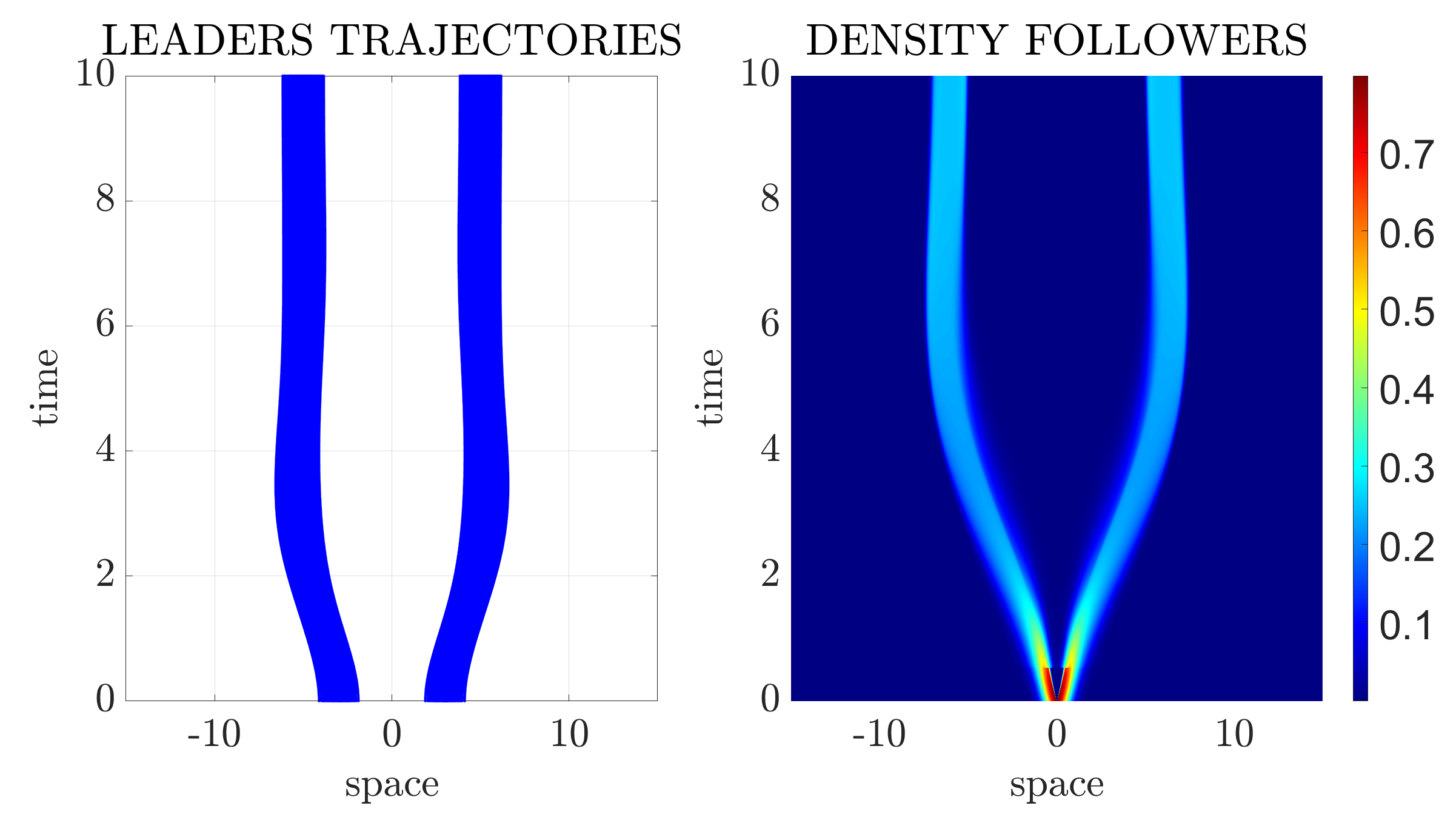}
\caption{micro-macro, $\alpha = 0$}
\end{subfigure}
\hfill
\begin{subfigure}[b]{0.32\textwidth}
\includegraphics[width=\linewidth]{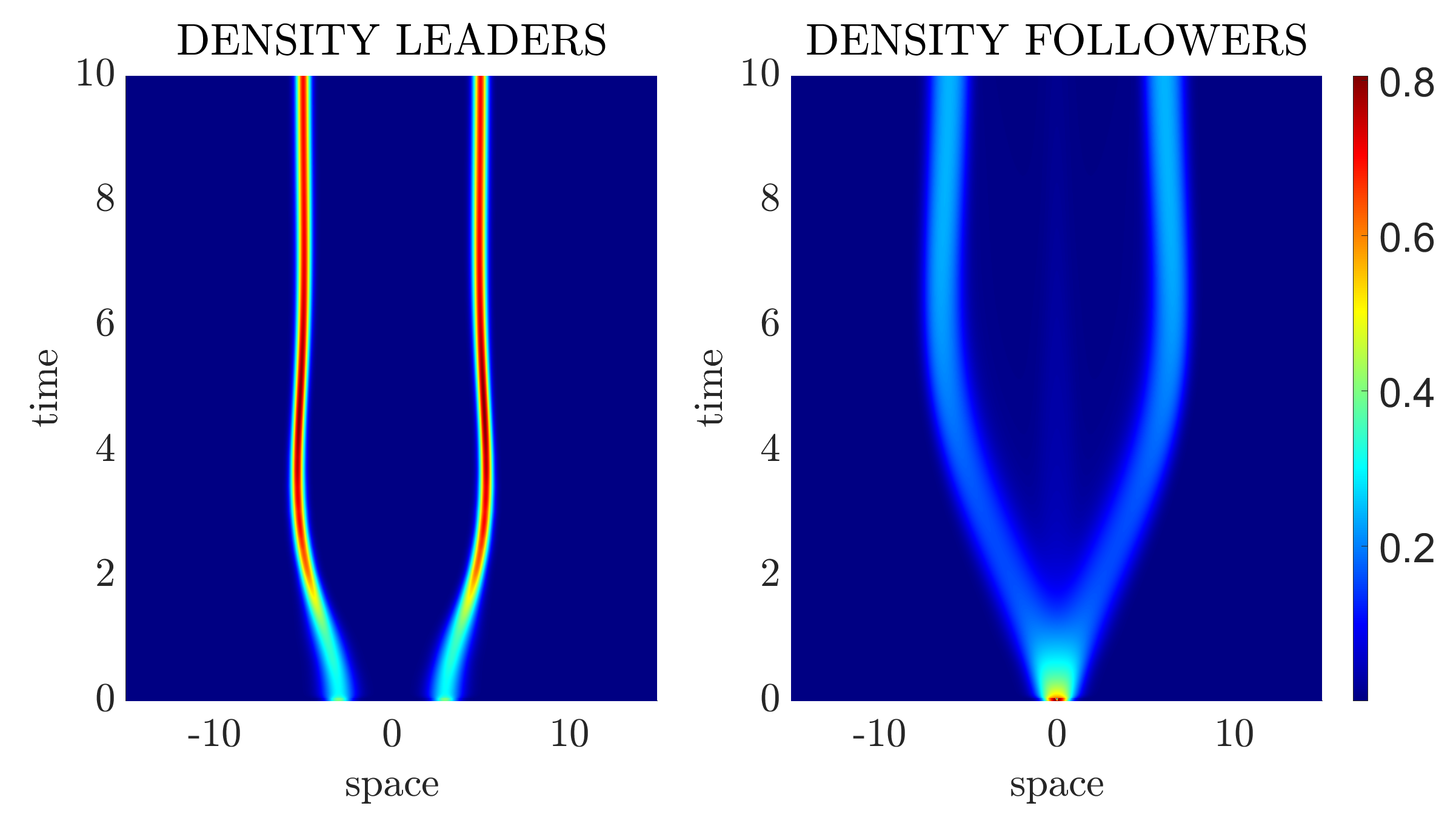}
\caption{macro-macro, $\alpha = 0$}
\end{subfigure}

\vspace{0.5em} 

\begin{subfigure}[b]{0.32\textwidth}
\includegraphics[width=\linewidth]{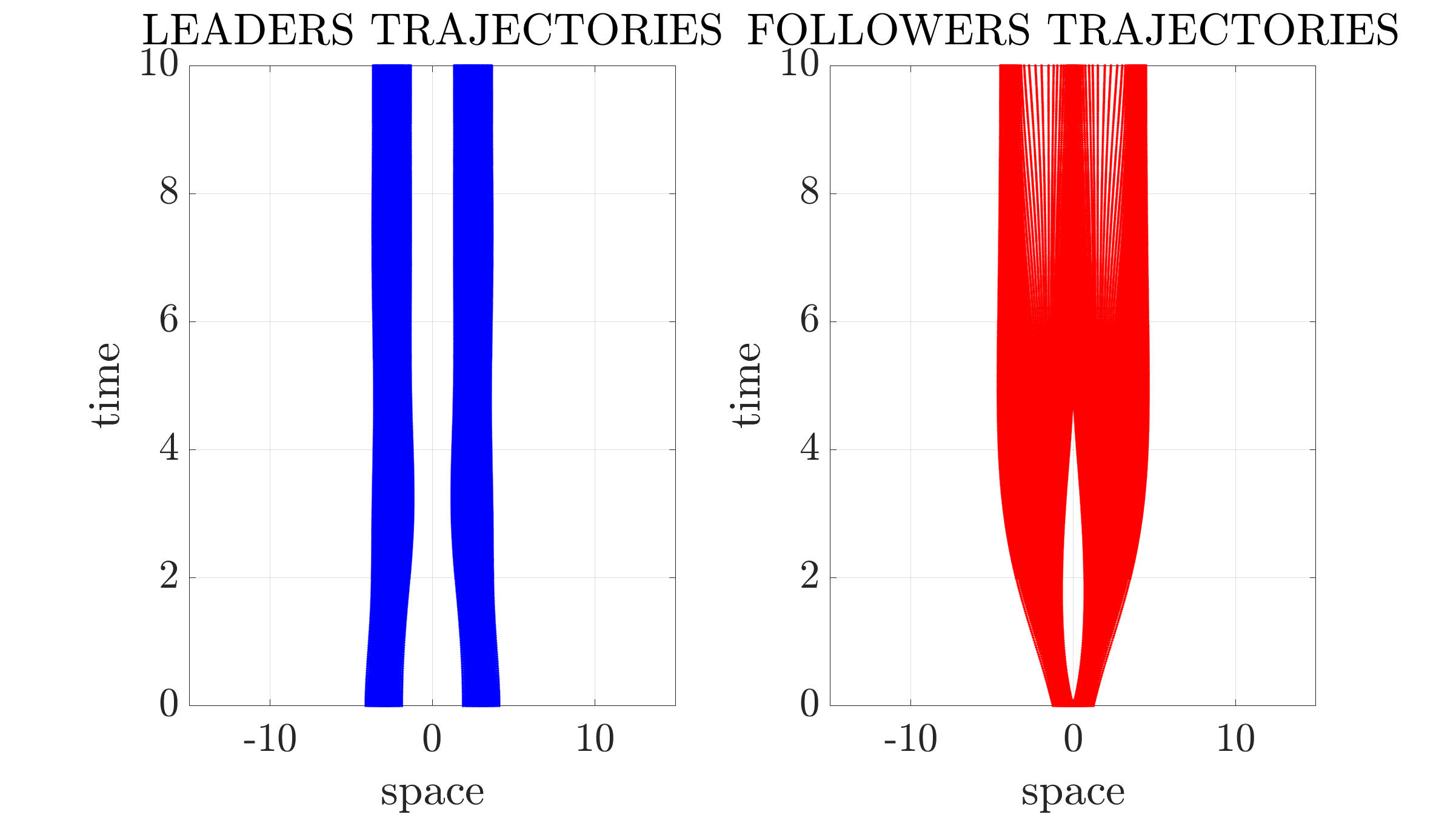}
\caption{micro-micro, $\alpha = 0.5$}
\end{subfigure}
\hfill
\begin{subfigure}[b]{0.32\textwidth}
\includegraphics[width=\linewidth]{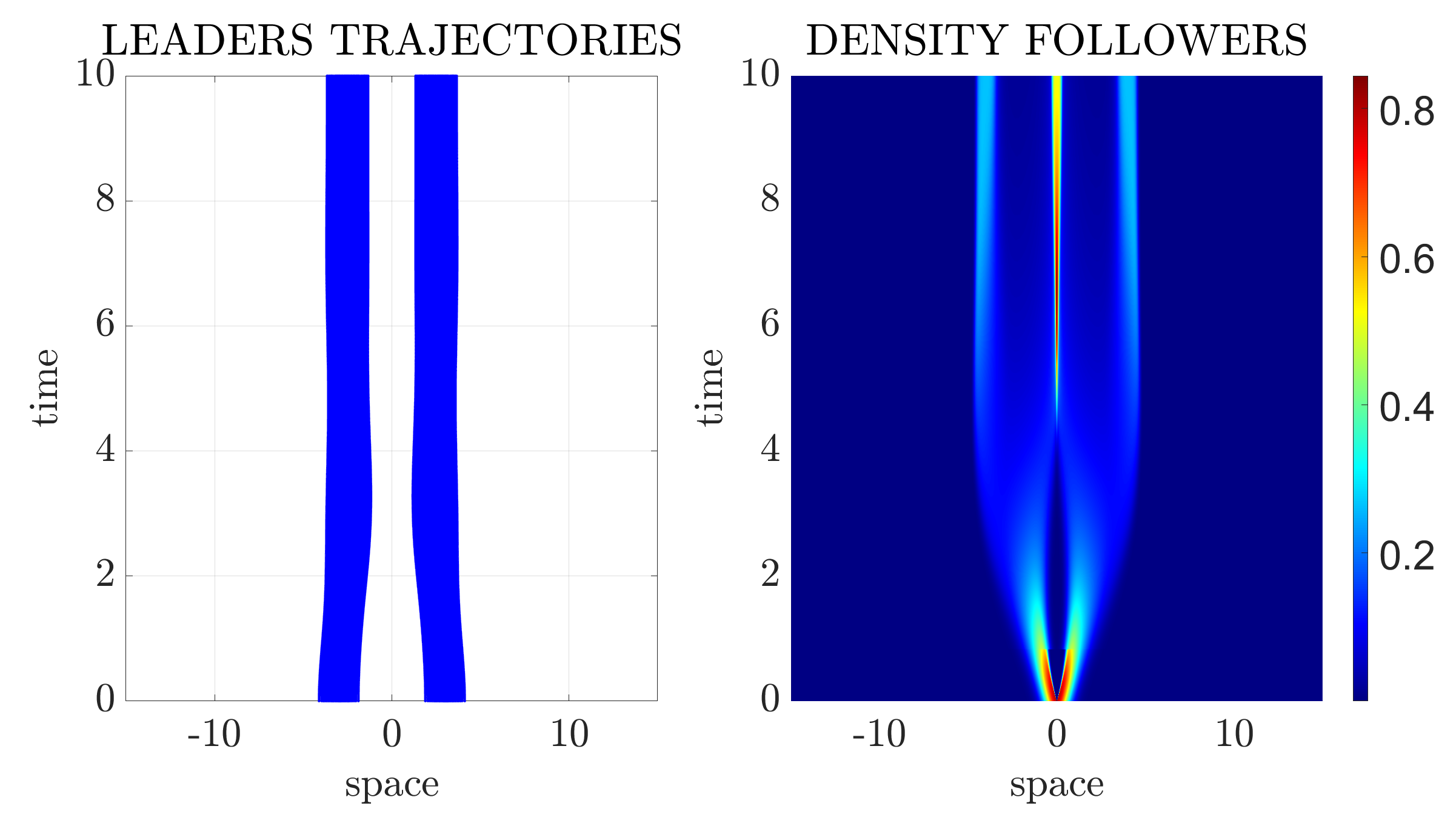}
\caption{micro-macro, $\alpha = 0.5$}
\end{subfigure}
\hfill
\begin{subfigure}[b]{0.32\textwidth}
\includegraphics[width=\linewidth]{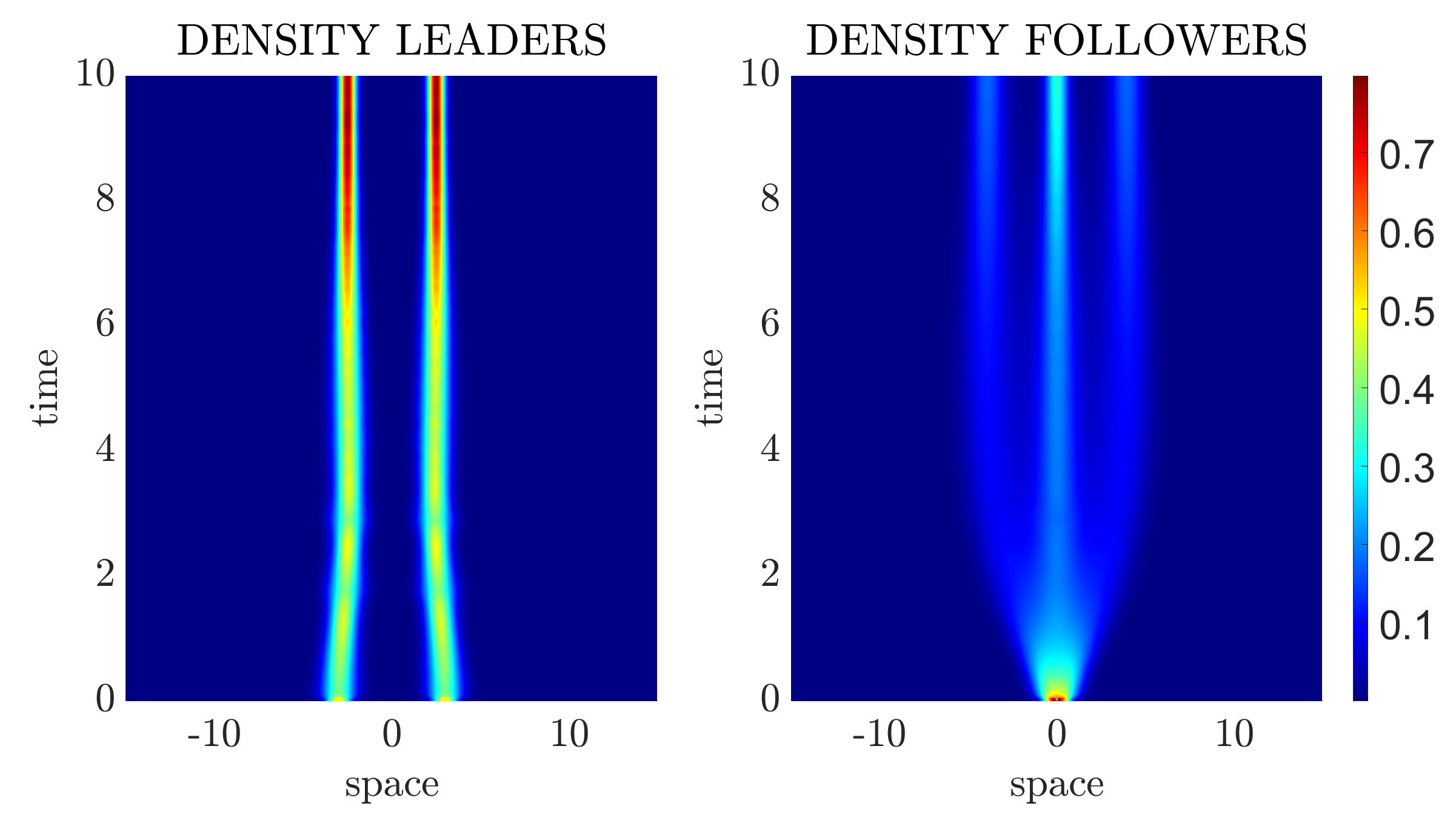}
\caption{macro-macro, $\alpha = 0.5$}
\end{subfigure}

\vspace{0.5em} 

\begin{subfigure}[b]{0.32\textwidth}
\includegraphics[width=\linewidth]{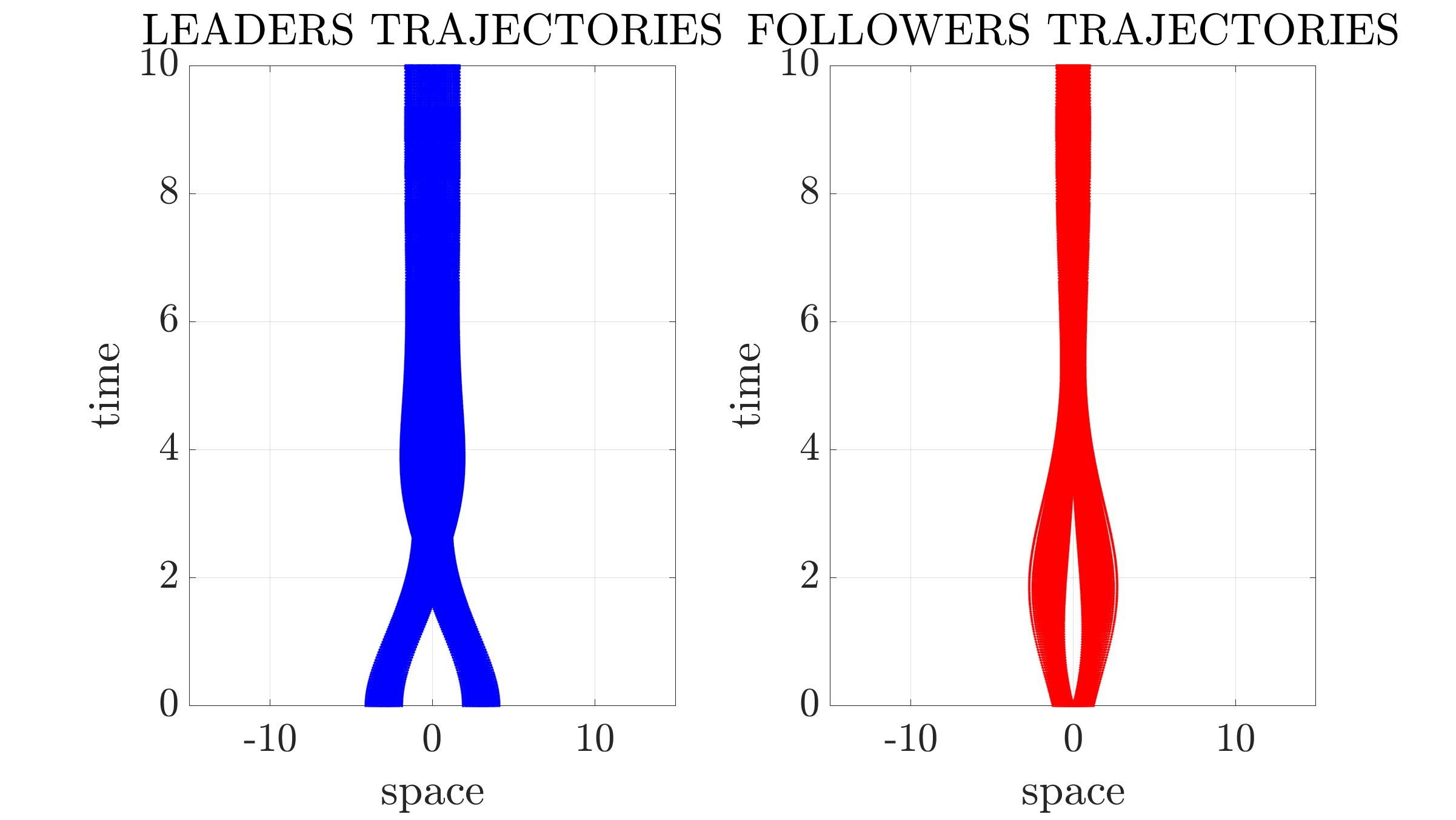}
\caption{micro-micro, $\alpha = 1$}
\end{subfigure}
\hfill
\begin{subfigure}[b]{0.32\textwidth}
\includegraphics[width=\linewidth]{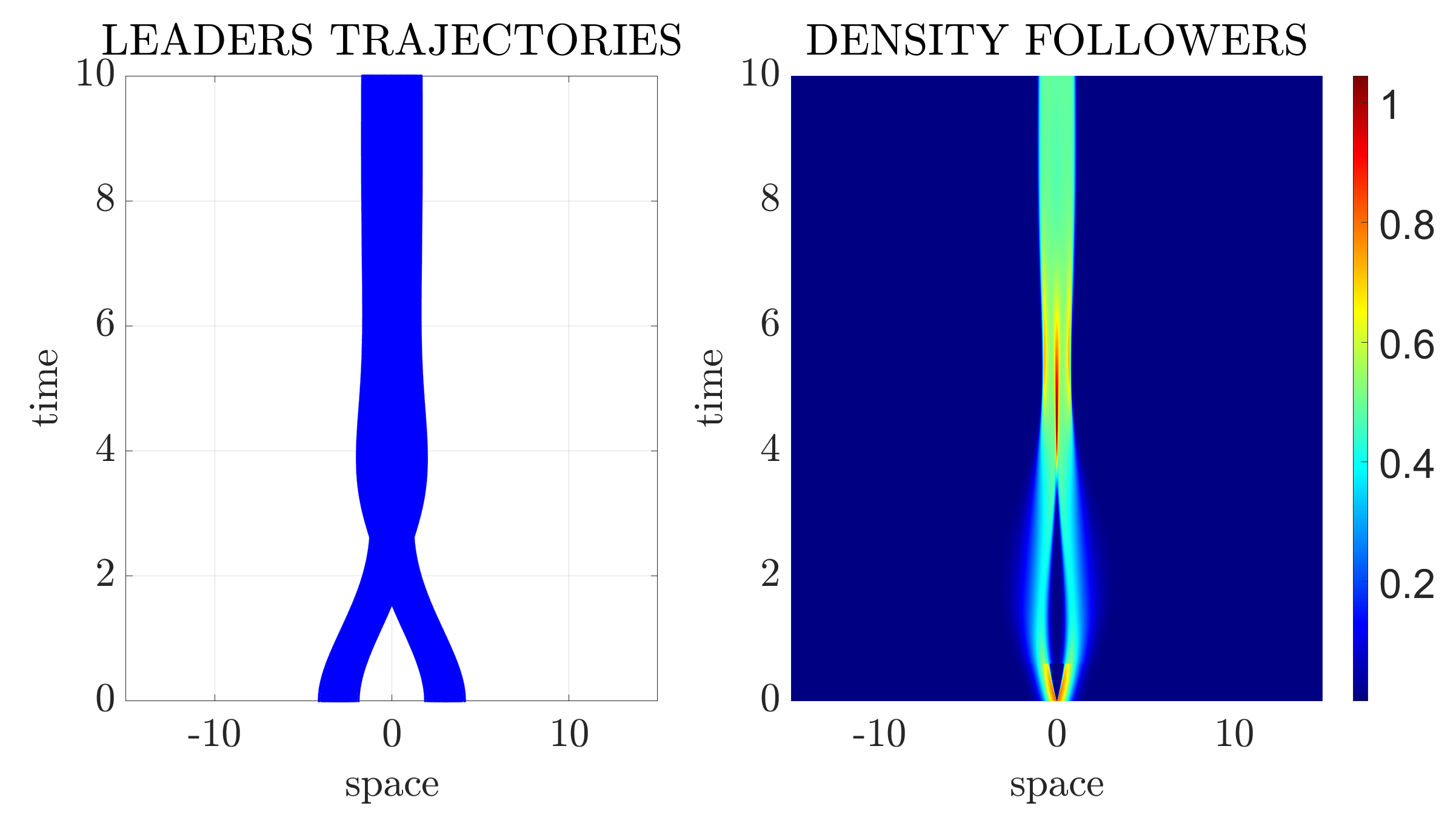}
\caption{micro-macro, $\alpha = 1$}
\end{subfigure}
\hfill
\begin{subfigure}[b]{0.32\textwidth}
\includegraphics[width=\linewidth]{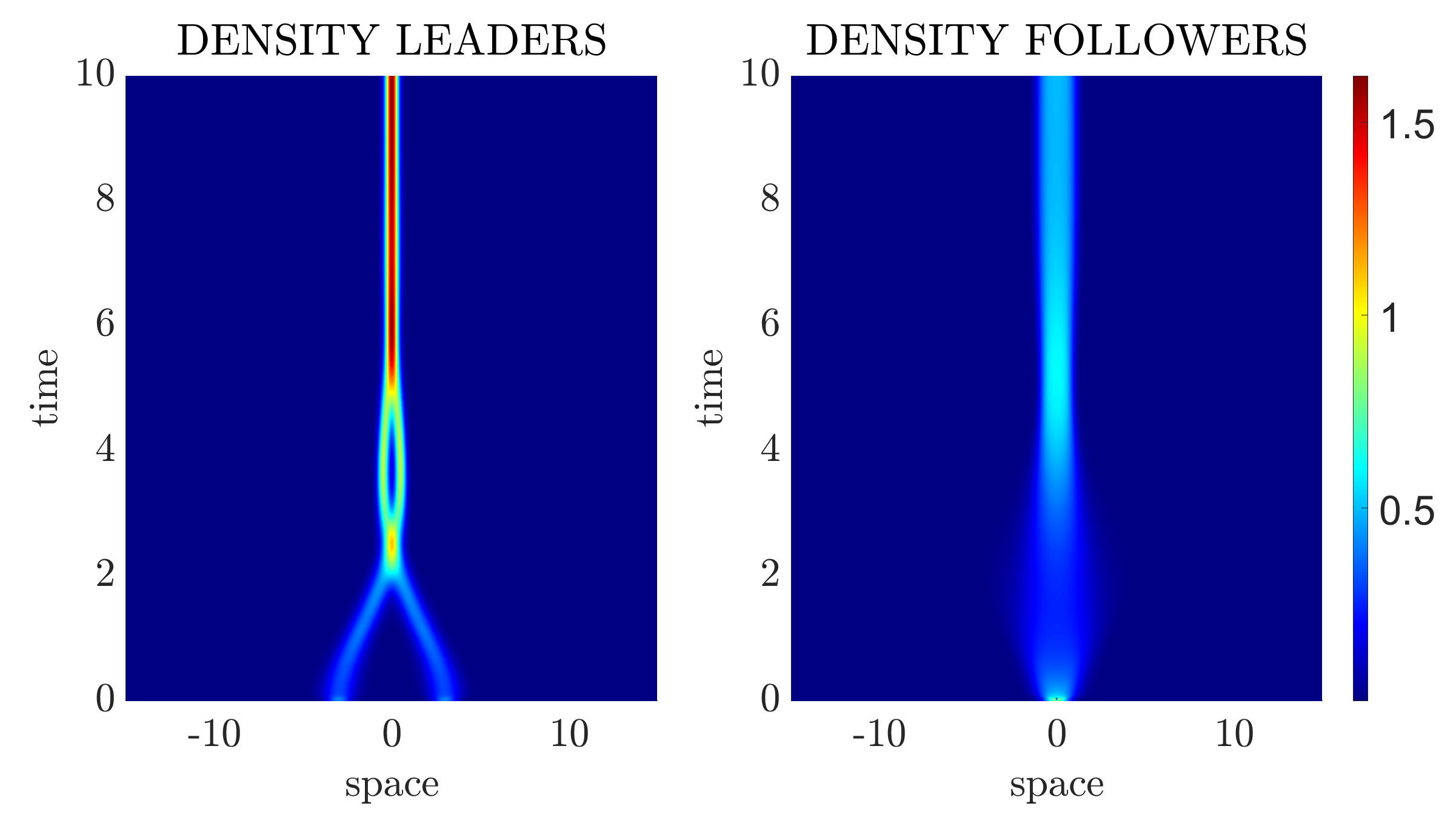}
\caption{macro-macro, $\alpha = 1$}
\end{subfigure}

\caption{Test 2. Comparison of micro/macro configurations for different values of $\alpha$. Each subfigure displays, in the left panel, the trajectories of the leaders in the microscopic model and the corresponding leader density in the macroscopic model; the right panel shows the analogous information for the followers.}
\label{fig:test2_1}
\end{figure}

\begin{figure}[ht]
\centering
\begin{subfigure}[b]{0.3\textwidth}
\includegraphics[width=\linewidth]{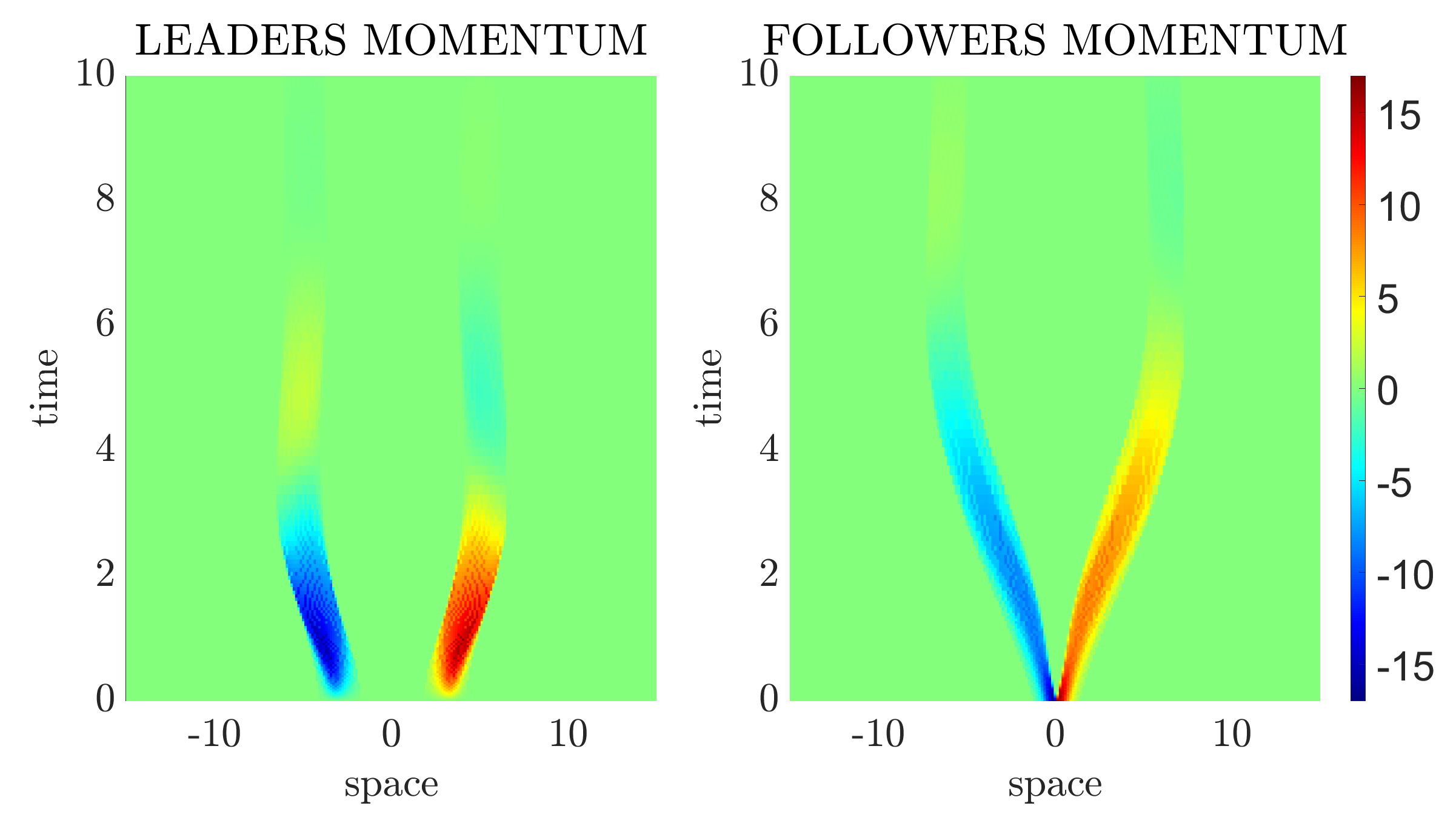}
\caption{micro-micro, $\alpha = 0$}
\end{subfigure}
\hfill
\begin{subfigure}[b]{0.3\textwidth}
\includegraphics[width=\linewidth]{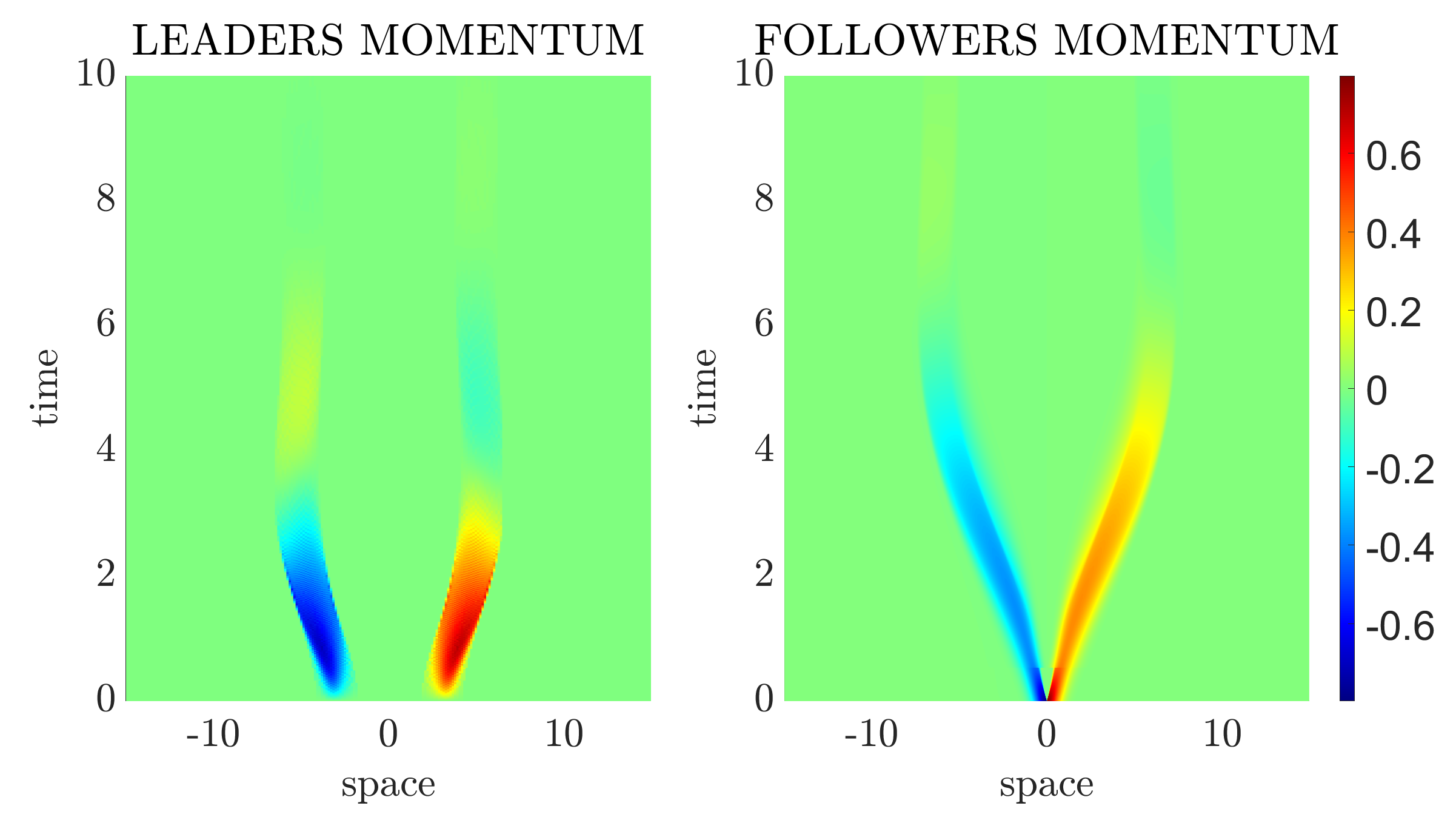}
\caption{micro-macro, $\alpha = 0$}
\end{subfigure}
\hfill
\begin{subfigure}[b]{0.3\textwidth}
\includegraphics[width=\linewidth]{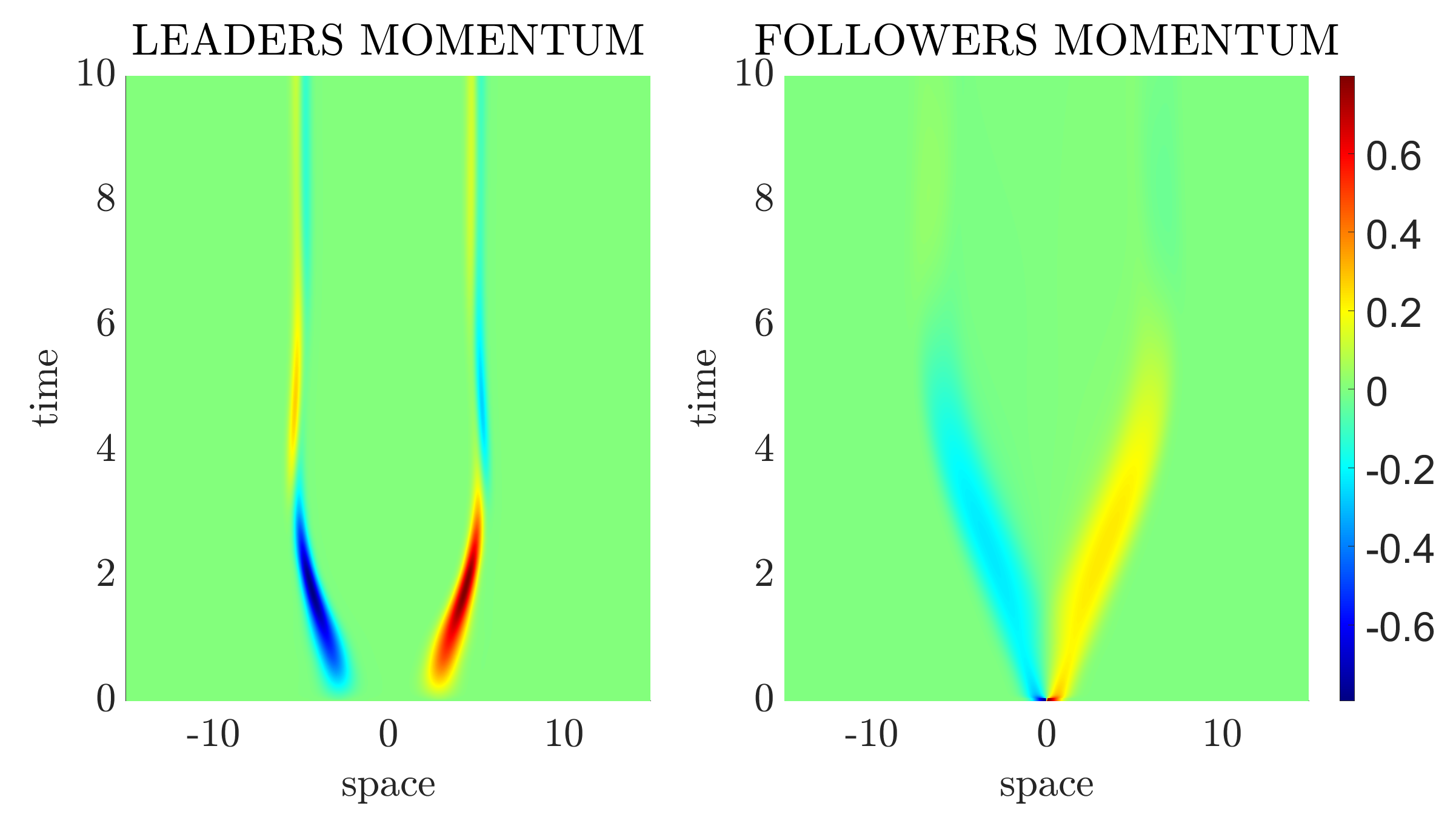}
\caption{macro-macro, $\alpha = 0$}
\end{subfigure}

\vspace{0.5em} 

\begin{subfigure}[b]{0.3\textwidth}
\includegraphics[width=\linewidth]{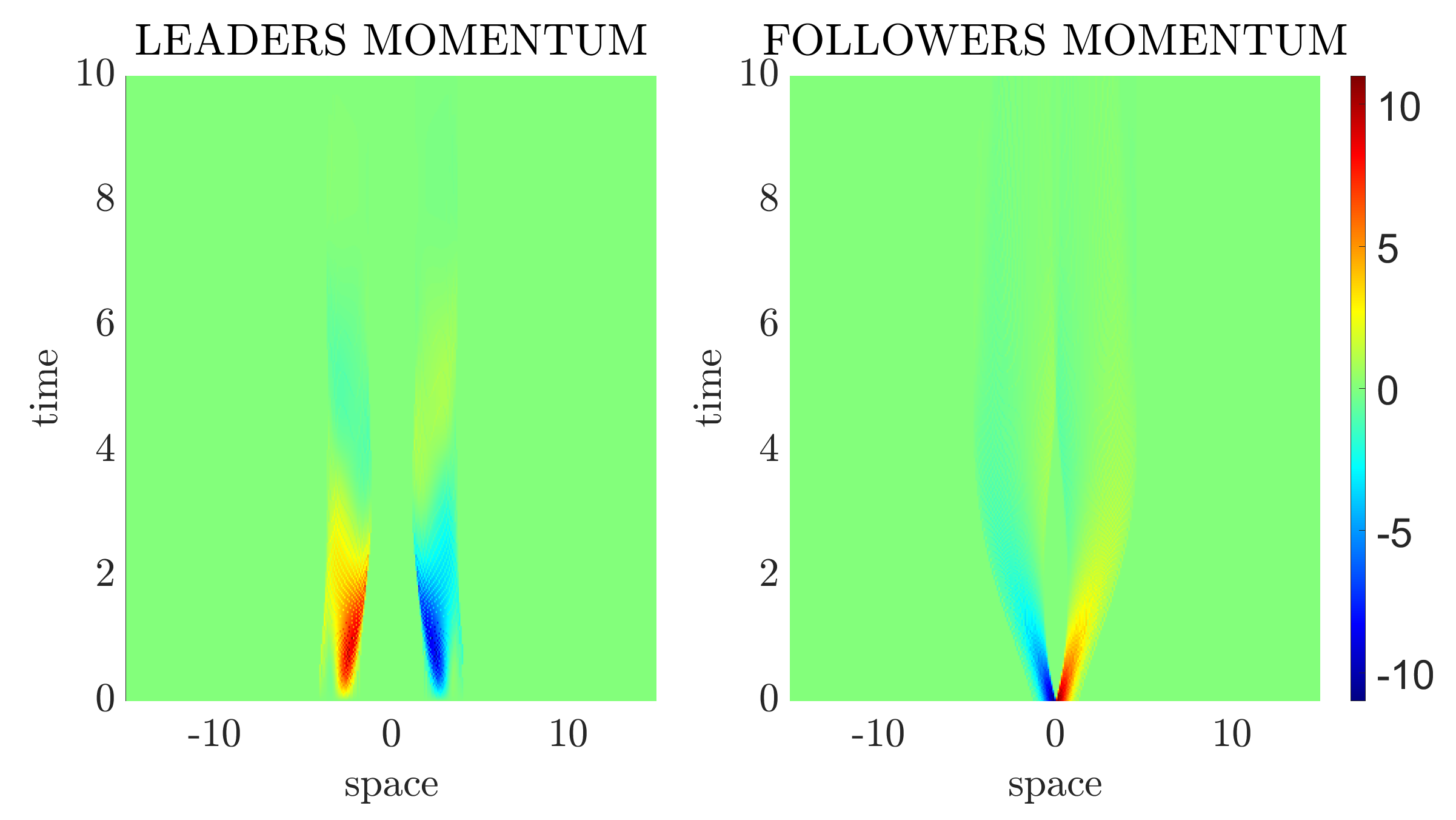}
\caption{micro-micro, $\alpha = 0.5$}
\end{subfigure}
\hfill
\begin{subfigure}[b]{0.3\textwidth}
\includegraphics[width=\linewidth]{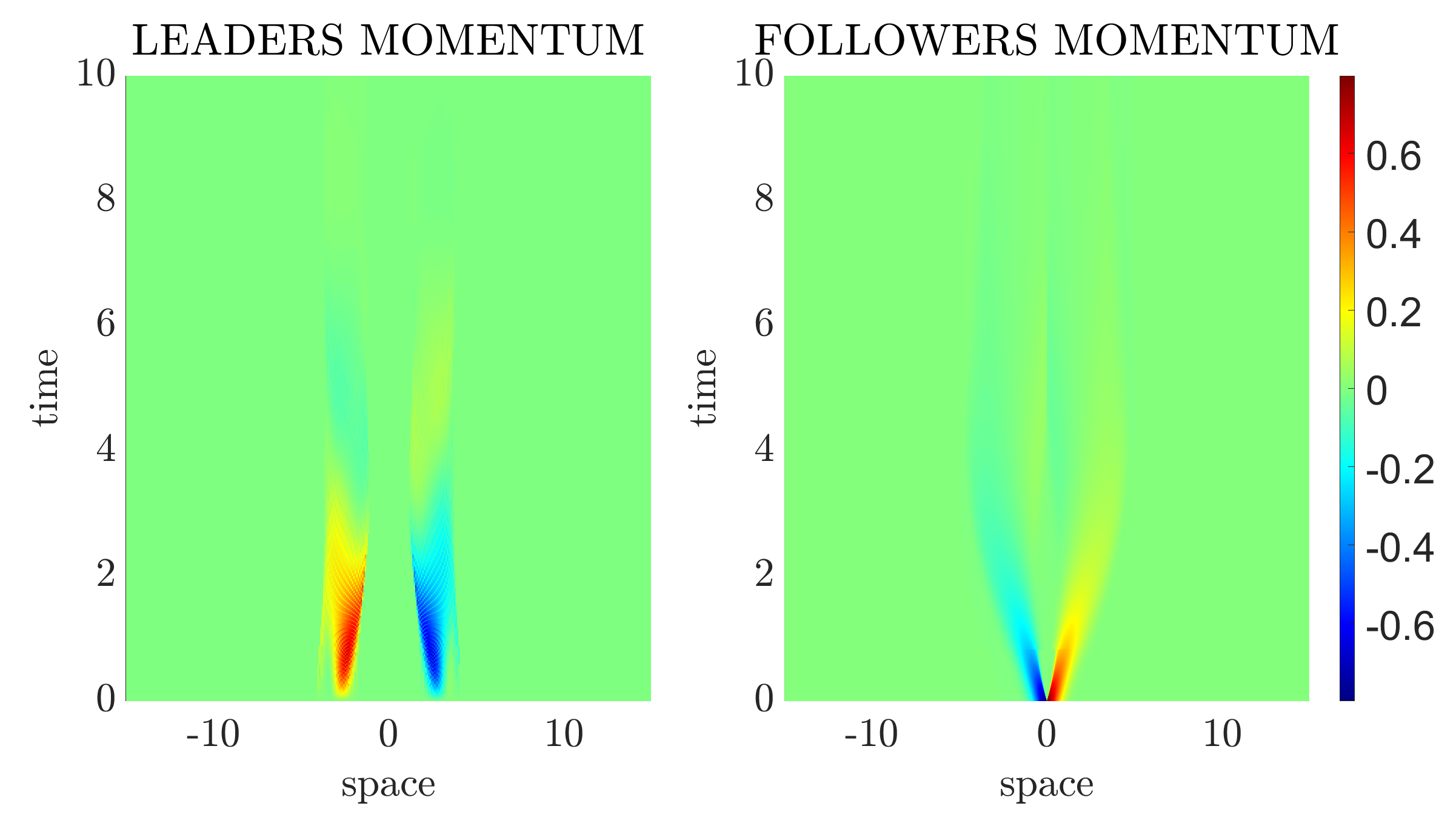}
\caption{micro-macro, $\alpha = 0.5$}
\end{subfigure}
\hfill
\begin{subfigure}[b]{0.3\textwidth}
\includegraphics[width=\linewidth]{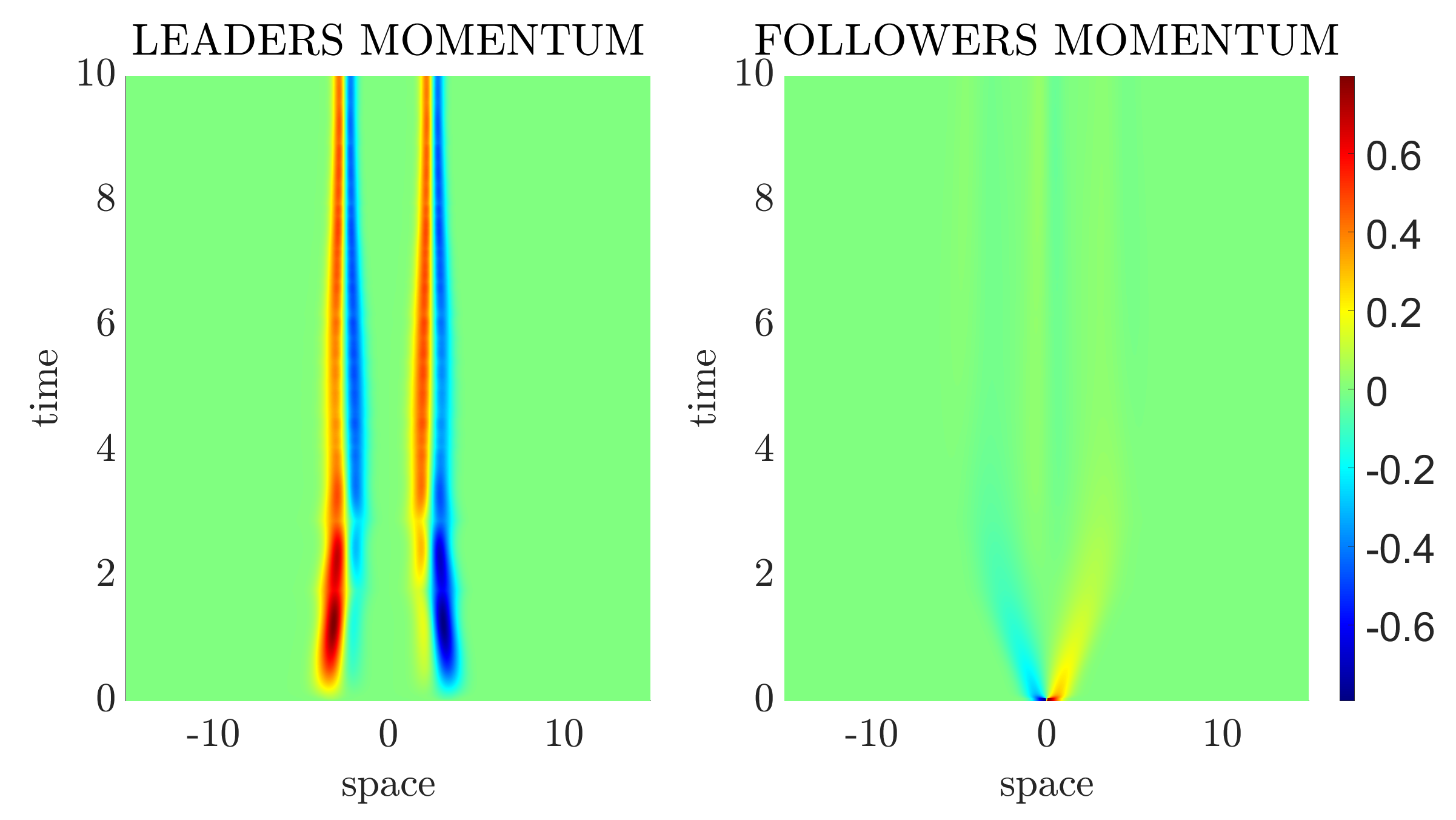}
\caption{macro-macro, $\alpha = 0.5$}
\end{subfigure}

\vspace{0.5em} 

\begin{subfigure}[b]{0.3\textwidth}
\includegraphics[width=\linewidth]{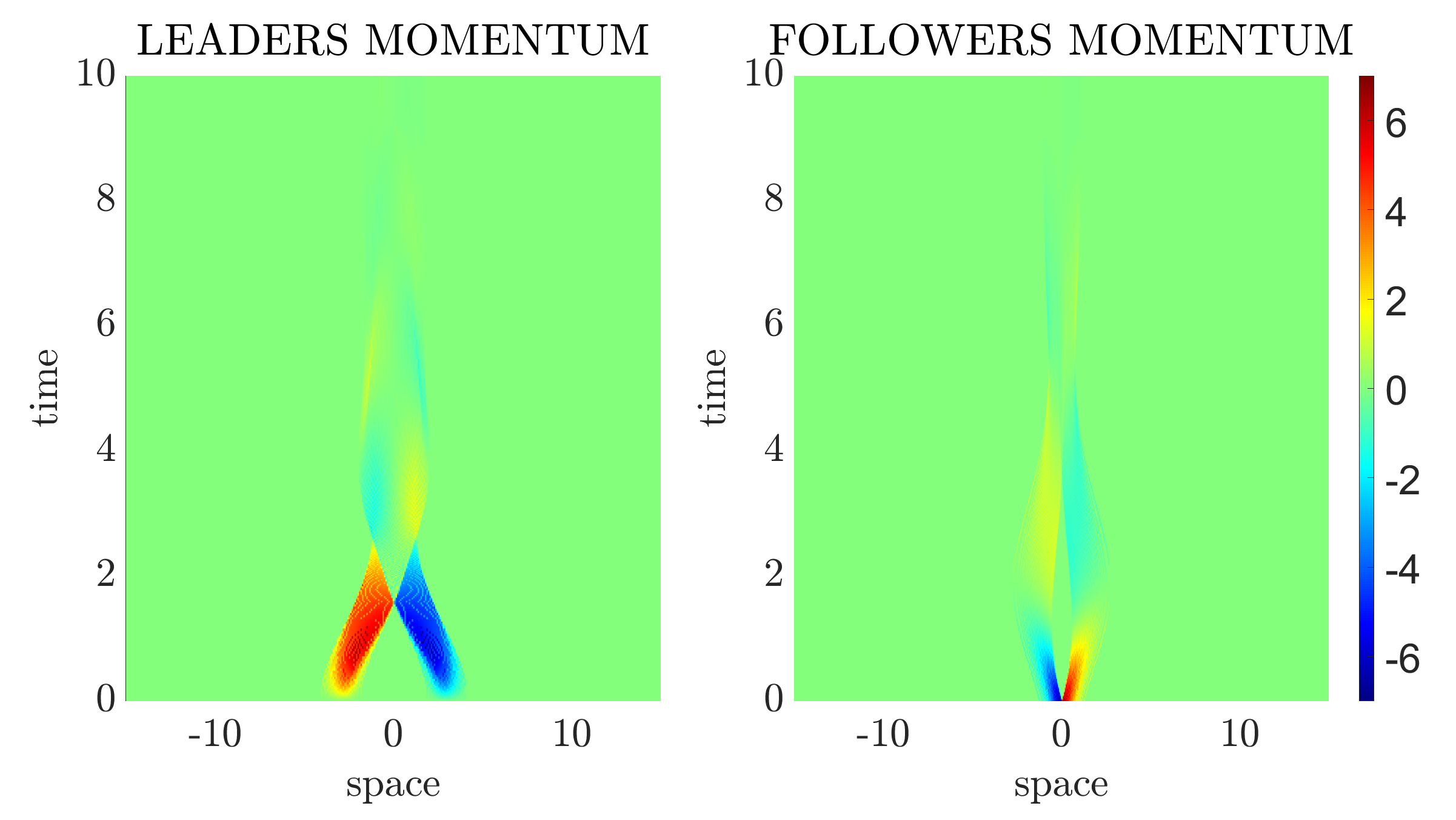}
\caption{micro-micro, $\alpha = 1$}
\end{subfigure}
\hfill
\begin{subfigure}[b]{0.3\textwidth}
\includegraphics[width=\linewidth]{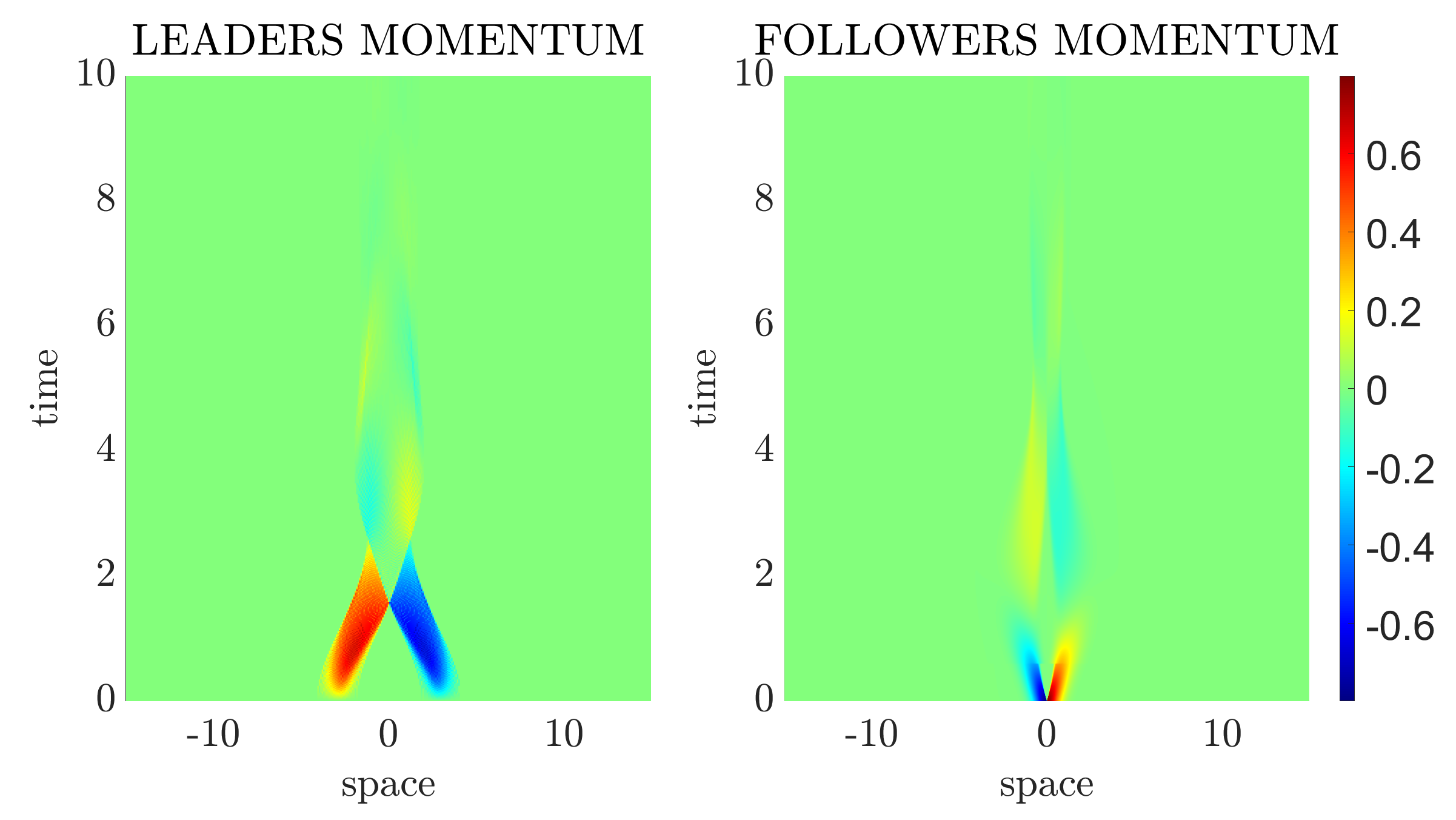}
\caption{micro-macro, $\alpha = 1$}
\end{subfigure}
\hfill
\begin{subfigure}[b]{0.3\textwidth}
\includegraphics[width=\linewidth]{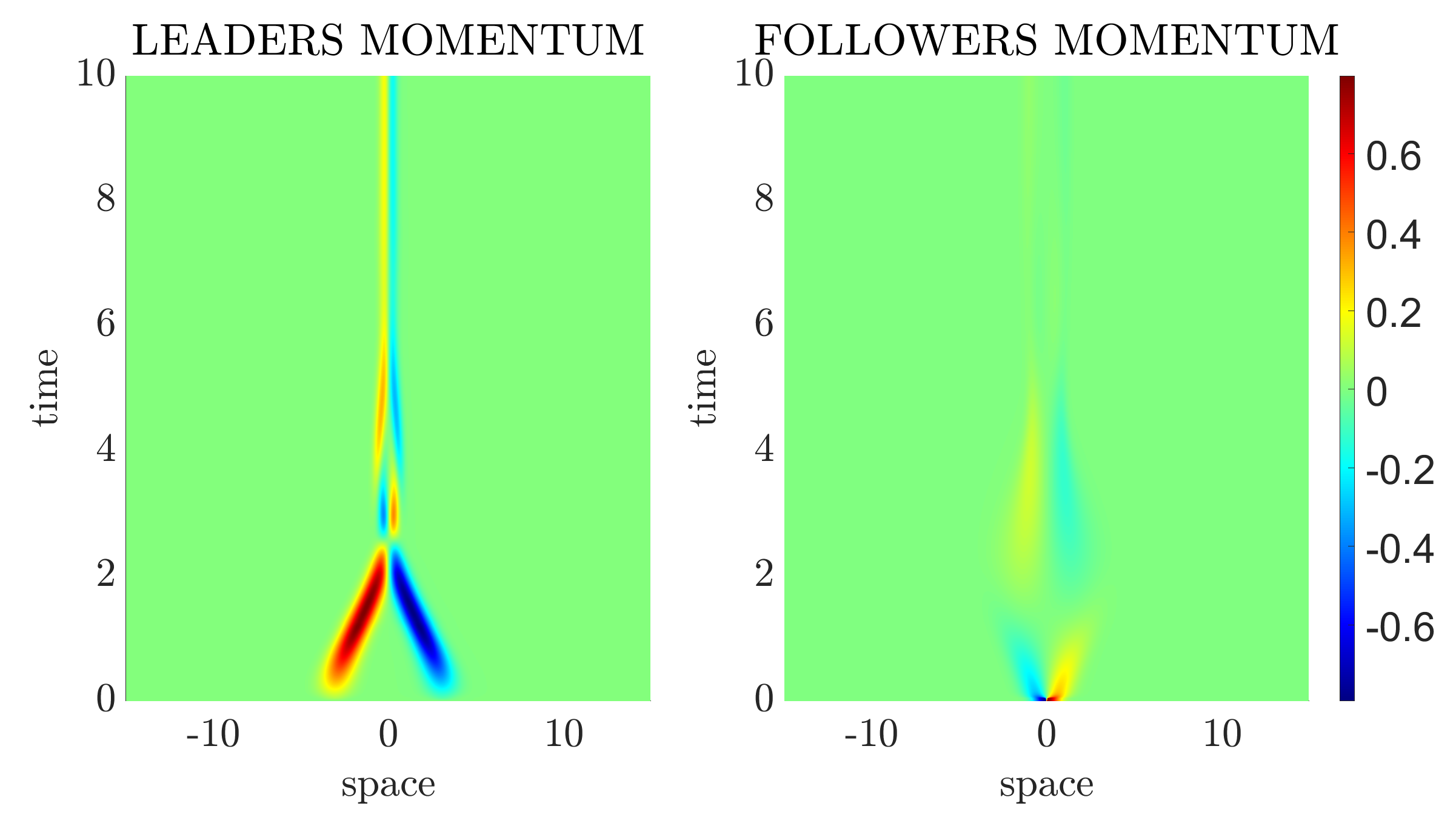}
\caption{macro-macro, $\alpha = 1$}
\end{subfigure}

\caption{Test 2. Comparison of micro/macro configurations for different values of $\alpha$. Each subfigure displays, in the left panel, the momentum of the leaders; the right panel shows the analogous information for the followers.}
\label{fig:test2_2}
\end{figure}

\subsection{Test 3: Finite-time blow-up and shock formation with structured leaders} 

In this test, we investigate the behavior of the three models in a scenario where the follower population, in its macroscopic description, develops a finite-time density blow-up and a shock in the velocity field.

The spatial domain is the interval $[-10, 10]$. The leader population is initially divided into two distinct clusters:
\[
\bar \rho_L(0,x,\xi) = \tfrac{1}{2}\Big(\mathsf{G}_{-7}(x)\delta_6(\xi) + \mathsf{G}_{7}(x)\delta_{-6}(\xi)\Big),
\quad
\bar u_L(0,x,\xi) = 0.
\]  Leaders are attracted to two fixed targets located at $\xi = \pm 6$ and interact through an attractive potential $W_L'(r) = r$. The followers are also initially split into two symmetric groups:
\[
\bar \rho_F(0,x) = \tfrac{1}{2}\Big(\mathsf{G}_{-5}(x) + \mathsf{G}_{5}(x)\Big), 
\quad
\bar u_F(0,x) = -\chi_{[-10,0]}(x)+\chi_{[0,10]}(x).
\] Followers interact repulsively among themselves via $W_F'(r) = -\tfrac{r}{2}$, and are attracted to the leaders through $W_C'(r) = r$. A velocity alignment mechanism is included, weighted by the influence function
\(
\phi(r) = \left(1 + |r|^2\right)^{-\beta/2}\), with $\beta = 0.5$. In the microscopic simulations, we use $N = M = 150$ particles, while the macroscopic models employ a grid spacing of $\Delta x = 0.005$.  
Throughout this test, we fix the coupling parameter at $\alpha = 0.5$ to assess how moderate leader–follower attraction affects the onset of concentration and shock formation.

The numerical results are summarized in Fig.~\ref{fig:test3}, which shows the particle trajectories and velocities for the microscopic case, and the corresponding density and mean velocity fields for the macroscopic models. We observe that shortly after time $t = 3$, the trajectories of the two follower groups collapse at $x = 0$, leading to a blow-up in the macroscopic density description. A stationary shock also forms in the follower velocity field, which is not directly visible at the microscopic level.

Additionally, the trajectories of the two leader groups intersect shortly after $t = 1$. However, as can be seen in the macroscopic representation, the leader dynamics remain separated in the structural variable $\xi$ and do not generate a shock in the velocity field. This is because the interaction distribution $g$ is assumed to be a convex combination of Dirac masses, which means that the leader population is described as a collection of distinct subpopulations indexed by $\xi$. As a result, the governing equations for the leaders naturally decompose into two separate subsystems, each representing one leader group that interacts with the followers independently.
Consequently, even if the physical positions of the leader groups coincide, their densities remain separated in the structural space, preventing nonphysical merging and ensuring that no singularities arise in the leaders’ velocity field, see Sec. \ref{app_ffs} for further details. This test highlights how the interplay between repulsion, attraction, and alignment can lead to finite-time singularities in the follower population, while the structured description of the leaders prevents the formation of analogous singularities in their velocity field.

Overall, this test confirms that the proposed multiscale framework is able to capture complex phenomena such as finite-time blow-up and shock formation in the followers' dynamics, while consistently preserving the separation of leader groups through the additional structural variable $\xi$.  
This highlights the importance of the micro-macro coupling in reproducing realistic aggregation and interaction patterns that may not be visible at a purely microscopic level, and demonstrates how the structured continuum description prevents nonphysical merging of subpopulations.  These observations reinforce the theoretical results on the well-posedness and convergence of the models, and provide a solid numerical basis for future investigations of more realistic interaction scenarios and control strategies.

\begin{figure}[ht]
\centering
\begin{subfigure}[b]{0.3\textwidth}
\includegraphics[width=\linewidth]{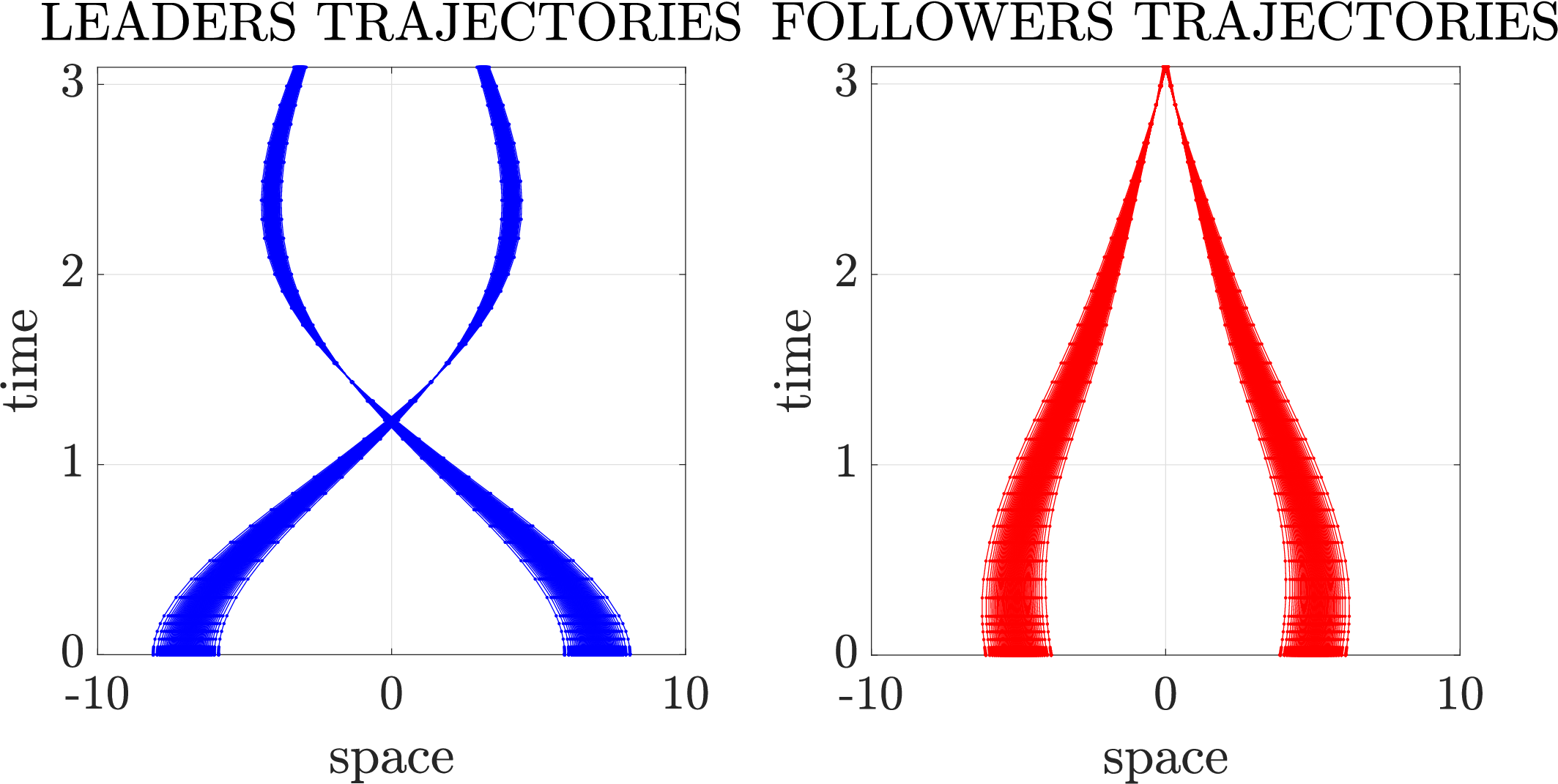}
\caption{micro-micro}
\end{subfigure}
\hfill
\begin{subfigure}[b]{0.3\textwidth}
\includegraphics[width=\linewidth]{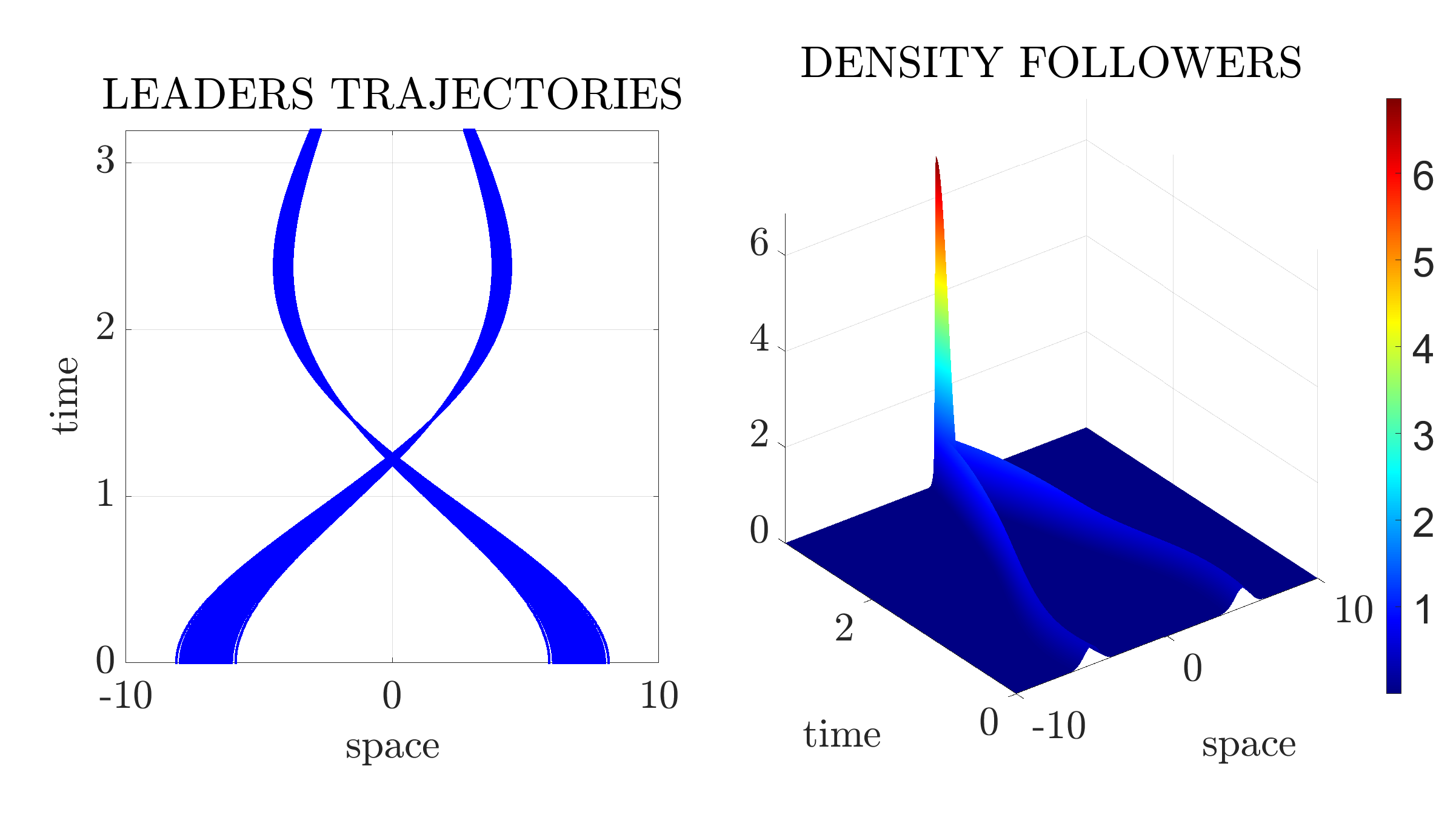}
\caption{micro-macro}
\end{subfigure}
\hfill
\begin{subfigure}[b]{0.3\textwidth}
\includegraphics[width=\linewidth]{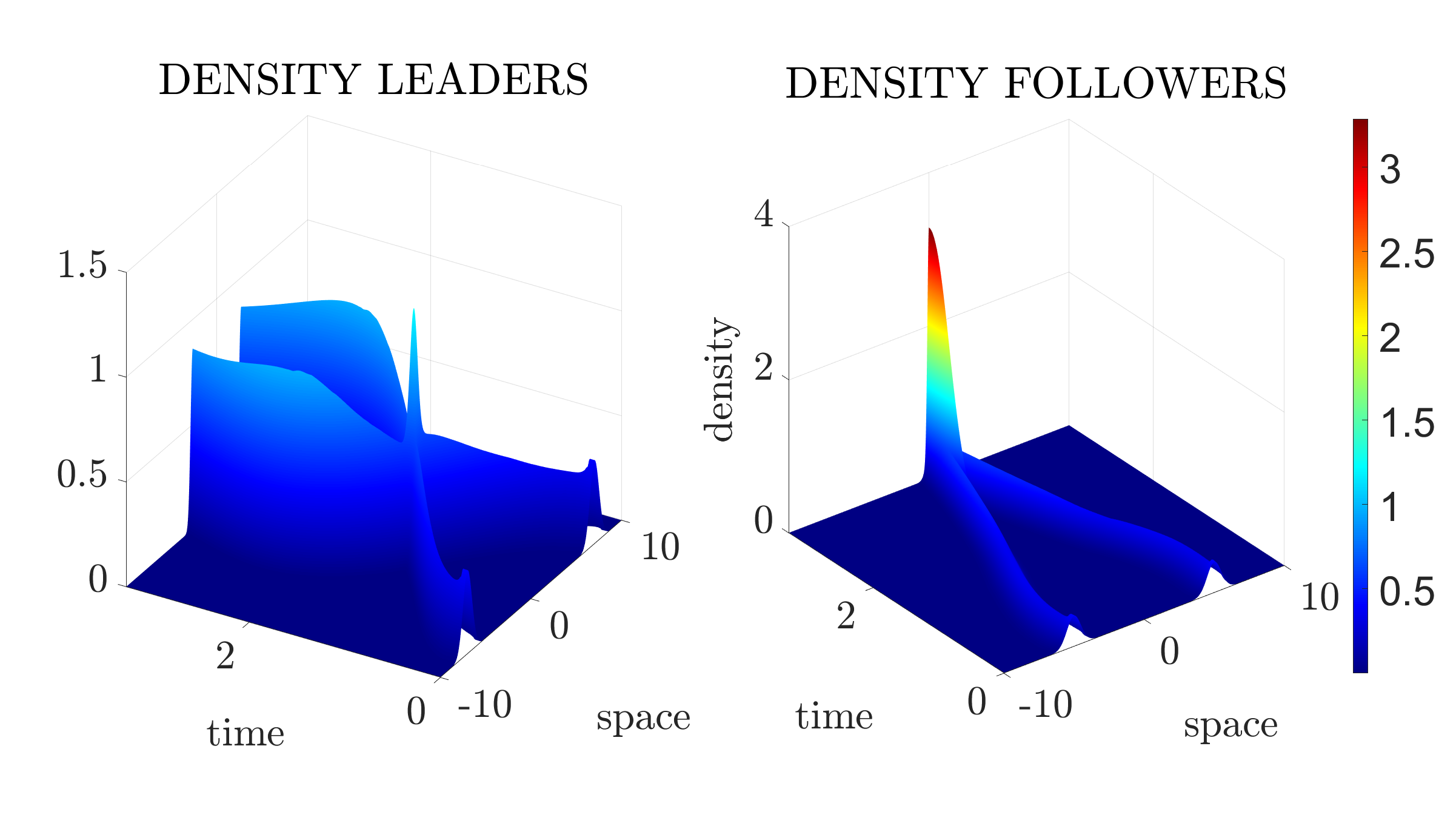}
\caption{macro-macro}
\end{subfigure}

\vspace{0.5em} 

\begin{subfigure}[b]{0.3\textwidth}
\includegraphics[width=\linewidth]{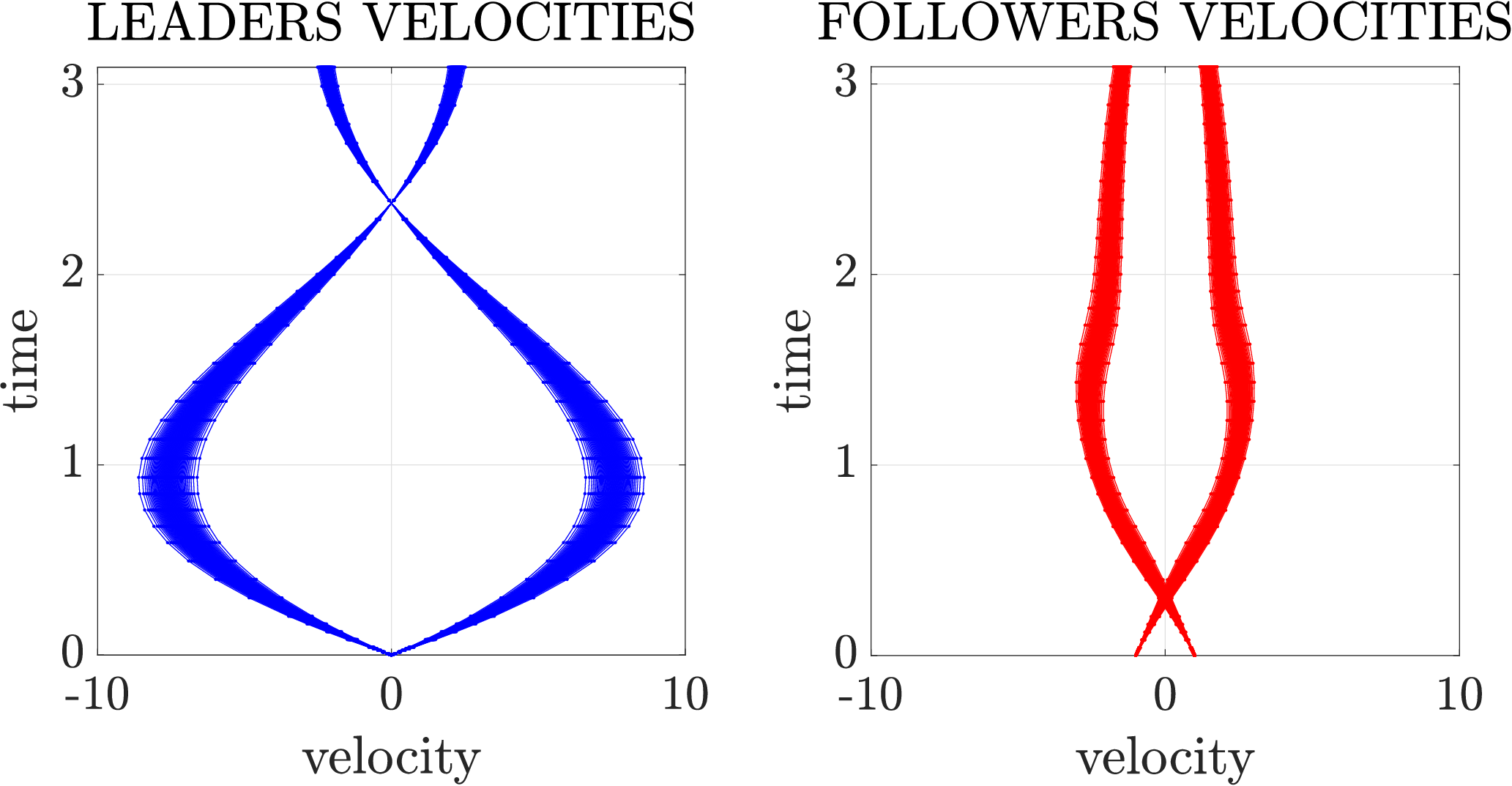}
\caption{micro-micro}
\end{subfigure}
\hfill
\begin{subfigure}[b]{0.3\textwidth}
\includegraphics[width=\linewidth]{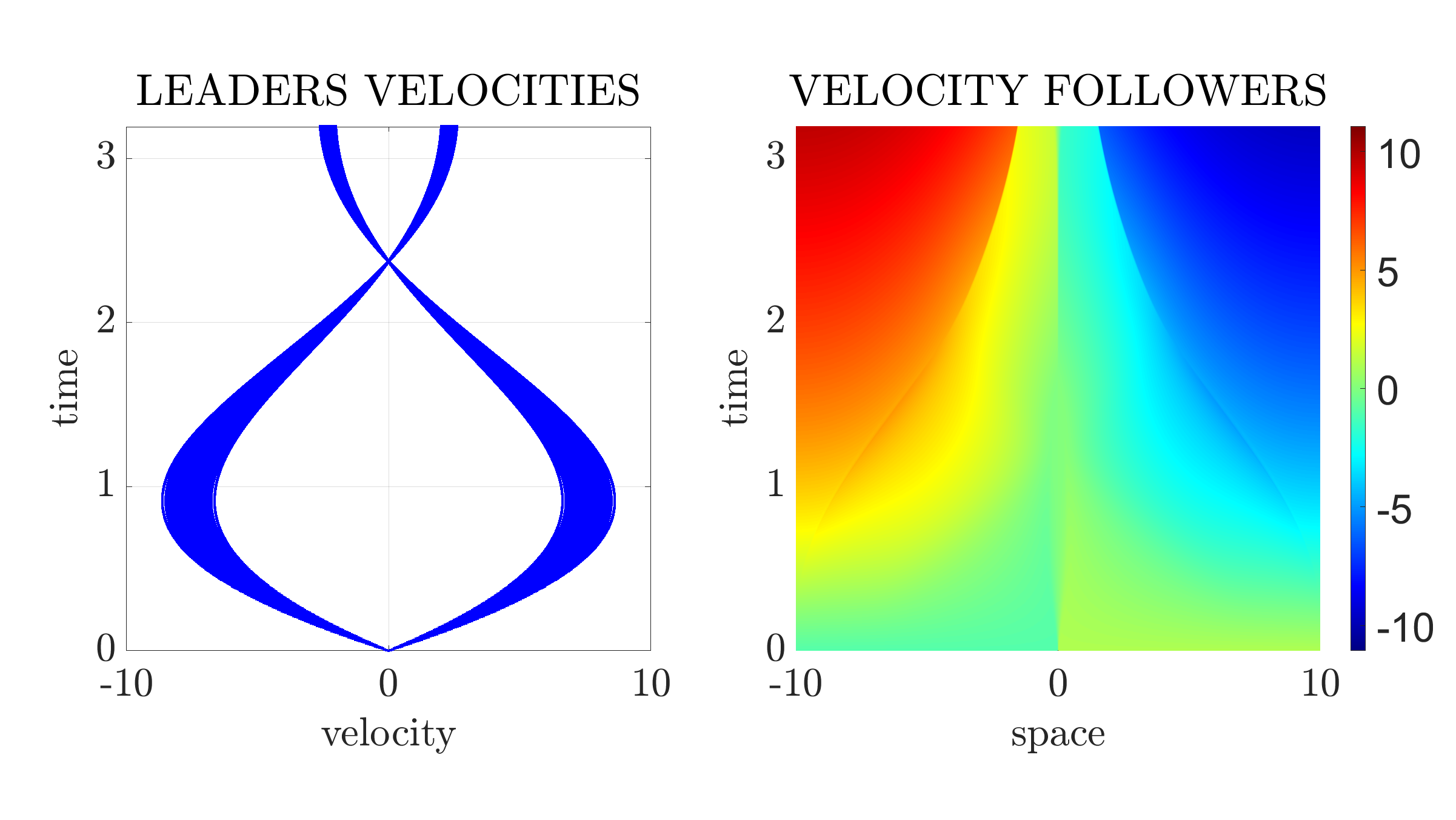}
\caption{micro-macro}
\end{subfigure}
\hfill
\begin{subfigure}[b]{0.3\textwidth}
\includegraphics[width=\linewidth]{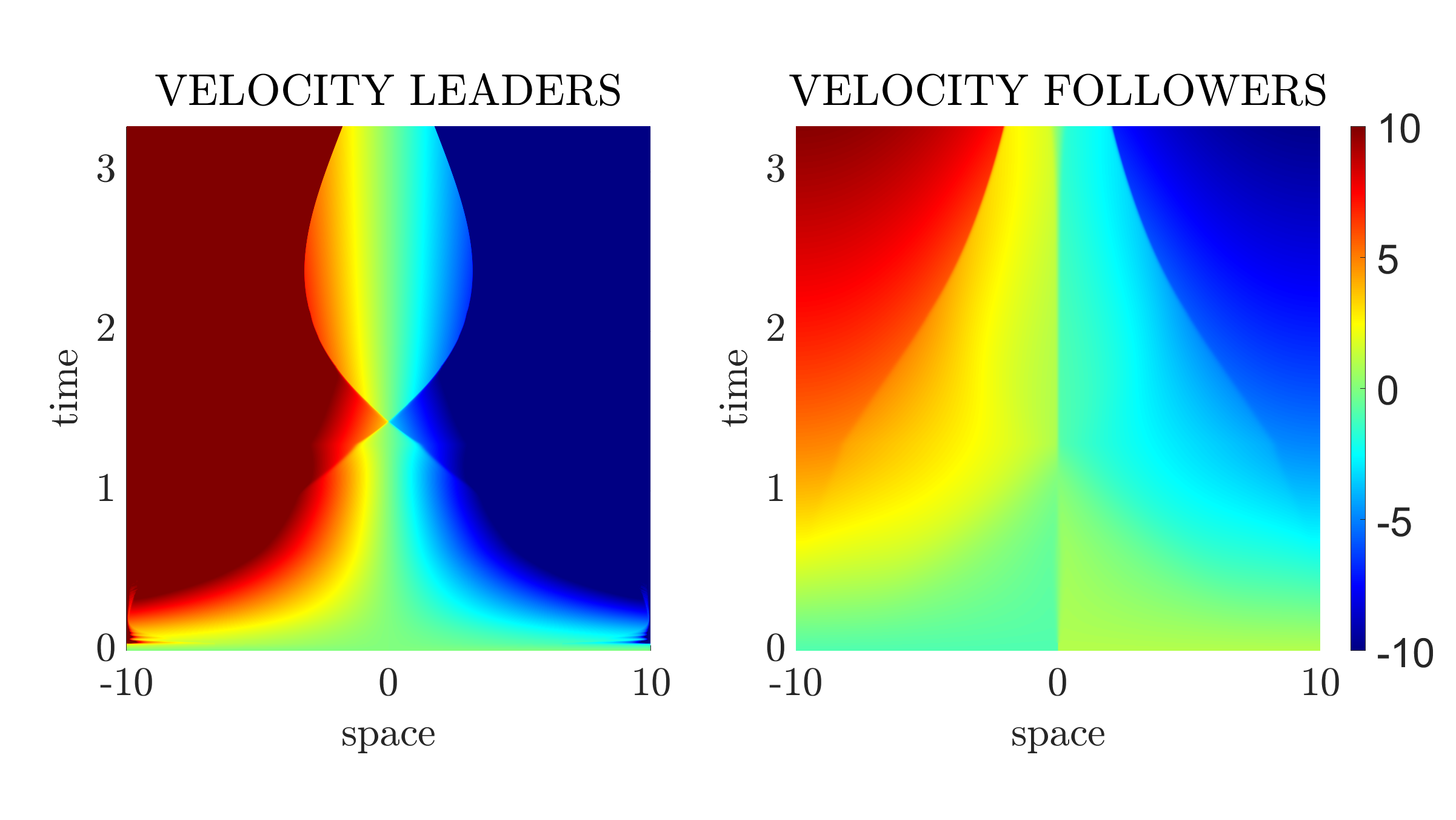}
\caption{macro-macro}
\end{subfigure}
\caption{Test 3. Subfigures A, B, and C show the trajectories (in the microscopic case) and the densities (in the macroscopic case) of the two populations of leaders and followers, while subfigures D, E, and F display their velocities.}
\label{fig:test3}
\end{figure}

%
%
%
%
%
 \section*{Acknowledgments}

 The research of YPC and SS was supported by the NRF grant no. 2022R1A2C1002820 and RS-202400406821. 

GA and MP acknowledge MUR-PRIN Project 2022 No. 2022N9BM3N ``Efficient numerical schemes and optimal control methods for time-dependent partial differential equations'' financed by the European Union - Next Generation EU. GA and MP are members of the GNCS-INdAM Group.

%
%
%
%
%

\appendix

%
%
%
%
%
\section{Existence and uniqueness of solutions to the limit systems}\label{app_existence}
In this appendix, we establish the local-in-time existence and uniqueness of regular solutions for two distinct mean-field limit systems derived from the leader-follower particle model. The first is the micro-macro system, which describes the coupling of finitely many leader particles with a continuum of followers governed by fluid-type equations. The second is the macro-macro system, in which both leaders and followers are represented as interacting fluid components. We begin with the analysis of the micro-macro system.
%
%
%
%
%
\subsection{Micro-macro system}
We consider a coupled system in which $N$ leader particles evolve under attraction/repulsive and alignment forces involving both other particles and the continuum fluid, while the fluid component follows pressureless Euler-type dynamics with additional coupling to the particles. Specifically, the system reads:
\begin{align}\label{pf_main}
\begin{aligned}
    &\ddt \bar{x}_i = \bar{v}_i, \quad i = 1,\dots, N, \quad t > 0,\\[1mm]
    &\ddt \bar{v}_i = -(1-\alpha)\left(\bar{x}_i - \bar x_{d_i} \right) -\alpha (\bar{x}_i - \langle x \rangle_{\rho_F} ) - \bar{v}_i - \nabla W_L *   \bar\varrho_{L}^N(\bar{x}_i),\\[1mm]
    &\partial_t \rho_F + \nabla \cdot (\rho_F u_F) = 0, \quad x\in \Omega, \quad t > 0,\\[1mm]
    &\partial_t (\rho_F u_F) + \nabla \cdot (\rho_F u_F\otimes u_F) = -\rho_F\, \nabla W_F * \rho_F - \rho_F\, \nabla W_C *   \bar\varrho_{L}^N \\
    &\hspace{5cm} + \rho_F \iint_{\Omega \times\mathbb{R}^d} \phi(x-y)\,(w - u_F(x))\,\bar\mu_L^N(\dy\dw),
\end{aligned}
\end{align}
where 
\[
\langle x \rangle_{\rho_F} = \int_{\Omega} x \rho_F \,\dx, \quad  \bar\varrho_L^{N} = \frac1N \sum_{i=1}^N \delta_{\bar{x}_i}, 
\quad \bar\mu_L^{N} = \frac1N \sum_{i=1}^N \delta_{(\bar{x}_i, \bar v_i)}.
\]

To streamline notation, we suppress the subscript $F$ in $\rho_F$ and $u_F$ and simply write $\rho$ and $u$ when no ambiguity arises. The resulting system takes the form:
\begin{align}\label{fluid}
\begin{aligned}
    &\ddt \bar{x}_i = \bar{v}_i, \quad i = 1,\dots, N, \quad t > 0,\\[1mm]
    &\ddt \bar{v}_i = -(1-\alpha)\left(\bar{x}_i - \bar x_{d_i} \right) -\alpha (\bar{x}_i - \langle x \rangle_{\rho} ) - \bar{v}_i - \nabla W_L *   \bar\varrho_{L}^N(\bar{x}_i),\\[1mm]
    &\partial_t \rho + \nabla \cdot (\rho u) = 0, \quad x\in \Omega, \quad t > 0,\\[1mm]
    &\partial_t (\rho u) + \nabla \cdot (\rho u\otimes u) = -\rho\, \nabla W_F * \rho - \rho\, \nabla W_C * \bar\varrho_{L}^N  + \rho \iint_{\Omega \times\mathbb{R}^d } \phi(x-y)\,(w-u(x))\,\bar\mu_L^N(\dy\dw),
\end{aligned}
\end{align}
with initial data
\[
(\rho,u)|_{t=0} = (\rho_0,u_0), \quad x \in \Omega,
\]
and prescribed initial particle positions and velocities
\[
\{ (\bar{x}_i(0),\bar{v}_i(0))\}_{i=1}^N = \{ (\bar{x}_{i0},\bar{v}_{i0})\}_{i=1}^N.
\]

Our aim is to prove the local-in-time existence and uniqueness of regular solutions to this coupled system. We begin by introducing the notion of regularity.
\begin{definition}\label{def_strong}
Let $s > d/2+1$ and $T>0$. A triple $( \{(\bar{x}_i,\bar{v}_i) \}_{i=1}^N, \rho, u)$ is said to be a regular solution of \eqref{fluid} on the time interval $[0,T]$ if
\begin{itemize}
    \item[(i)] The particle component $\{ (\bar{x}_i,\bar{v}_i) \}_{i=1}^N$ is a classical solution of the particle equations in \eqref{fluid} and satisfies
    \[
    (\bar{x}_i,\bar{v}_i) \in \calC^1[0,T] \times \calC^1[0,T].
    \]
    \item[(ii)] The fluid component $(\rho, u)$ satisfies
    \[
    (\rho, u) \in C\bigl([0,T]; H^s(\Omega)\bigr) \times C\bigl([0,T]; H^{s+1}(\Omega)\bigr),
    \]
    and solves the fluid equations in \eqref{fluid} in the sense of distributions.
\end{itemize}
\end{definition}

\begin{theorem}\label{thm_ex2}
Let $s > d/2+1$. Suppose that the interaction potentials and communication weights satisfy assumptions ${\bf (A1)}$--${\bf (A3)}$ and additionally, 
\[
\nabla W_F \in \calW^{1,1}(\Omega), \quad \nabla W_C \in H^{s+1}(\Omega), \quad \phi \in H^{s+1}(\Omega).
\]
Then, for any constants $0 < N < M$, there exists a positive time $T^*$ (depending only on $N$ and $M$) such that if
\[
\max\left\{\|\rho_0\|_{H^s}, \, \|u_0\|_{H^{s+1}}\right\}\leq N, \quad \inf_{x \in \Omega}\rho_0(x) >0, \quad \rho_0 \in L^1_2(\Omega),
\]
and given initial particle data $\{ (\bar{x}_{i0},\bar{v}_{i0})\}_{i=1}^N$, the coupled particle-fluid system \eqref{fluid} has a unique regular solution $(\{ (\bar{x}_i,\bar{v}_i) \}_{i=1}^N, \rho, u)$ in the sense of Definition \ref{def_strong} satisfying
\[
\max\left\{ \sup_{0\le t\le T^*}\|\rho(t,\cdot)\|_{H^s}, \, \sup_{0\le t\le T^*}\|u(t,\cdot)\|_{H^{s+1}} \right\}\le M.
\]
\end{theorem}

\begin{remark}
The regularity assumptions imposed on the interaction potential $W_C$ and the communication weight $\phi$ are more restrictive than those on $W_F$. This is because the terms involving $W_C$ and $\phi$ capture the interactions with the particle system, for which we only assume $\mathcal{P}_2$-regularity on $\bar\mu_L^N$. Consequently, stronger regularity for $W_C$ and $\phi$ is required to obtain regular solutions for the fluid component in \eqref{pf_main}.
\end{remark}

\begin{remark}In the periodic domain case, $\Omega = \T^d$, the assumptions $\nabla W_F \in \mathcal{W}^{1,1}(\Omega)$ and $\rho_0 \in L^1_2(\Omega)$ can be removed. 
\end{remark}

%
%
%
%
%

\subsubsection{Approximate solutions}

We construct approximate solutions $(\rho^n,u^n)$ to the system \eqref{fluid} by solving the following iterative scheme:
\begin{align}\label{fluid_app}
\begin{aligned}
    &\ddt \bar{x}_i^{n+1} = \bar{v}_i^{n+1}, \quad i = 1,\dots, N, \quad t > 0,\\[1mm]
    &\ddt \bar{v}_i^{n+1} = -(1-\alpha)\left(\bar{x}_i^{n+1} - \bar x_{d_i}\right) -\alpha \left(\bar{x}_i^{n+1} - \langle x \rangle_{\rho^{n+1}} \right) - \bar{v}_i^{n+1} - \nabla W_L *  \bar\varrho_{L}^{N, n+1}(\bar{x}_i^{n+1}),\\[1mm]
&\partial_t \rho^{n+1} + u^n \cdot \nabla \rho^{n+1} + \rho^{n+1}\, \nabla \cdot u^n = 0,\\[1mm]
&\rho^{n+1}\, \partial_t u^{n+1} + \rho^{n+1}\, u^n \cdot \nabla u^{n+1} = -\rho^{n+1}\, \nabla W_F * \rho^{n+1} - \rho^{n+1}\, \nabla W_C *  \bar\varrho_L^{N, n+1}\\[1mm]
&\quad + \rho^{n+1} \iint_{\Omega\times\mathbb{R}^d } \phi(x-y) ( w - u^{n+1}(x))\, \bar\mu_L^{N, n+1}(\dy\dw),
\end{aligned}
\end{align}
where 
\[
 \langle x \rangle_{\rho^{n}} = \int_{\Omega} x \rho^n \,\dx, \quad  \bar\varrho_L^{N, n} = \frac1N \sum_{i=1}^N \delta_{\bar{x}_i^{n}}, 
 \quad 
\bar\mu_L^{N, n} = \frac1N \sum_{i=1}^N \delta_{(\bar{x}_i^{n}, \bar{v}_i^{n})}.
\]
for $n \in \N$.
The initial data and the first iterate are prescribed by
\[
(\bar{x}_i^{n}(0), \bar{v}_i^{n}(0))= (\bar x_i(0), \bar v_i(0)), \quad (\rho^n(0,x), u^n(0,x)) = (\rho_0(x), u_0(x)), \quad \forall\, n \ge 1,\; x \in \Omega,
\]
and
\[
u^0(t,x) = u_0(x), \quad (t,x) \in \mathbb{R}_+ \times \Omega.
\]

Our goal in this part is to prove the following proposition.

\begin{proposition}\label{prop_app}
Suppose that the initial data $(\rho_0, u_0)$ satisfy the conditions in Theorem \ref{thm_ex2}. Then there exists a time $T^*>0$, depending only on $M$ and $N$, such that the system \eqref{fluid_app} admits a sequence of unique regular solutions $\{ (\rho^n, u^n)\}_{n\in\N}$ on $[0,T^*]$ satisfying
\begin{equation}\label{uni_app0}
\max\left\{\sup_{n\in\N} \sup_{0\le t\le T^*}   \|\rho^n(t)\|_{H^s}, \, \sup_{n\in\N} \sup_{0\le t\le T^*}\|u^n(t)\|_{H^{s+1}}  \right\} \le M.
\end{equation}
\end{proposition}

We now present two key lemmas that yield uniform bounds on the approximate solutions.

\begin{lemma}[Estimate of $\rho^n$]\label{lem_rho}
Let $T>0$. For a given $u^n \in L^\infty(0,T; H^{s+1}(\Omega))$ with
\[
\|u^n\|_{L^\infty(0,T; H^{s+1})} \le M,
\]
there exists a constant $T_1 \in (0,T]$ such that
\[
\int_{\Omega} \rho^{n+1}\,\dx = \int_{\Omega} \rho_0\,\dx, \ \forall \, n \in \N,  \quad \sup_{n \in \N}\int_{\Omega} \rho^{n+1} |x|^2\,\dx < \infty,
\quad
\rho^{n+1}(t,x) > 0 \quad \forall \, (t,x) \in [0,T_1]\times \Omega,
\]
and
\[
\sup_{0\le t\le T_1} \|\rho^{n+1}(t)\|_{H^s} \le M.
\]
\end{lemma}
\begin{proof}It is clear from the continuity equation that
\[
\ddt\int_{\Omega} \rho^{n+1}\,\dx = 0,
\]
and thus the first assertion is obtained. 

Define the characteristic curve $\eta^{n+1}(s) = \eta^{n+1}(s;t,x)$ by
\[
\dds\eta^{n+1}(s) = u^n\bigl(s, \eta^{n+1}(s)\bigr), \quad \eta^{n+1}(t)=x, \quad s\in [0,T].
\]
Along this characteristic, the transport equation for $\rho^{n+1}$ implies that
\[
\rho^{n+1}(t,x) = \rho_0\bigl(\eta^{n+1}(0)\bigr) \exp\!\Bigl(-\int_0^t (\nabla \cdot u^n)\bigl(s, \eta^{n+1}(s)\bigr)\, \ds\Bigr).
\]
Since the Sobolev embedding (with $s > d/2+1$) guarantees that
\[
\|\nabla u^n\|_{L^\infty} \le C \|u^n\|_{H^{s+1}} \le C M,
\]
we deduce that
\[
\rho^{n+1}(t,x) \ge \rho_0\bigl(\eta^{n+1}(0)\bigr)\exp(-CMT) > 0,
\]
for $t\in [0,T_1]$, provided $T_1>0$ is chosen small enough.

We also estimate
\begin{align*}
\frac12\ddt\int_{\Omega} \rho^{n+1} |x|^2 &= \int_{\Omega} \rho^{n+1} x \cdot u^n\,\dx \cr
&\leq \frac12\int_{\Omega} \rho^{n+1} |x|^2\,\dx + \frac12 \int_{\Omega} \rho^{n+1} |u^n|^2\,\dx \cr
&\leq \frac12\int_{\Omega} \rho^{n+1} |x|^2\,\dx + M \int_{\Omega} \rho_0\,\dx,
\end{align*}
where we used the positivity of $\rho^{n+1}$ and the mass conservation. Applying Gr\"onwall's lemma concludes the desired result.

A standard energy estimate (see, e.g., \cite[Appendix A]{CK16}) shows that
\[
\ddt \|\rho^{n+1}\|_{H^s}^2 \le C M\, \|\rho^{n+1}\|_{H^s}^2,
\]
so that by Gr\"onwall's lemma,
\[
\sup_{0\le t\le T_1}\|\rho^{n+1}(t)\|_{H^s} \le \|\rho_0\|_{H^s}\, e^{CMT_1}.
\]
Choosing $T_1>0$ small enough to ensure that $\|\rho_0\|_{H^s}\, e^{CMT_1} \le M$ completes the proof.
\end{proof}

\begin{lemma}[Estimate of $(\bar x^{n+1}, \bar v^{n+1})$]\label{lem_par} Let the assumption of Lemma \ref{lem_rho} hold. Then there exists a unique classical solution $(\bar x^{n+1}, \bar v^{n+1})$ to the particle system on the time interval $[0,T_1]$ in \eqref{fluid_app}. Moreover, we have
\[
\sup_{0 \leq t \leq T_1} \left(| \bar x_i^n(t)|^2 + | \bar v_i^n(t)|^2 \right)  <  \left(| \bar x_i(0)|^2 + | \bar v_i(0)|^2 + |\bar{x}_{d_i}|^2 \right) + C
\]
for some $C>0$ independent of $N$ and $n$.
\end{lemma}
\begin{proof}It follows from Lemma \ref{lem_rho} that $\langle x \rangle_{\rho^{n+1}}$ is well-defined on the time interval $[0,T_1]$. Thus, the classical Cauchy--Lipschitz theory guarantees the existence of unique classical solutions. 

We also estimate
\begin{align*}
 \frac12\ddt \left(| \bar x_i^{n+1}(t)|^2 + | \bar v_i^{n+1}(t)|^2 \right) 
&  = (1-\alpha) \bar{x}_{d_i} \cdot  \bar v_i^{n+1}  + \alpha  \bar v_i^{n+1} \cdot \langle x \rangle_{\rho^{n+1}} -   |\bar v_i^{n+1}|^2 -   \bar v_i^{n+1} \cdot  \nabla W_L *  \bar\varrho_{L}^N(\bar{x}_i^{n+1})\cr
&  \leq  \frac{1-\alpha}{2}|\bar{x}_{d_i}|^2 + \frac{\alpha}2 | \langle x \rangle_{\rho^{n+1}}|^2 + \frac12  |\nabla W_L *  \bar\varrho_{L}^N(\bar{x}_i^{n+1})|^2 \cr
&  \leq |\bar{x}_{d_i}|^2 + \int_{\Omega} \rho^{n+1}|x|^2\,\dx + \|\nabla W_L\|_{L^\infty}^2
\end{align*}
due to $\alpha \in [0,1]$. Integrating the above inequality with respect to $t$ concludes the desired result.
\end{proof}

\begin{lemma}[Estimate of $u^n$]\label{lem_u}
Let $T>0$. For a given $u^n \in L^\infty(0,T; H^{s+1}(\Omega))$ with
\[
\|u^n\|_{L^\infty(0,T; H^{s+1})} \le M,
\]
there exists a constant $T_2 \in (0,T_1]$ such that
\[
\sup_{0\le t\le T_2} \|u^{n+1}(t)\|_{H^{s+1}} \le M.
\]
\end{lemma}
\begin{proof}
Since $\rho^{n+1}>0$ (by Lemma \ref{lem_rho}), we can divide the momentum equation in \eqref{fluid_app} by $\rho^{n+1}$ to write
\begin{align}\label{u_eq}
\begin{aligned}
 \partial_t u^{n+1} + u^n\cdot \nabla u^{n+1}  
&  = - \nabla W_F * \rho^{n+1} - \nabla W_C * \bar\varrho_L^{N, n+1} \cr
&\quad + \iint_{\Omega \times\mathbb{R}^d} \phi(x-y) \bigl(w-u^{n+1}(x)\bigr) \bar\mu_L^{N, n+1}(\dy\dw).
\end{aligned}
\end{align}
Before deriving $H^{s+1}$ estimates for $u^{n+1}$, we note that there exists a constant $C>0$, independent of $n$ and $M$, such that
\begin{align*}
\|\nabla W_F * \rho^{n+1}\|_{H^{s+1}} &\le C\, \|\nabla W_F\|_{W^{1,1}}\, \|\rho^{n+1}\|_{H^s} \le C M,\\[1mm]
\|\nabla W_C *  \bar\varrho_L^{N, n+1}\|_{H^{s+1}} &\le C\, \|\nabla W_C\|_{H^{s+1}}\, \| \bar\varrho_L^{N, n+1}\|_{\mathcal{P}} \le C,\\[1mm]
\|\phi *  \langle v \rangle_{\bar\mu^{N, n+1}_L}\|_{H^{s+1}} &\le C\, \|\phi\|_{H^{s+1}}\, \|\bar\mu_L^{N, n+1}\|_{\mathcal{P}_2} \le C,
\end{align*}
and
\begin{align*}
 \| u^{n+1}\, (\phi *  \bar\varrho^{N, n+1}_L)\|_{H^{s+1}}  
&  \le C\Bigl(\|u^{n+1}\|_{H^{s+1}}\, \|\phi * \bar\varrho^{N, n+1}_L\|_{L^\infty} + \|u^{n+1}\|_{L^\infty}\, \|\phi *  \bar\varrho^{N, n+1}_L\|_{H^{s+1}}\Bigr)\\[1mm] 
&  \le C\, \|u^{n+1}\|_{H^{s+1}}\Bigl(\|\phi\|_{L^\infty}\, \|\bar\varrho_L^{N, n+1}\|_{\mathcal{P}} + \|\phi\|_{H^{s+1}}\, \|\bar\varrho_L^{N, n+1}\|_{\mathcal{P}}\Bigr)\\[1mm]
&  \le C\, \|u^{n+1}\|_{H^{s+1}},
\end{align*}
where we have set
\[
\langle v \rangle_{\bar\mu^{N, n+1}_L} := \int_{\mathbb{R}^d} v\, \bar\mu^{N, n+1}_L(\dv),
\]
and used the bound from Lemma \ref{lem_rho} along with the inequality
\[
\iiint_{\Omega\times\mathbb{R}^d} |v|\, \bar\mu^{N, n+1}_L(\dx\dv) \le \|\bar\mu^{N, n+1}_L\|_{\mathcal{P}_2}\,\|\bar\mu^{N, n+1}_L\|_{\mathcal{P}} = \|\bar\mu^{N, n+1}_L\|_{\mathcal{P}_2}.
\]

To obtain $H^{s+1}$ estimates for $u^{n+1}$, we differentiate \eqref{u_eq} $k$ times (with $0\le k\le s+1$) and take the $L^2$ inner product with $\nabla^k u^{n+1}$. A standard commutator estimate then shows that
$$\begin{aligned}
 \frac12\ddt\|\nabla^k u^{n+1}\|_{L^2}^2  
&  = \frac12\int_{\Omega} (\nabla \cdot  u^n)|\nabla^k u^{n+1}|^2\,\dx - \int_\Omega \left(\nabla^k ( u^n \cdot \nabla u^{n+1}) - u^n \cdot \nabla^{k+1} u^{n+1}\right) \cdot \nabla^k u^{n+1}\,\dx\cr
&\quad - \int_\Omega \nabla^k \left(\nabla W_F * \rho^{n+1}  + \nabla W_C *  \bar\varrho_L^{N, n+1} \right)\cdot\nabla^k u^{n+1}\,\dx\cr
&\quad + \int_\Omega \nabla^k \left(\iint_{\Omega \times \R^d} \phi(x-y) (w-u^{n+1}(x)) \bar\mu_L^{N, n+1}(\dy\dw) \right) \cdot \nabla^k u^{n+1}\,\dx\cr
& =: \sum_{i=1}^4 I_i,
\end{aligned}$$
where the terms on the right-hand side satisfy
\begin{align*}
I_1 &\le \|\nabla u^n\|_{L^\infty}\,\|\nabla^k u^{n+1}\|_{L^2}^2 \le C M\, \|\nabla^k u^{n+1}\|_{L^2}^2,\\[1mm]
I_2 &\le C\Bigl(\|\nabla^k u^n\|_{L^2}\,\|\nabla u^{n+1}\|_{L^\infty} + \|\nabla u^n\|_{L^\infty}\,\|\nabla^k u^{n+1}\|_{L^2}\Bigr)\|\nabla^k u^{n+1}\|_{L^2}\\[1mm]
&\le C M\, \Bigl(\|u^{n+1}\|_{H^{s-1}} + \|\nabla^k u^{n+1}\|_{L^2}\Bigr)\|\nabla^k u^{n+1}\|_{L^2},\\[1mm]
I_3 &\le \|\nabla^k \bigl(\nabla W_F * \rho^{n+1} + \nabla W_C * \bar\varrho_L^{N, n+1}\bigr)\|_{L^2}\,\|\nabla^k u^{n+1}\|_{L^2} \le C(M+1)\,\|\nabla^k u^{n+1}\|_{L^2},\\[1mm]
I_4 &\le C\Bigl(\|\nabla^k (\phi *  \langle v \rangle_{\bar\mu^{N, n+1}_L})\|_{L^2} + \|\nabla^k\bigl(u^{n+1}\, (\phi *   \bar\varrho^{N, n+1}_L)\bigr)\|_{L^2}\Bigr)\|\nabla^k u^{n+1}\|_{L^2}\\[1mm]
&\le C\Bigl(1 + \|u^{n+1}\|_{H^{s+1}}\Bigr)\|\nabla^k u^{n+1}\|_{L^2}.
\end{align*}
Summing over $0\le k\le s+1$, we obtain the differential inequality
\[
\ddt\|u^{n+1}\|_{H^{s+1}} \le C(1+M) \Bigl(1+\|u^{n+1}\|_{H^{s+1}}\Bigr).
\]
An application of Gr\"onwall's lemma then gives
\[
\|u^{n+1}(t)\|_{H^{s+1}} \le \|u_0\|_{H^{s+1}}\, e^{C(1+M)T_1} + \Bigl(e^{C(1+M)T_1}-1\Bigr).
\]
By choosing $T_2 \le T_1$ sufficiently small so that the right-hand side is bounded by $M$, we conclude that
\[
\sup_{0\le t\le T_2} \|u^{n+1}(t)\|_{H^{s+1}} \le M.
\]
This completes the proof.
\end{proof}

\begin{proof}[Proof of Proposition \ref{prop_app}]
The existence and uniqueness of regular solutions to the approximate system \eqref{fluid_app} follow from standard theory since the potentials and weight functions are smooth. The uniform-in-$n$ estimates stated in \eqref{uni_app} are a direct consequence of Lemmas \ref{lem_rho} and \ref{lem_u}. Therefore, the proposition is proved.
\end{proof}

%
%
%
%
%
\subsubsection{Proof of Theorem \ref{thm_ex2}}\label{ssec_thm_ex2}
We now complete the proof of Theorem \ref{thm_ex2} by showing that the approximate solutions constructed in Proposition \ref{prop_app} converge to a unique strong solution of \eqref{fluid}.

Subtracting the equations in the iterative scheme \eqref{fluid_app} for consecutive iterates, we obtain
$$\begin{aligned}
    &\ddt (\bar{x}_i^{n+1} -\bar{x}_i^{n} ) = \bar{v}_i^{n+1} - \bar{v}_i^{n}, \quad i = 1,\dots, N, \quad t > 0,\\[1mm]
    &\ddt (\bar{v}_i^{n+1} - \bar{v}_i^{n}) = -(\bar{x}_i^{n+1} - \bar{x}_i^{n})  + \alpha( \langle x \rangle_{\rho^{n+1}} - \langle x \rangle_{\rho^{n}} )- (\bar{v}_i^{n+1} - \bar{v}_i^{n})\cr
    &\hspace{2.5cm} - \nabla W_L *  ( \bar\varrho_{L}^{N, n+1} - \bar\varrho_{L}^{N, n})(\bar{x}_i^{n+1}) - \nabla W_L * (   \bar\varrho_{L}^{N, n}(\bar x_i^{n+1}) -  \bar\varrho_{L}^{N, n}(\bar x_i^{n})) ,\\[1mm]
&\partial_t (\rho^{n+1} - \rho^n) + (u^n - u^{n-1})\cdot \nabla_x \rho^{n+1} + u^{n-1} \cdot \nabla_x (\rho^{n+1} - \rho^n)\cr
&\quad + (\rho^{n+1} - \rho^n)\nabla_x \cdot u^n + \rho^n \nabla_x \cdot (u^n - u^{n-1}) = 0,\cr
&\partial_t (u^{n+1} - u^n) + (u^n - u^{n-1})\cdot\nabla_x u^{n+1} + u^{n-1} \cdot \nabla_x (u^{n+1} - u^n)\cr
&\quad = -\nabla W_F * (\rho^{n+1} - \rho^n) - \nabla W_C *   (\bar\varrho^{N, n+1}_L - \bar\varrho^{N, n}_L) + \phi *   (\langle v \rangle_{\bar\mu^{N, n+1}_L} - \langle v \rangle_{\bar\mu^{N, n}_L}) \cr
&\qquad - (u^{n+1} - u^n) \phi *  \bar\varrho^{N, n+1}_L - u^n \phi *   (\bar\varrho^{N, n+1}_L - \bar\varrho^{N, n}_L).
\end{aligned}$$
Similarly as in Lemma \ref{lem_par}, we find 
\begin{align*}
&\frac12\ddt\left(|\bar{x}_i^{n+1} -\bar{x}_i^{n}|^2 + |\bar{v}_i^{n+1} - \bar{v}_i^{n}|^2 \right) \cr
&\quad \leq |\langle x \rangle_{\rho^{n+1}} - \langle x \rangle_{\rho^{n}}|^2 + \|\nabla W_L\|_{\rm Lip}\left( |\bar{x}_i^{n+1} -\bar{x}_i^{n}|^2 + \frac1N \sum_{j=1}^N |\bar{x}_j^{n+1} -\bar{x}_j^{n}|^2\right).
\end{align*}
On the other hand, 
\[
\ddt\int_{\Omega} (\rho^{n+1} - \rho^n) x\,\dx = \int_{\Omega} \rho^{n+1}(u^n - u^{n-1})\,\dx + \int_{\Omega} (\rho^{n+1} - \rho^n) u^{n-1}\,\dx,
\]
and thus
\begin{align*}
|(\langle x \rangle_{\rho^{n+1}} - \langle x \rangle_{\rho^{n}})(t)| &\leq \|\rho^{n+1}\|_{L^\infty(0,T^*; L^2)} \int_0^t \|(u^n - u^{n-1})(\tau)\|_{L^2}\,\dtau \cr
&\quad + \|u^{n-1}\|_{L^\infty(0,T^*;L^2)} \int_0^t \|(\rho^{n+1} - \rho^n)(\tau)\|_{L^2}\,\dtau.
\end{align*}
This yields
\[
 |(\langle x \rangle_{\rho^{n+1}} - \langle x \rangle_{\rho^{n}})(t)|^2 \leq C\int_0^t \left(\|(u^n - u^{n-1})(\tau)\|_{L^2}^2 + \|(\rho^{n+1} - \rho^n)(\tau)\|_{L^2}^2\right)\dtau.
\]
Hence we have
\begin{align*}
&\frac1{N}\sum_{i=1}^N\left(|(\bar{x}_i^{n+1} -\bar{x}_i^{n})(t)|^2 + |(\bar{v}_i^{n+1} - \bar{v}_i^{n})(t)|^2 \right)   \leq C\int_0^t \left(\|(u^n - u^{n-1})(\tau)\|_{L^2}^2 + \|(\rho^{n+1} - \rho^n)(\tau)\|_{L^2}^2\right)\dtau,
\end{align*}
where $C>0$ is independent of $n$ and $N$.

Note that for $k=0,1$
\[
|\nabla^k \phi *  (\bar\varrho^{N, n+1}_L - \bar\varrho^{N, n}_L)|^2 \leq \|\nabla^k \phi\|_{\rm Lip}^2 \frac1N \sum_{i=1}^N|(\bar{x}_i^{n+1} -\bar{x}_i^{n})(t)|^2
\]
and
\begin{align*}
&|\nabla^k \phi *   (\langle v \rangle_{\bar\mu^{N, n+1}_L} - \langle v \rangle_{\bar\mu^{N, n}_L})|^2 \cr
&\quad \leq \|\nabla^k \phi\|_{\rm Lip}^2 \frac1N \sum_{i=1}^N|(\bar{x}_i^{n+1} -\bar{x}_i^{n})(t)|^2 |\bar v_i^{n+1}|^2 + \|\nabla^k\phi\|_{L^\infty}^2 \frac1N \sum_{i=1}^N  |(\bar{v}_i^{n+1} - \bar{v}_i^{n})(t)|^2 \cr
&\quad \leq C \frac1{N}\sum_{i=1}^N\left(|(\bar{x}_i^{n+1} -\bar{x}_i^{n})(t)|^2 + |(\bar{v}_i^{n+1} - \bar{v}_i^{n})(t)|^2 \right) 
\end{align*}
due to Lemma \ref{lem_par}. Then straightforward computations together with \eqref{uni_app0} give
\[
\|(\rho^{n+1} - \rho^n)(t)\|_{L^2}^2 \leq C\int_0^t \left(\|(\rho^{n+1} - \rho^n)(\tau)\|_{L^2}^2 + \|(u^n - u^{n-1})(\tau)\|_{H^1}^2 \right)\dtau
\]
and
\begin{align*}
&\|(u^{n+1} - u^n)(t)\|_{H^1}^2 \leq C\int_0^t \left(\|(u^{n+1} - u^n)(\tau)\|_{H^1}^2 + \|(\rho^{n+1} - \rho^n)(\tau)\|_{L^2}^2 + \|(u^n - u^{n-1})(\tau)\|_{H^1}^2 \right)\dtau.
\end{align*}
Combining the above two differential inequalities and applying Gr\"onwall's lemma, we obtain
\[
\left( \|\rho^{n+1} - \rho^n\|_{L^2}^2 + \|u^{n+1} - u^n\|_{H^1}^2\right) \leq \frac{Ct^n}{n!}
\]
and thus
\[
\frac1{N}\sum_{i=1}^N\left(|(\bar{x}_i^{n+1} -\bar{x}_i^{n})(t)|^2 + |(\bar{v}_i^{n+1} - \bar{v}_i^{n})(t)|^2 \right)\leq \frac{Ct^n}{n!}
\]
for $0 \leq t \leq T^*$, and this implies that 
\[
(\bar{x}_i^{n},\bar{v}_i^{n}) \mbox{ are uniformly Cauchy sequences in } C[0,T^*]\times C[0,T^*]
\]
and
\[
(\rho^n,u^n) \mbox{ are Cauchy sequences in } C([0,T^*];L^2(\Omega)) \times C([0,T^*];H^1(\Omega)).
\]
Thus we obtain limit functions $(\bar x_i, \bar v_i)$ and $(\rho, u)$ such that
\[
(\bar{x}_i^{n},\bar{v}_i^{n}) \to (\bar x_i,  \bar v_i) \quad \mbox{in} \quad C[0,T^*]\times C[0,T^*]\quad \mbox{as} \quad n \to \infty.
\]
and
\[
(\rho^n,u^n) \to (\rho,  u) \quad \mbox{in} \quad C([0,T^*];L^2(\Omega)) \times C([0,T^*];L^2(\Omega)) \quad \mbox{as} \quad n \to \infty.
\]
Then this together with \eqref{uni_app} and using Gagliardo--Nirenberg interpolation inequality yields the above convergences for $(\rho^n,u^n)$ in fact in $ C([0,T^*];H^{s-1}(\Omega)) \times C([0,T^*];H^s(\Omega))$. In order to show that limiting functions $(\rho,  u)$ are in the desired Sobolev spaces, $C([0,T^*];H^s(\Omega)) \times C([0,T^*];H^{s+1}(\Omega))$, we can use standard functional analytic arguments whose details can be found in \cite[Appendix A]{CCZ16}. For the uniqueness of solutions, let $(\rho_1,u_1)$ and $(\rho_2,u_2)$ be the strong solutions obtained above with the same initial data $(\rho_0,u_0)$. Then by using almost the same argument as above, we get
\[
\|(\rho_1 - \rho_2)(\cdot,t)\|_{L^2}^2 + \|(u_1 - u_2)(\cdot,t)\|_{H^1}^2 \leq C\int_0^t \|(\rho_1 - \rho_2)(\cdot,s)\|_{L^2}^2 + \|(u_1 - u_2)(\cdot,s)\|_{H^1}^2\,\ds,
\] 
for $t \in [0,T^*]$. This leads
\[
\|(\rho_1 - \rho_2)(\cdot,t)\|_{L^2}^2 + \|(u_1 - u_2)(\cdot,t)\|_{H^1}^2 \equiv 0 \quad \mbox{for all} \quad t \in [0,T^*],
\]
and by using the same ideas used in the part of the existence of solutions, we can show that they are indeed the same in $C([0,T^*];H^s(\Omega)) \times C([0,T^*];H^{s+1}(\Omega))$. The uniqueness of $(\bar x_i, \bar v_i)$ simply follows.

%
%
%
%
%
\subsection{Macro-macro system} \label{app_ffs}

In this part, we present the local-in-time existence and uniqueness of regular solutions to the fluid-fluid system introduced in the introduction. We restrict our analysis to the periodic domain $\Omega = \T^d$, and assume the interaction distribution $g$ is absolutely continuous with respect to Lebesgue measure. Abusing notation we write
\begin{align*}
    g(\dxi) = g(\xi) \,\dxi,
\end{align*}
where the Lebesgue density $g$ is assumed to satisfy
\begin{equation*}
    0 < g \in C^0(\T^d), \quad \int_{\T^d} g(\xi) \,\dxi = 1.
\end{equation*}
Here we impose the strict positivity of $g$ in order to ensure that vacuum may be avoided in both the $x$ and $\xi$ variables.

We recall that the governing equations are as follows. For each $\xi\in \T^d$:
\begin{equation}\label{eq: macmac xi_k}
    \begin{split}
        &\p_t \rho_L(t,x,\xi) + \nabla_x\cdot (\rho_L u_L) (t,x,\xi) = 0,\\
        &\p_t (\rho_L u_L)(t,x,\xi) + \nabla_x\cdot (\rho_L u_L \otimes u_L) (t,x,\xi) \\
        &\quad = -(1-\alpha) \rho_L(t,x,\xi) \cdot (x-\xi)  - \alpha \rho_L(t,x,\xi) \cdot (x-\langle x \rangle_{\rho_F}) - (\rho_L u_L) (t,x,\xi) \\
        &\qquad - \rho_L(t,x,\xi) \cdot \iint_{\T^d\times\T^d} \nabla W_L(x-y)\rho_L(y,\xi_*) g(\dxi_*)\dy.
    \end{split}
\end{equation}

For the followers:
\begin{equation} \label{eq: macmac followers dirac}
\begin{split}
    &\p_t (\rho_F)(t,x) + \nabla_x\cdot (\rho_F u_F)(t,x) = 0,\\
        &\p_t (\rho_F u_F)(t,x) + \nabla_x \cdot (\rho_L u_L \otimes u_L)(t,x) \\
        &\quad = -(\rho_F \nabla W_F *_x \rho_F)(t,x) - \rho_F(t,x) \iint_{\T^d\times\T^{d}} \nabla W_C(x-y)\rho_L(y,\xi_*) g(\dxi_*)\dy  \\
        &\qquad +\rho_F(t,x) \iint_{\T^d\times\T^{d}} \phi(x-y)(u_L(y,\xi_*) - u_F(x)) \rho_L(y,\xi_*) g(\dxi_*) \dy .
    \end{split}
\end{equation}

We now clarify the notion of regular solutions used in the theorem below.
\begin{definition}\label{def:reg-fluid-fluid}
Let $s > d/2 + 1$ and $T > 0$. A collection $\big(\rho_L, u_L,\rho_F, u_F\big)$ is called a \emph{regular solution} on $[0,T]$ if the following hold:
\begin{itemize}
    \item[(i)] The leader components satisfy
    
    \[
    \begin{split}
    (\rho_L, u_L) &\in \Big[ L^\infty(\T^d_\xi;C([0,T];H^s(\T^d_x))) \times L^\infty(\T^d_\xi;C([0,T];H^{s+1}(\T^d_x))) \Big] \\
    &\quad \cap \Big[ L^\infty(0,T;W^{1,\infty}(\T^d_\xi; H^{s-1}(\T^d_x))) \times L^\infty(0,T; W^{1,\infty}(\T^d_\xi; H^s(\T^d_x))) \Big] ,
    \end{split}
    \]
    and solve the corresponding subsystem in the sense of distributions.

    \item[(ii)] The follower component satisfies
    \[ (\rho_F, u_F) \in C([0,T]; H^s(\mathbb{T}^d)) \times C([0,T]; H^{s+1}(\mathbb{T}^d)), \]
    and solves the follower equation in the sense of distributions.
\end{itemize}
\end{definition}
Note that according to the regularity conditions, the maps $\xi \mapsto \rho_L(t,\cdot,\xi) \in H^{s-1}(\T^d_x)$ and $\xi\mapsto u_L(t,\cdot,\xi) \in H^s(\T^d_x)$ should be Lipschitz continuous with a Lipschitz constant that is uniformly bounded in $t\in [0,T]$.

We are now in a position to state the local well-posedness result.

\begin{theorem}\label{thm:fluid-fluid}
Let $s > d/2 + 1$.  Suppose that the interaction potentials and communication weights satisfy the assumptions in Theorem \ref{thm_ex2}.
Then, for any constants $0 < N < M$, there exists a positive time $T^*$, depending only on $N$ and $M$, such that if
\begin{align}
&\max\left\{  \|\rho_L(0)\|_{H^s_g}, \,   \|u_L(0)\|_{H^{s+1}_g}, \, \|\rho_F(0)\|_{H^s}, \,\|u_F(0)\|_{H^{s+1}}\right\} \leq N, \nonumber \\
& \xi \mapsto \rho_L(0,\cdot,\xi) \in W^{1,\infty}(\T^d_\xi; H^{s-1}(\T^d_x)), \quad \xi\mapsto u_L(0,\cdot,\xi) \in W^{1,\infty}(\T^d_\xi; H^s(\T^d_x)),  \label{eq: init: Hs cont}
\end{align}
and
\[
\inf_{x \in \mathbb{T}^d} \rho_F(0,x) > 0, \quad  \inf_{(x,\xi)\in \T^d\times\T^d}  \rho_L(0,x,\xi)  > 0,
\]
then the coupled macro-macro system \eqref{eq: macmac xi_k}--\eqref{eq: macmac followers dirac} has a unique regular solution $\big(\rho_L, u_L, \rho_F, u_F\big)$ in the sense of Definition \ref{def:reg-fluid-fluid}, satisfying
\begin{equation}\label{bd_ff}
\max\left\{\sup_{0\le t\le T^*} \|\rho_L(t)\|_{H^s_g}, \, \sup_{0\le t\le T^*} \|u_L(t)\|_{H^{s+1}_g}, \, \sup_{0\le t\le T^*} \|\rho_F(t)\|_{H^s}, \, \sup_{0\le t\le T^*} \|u_F(t)\|_{H^{s+1}} \right\} \leq M.
\end{equation}
\end{theorem}

%
%
%
%
%

\subsubsection{Approximate solutions} Our starting point is again the following approximation scheme. For each $n\in \bbN$, we define $(\rho_L^{n+1}, u_L^{n+1}, \rho_F^{n+1}, u_F^{n+1})$ as the solution to the following system:    
\begin{align}\label{app_ff}
    \begin{aligned}
      &\p_t \rho^{n+1}_L + \nabla_x\cdot (\rho_L^{n+1} u_L^n) = 0,\\
        &\rho_L^{n+1} \p_t u_L^{n+1} + \rho_L^{n+1} (u_L^n \cdot \nabla_x) u_L^{n+1} = - x \rho_L^{n+1} + (1-\alpha) \xi \rho_L^{n+1} + \alpha \rho_L^{n+1}  \langle x \rangle_{\rho_{F^{n+1}}} \\
        &\hspace{6cm} - \rho_L^{n+1} u_L^{n+1} - \rho_L^{n+1} \int_{\T^d} (\nabla W_L * \rho_L^{n+1})(x,\xi_*)g(\dxi_*),   \\
             &\p_t \rho_F^{n+1} + \nabla_x\cdot (\rho_F^{n+1} u_F^n) = 0,\\
        &\rho_F^{n+1} \p_t u_F^{n+1} + \rho_F^{n+1} (u_F^n \cdot \nabla_x) u_F^{n+1}   = \rho_F^{n+1} \nabla W_F * \rho_F^{n+1} - \rho_F^{n+1} \int_{\T^d} (\nabla W_C * \rho_L^{n+1})(x,\xi_*) g(\dxi_*)   \\
        &\hspace{5.5cm} + \rho_F^{n+1} \iint_{\T^d\times\T^d} \phi(x-y)(u_L^{n}(y,\xi_*) - u_F^n(x)) \rho_L^{n+1}(y,\xi_*) g(\dxi_*)\dy.   
        \end{aligned}
    \end{align}
The initial data and the first iterate are defined by
    \begin{align*}
        (\rho_L^{n+1}(0,x,\xi),u_L^{n+1}(0,x,\xi),\rho_F^{n+1}(0,x),u_F^{n+1}(0,x))&= (\rho_L^0(t,x,\xi), u_L^0(t,x,\xi), \rho_F^0(t,x), u_F^0(t,x)) \\
        &= (\rho_L(0,x,\xi), u_L(0,x,\xi), \rho_F(0,x), u_F(0,x)),
    \end{align*}
so that the iterative scheme is initialized consistently with the prescribed data. It is worth mentioning that given $(\rho_L^n,u_L^n,\rho_F^n,u_F^n)$ in the class \eqref{uni_app} as mentioned below, the two densities $(\rho_L^{n+1},\rho_F^{n+1})$ are uniquely determined via the method of characteristics, whereas the velocities $(u_L^{n+1},u_F^{n+1})$ are uniquely determined via classical theory for linear PDE.
    
In parallel to the previous subsection, the primary step is to verify the uniform regularity of the sequence generated by this approximation. To that end, we state the following result.
    \begin{proposition} Suppose that the initial data $(\rho_L(0), u_L(0), \rho_F(0),u_F(0))$ satisfy the conditions in Theorem \ref{thm:fluid-fluid}. Then there exists a time $T^*>0$, depending only on $M$ and $N$, such that the system \eqref{app_ff} admits a sequence of unique regular solutions $\{ (\rho_L^n, u_L^n, \rho_F^n, u_F^n)\}_{n\in\mathbb{N}}$ on $[0,T^*]$ satisfying
\begin{align}\label{uni_app}
\begin{aligned}
&\max\left\{\sup_{n\in\N}\sup_{0\le t\le T^*} \|\rho_L^n(t)\|_{H^s_g}, \, \sup_{n\in\N}\sup_{0\le t\le T^*} \|u_L^n(t)\|_{H^{s+1}_g}, \right. \\
&\hspace{3cm} \left. \sup_{n\in\N}\sup_{0\le t\le T^*} \|\rho_F^n(t)\|_{H^s}, \, \sup_{n\in\N}\sup_{0\le t\le T^*} \|u_F^n(t)\|_{H^{s+1}} \right\} \leq M.
\end{aligned}
\end{align}
    \end{proposition}
\begin{proof}
We proceed under the induction hypothesis that $(\rho_L^n, u_L^n, \rho_F^n, u_F^n)$ satisfy the bounds asserted in \eqref{bd_ff} on a time interval $[0, T]$, where $T > 0$ is to be chosen sufficiently small.

As in Lemma \ref{lem_rho}, applying the method of characteristics yields $T > 0$, dependent only on $M$, such that for all $n \in \mathbb{N}$ and $t \in [0, T]$,
    \begin{align*}
       \rho_L^{n}(t,x,\xi)  > 0, \quad \rho_F^n (t,x) > 0.
    \end{align*}
    Moreover, we obtain 
\begin{align*}
    &\sup_{t\in [0,T]}  \|\rho_L^{n+1}(t)\|_{H^s_g} \le \|\rho_L(0)\|_{H^s_g}\, e^{CMT}, \quad \sup_{t\in [0,T]} \|\rho_F^{n+1}(t)\|_{H^s} \le \|\rho_F(0)\|_{H^s}\, e^{CMT}.
\end{align*}
Choosing $T_1 \leq T$ small enough so that both bounds are less than $M$ ensures the desired control over the densities.
 
Next, we divide the momentum equations by the strictly positive densities $\rho_L^{n+1}$ and $\rho_F^{n+1}$, respectively, yielding
    \begin{equation*}
    \begin{split}
        \p_t u_L^{n+1} + (u_L^n\cdot \nabla_x) u_L^{n+1} &= -x + (1-\alpha)  \xi  + \alpha \langle x \rangle_{\rho_F^{n+1}}  - u_L^{n+1} -  \int_{\T^d} (\nabla W_L * \rho_L^{n+1})(x,\xi_*) g(\dxi_*)  ,\\
        \p_t u_F^{n+1} + (u_F^n\cdot \nabla_x)u_F^{n+1} &= \nabla W_F * \rho_F^{n+1} -  \int_{\T^d} (\nabla W_C * \rho_L^{n+1})(x,\xi_*) g(\dxi_*)  \\
        &\quad +  \iint_{\T^d\times\T^d} \phi(x-y) (u_L^n(y,\xi_*) - u_F^n(x)) \rho_L^{n+1}(y,\xi_*)\, g(\dxi_*) \dy  .
    \end{split}
    \end{equation*}
    
The estimate for $\|u_F^{n+1}(t)\|_{H^{s+1}}$ proceeds analogously to Lemma \ref{lem_u} and yields:
\[
\|u_F^{n+1}(t)\|_{H^{s+1}} \le \|u_F(0)\|_{H^{s+1}}\, e^{C(1+M)T_1} + \Bigl(e^{C(1+M)T_1}-1\Bigr).
\]
We then choose $T_2 \leq T_1$ sufficiently small so that
\[
\sup_{0\le t\le T_2} \|u_F^{n+1}(t)\|_{H^{s+1}}\le M.
\]

We now focus on the estimate for $\|u_L^{n+1}(t)\|_{H^{s+1}_g}$. Differentiating the momentum equation up to order $k\le s+1$ and testing against $\nabla^k u_L^{n+1}$ gives:
\begin{align*}
& \frac12\ddt\|\nabla^k u_L^{n+1}\|_{L^2}^2  + \|\nabla^k u_L^{n+1}\|_{L^2}^2 \cr
&\quad  = \frac12\int_{\T^d} (\nabla \cdot  u_L^n)|\nabla^k u_L^{n+1}|^2\,\dx - \int_{\T^d} \left(\nabla^k ( u_L^n \cdot \nabla u_L^{n+1}) - u_L^n \cdot \nabla^{k+1} u_L^{n+1}\right) \cdot \nabla^k u_L^{n+1}\,\dx\cr
 &\qquad - \int_{\T^d} \nabla^k u_L^{n+1} \cdot \nabla^k (x-  \xi + \alpha(\xi -   \langle x \rangle_{\rho_F^{n+1}}))\,\dx -  \iint_{\T^d\times\T^d}  \nabla^k u_L^{n+1} \cdot \nabla^k (\nabla W_L * \rho_L^{n+1})(x,\xi_*)\,g(\dxi_*)\dx  \cr
&\quad =: \sum_{i=1}^4 J_i,
\end{align*}
where
\begin{align*}
J_1 &\leq CM\|\nabla^k u_L^{n+1}\|_{L^2}^2, \cr 
J_2 &\leq CM(\|u_L^{n+1}\|_{H^{s-1}}+\|\nabla^k u_L^{n+1}\|_{L^2}^2)\|\nabla^k u_L^{n+1}\|_{L^2},  \cr
J_3 &\leq \|\nabla^k u_L^{n+1}\|_{L^2}(1 + (1-\alpha) + \alpha\|\rho_F^{n+1}\|_{L^1}) \leq C\|\nabla^k u_L^{n+1}\|_{L^2}(1 + R + M), \cr
 J_4 &\leq \|\nabla^k u_L^{n+1}\|_{L^2_g}\|\nabla W_L\|_{\calW^{1,1}} \|\rho_L^{n+1}\|_{H^s_g} \leq CM\|\nabla^k u_L^{n+1}\|_{L^2_g}.
\end{align*}
In the above estimate for $J_3$, we used the fact that $\xi\in \T^d$ and therefore $|\xi|\le 1$. For the last estimate in $J_4$, we used the fact that $g$ is a probability measure. Summing over $0\le k\le s+1$ and applying Gr\"onwall's inequality, we deduce
\[
\|u_L^{n+1}(t)\|_{H^{s+1}_g} \le \|u_L(0)\|_{H^{s+1}_g}  \, e^{C(1+R+M)T_1} + \Bigl(e^{C(1+R+M)T_1}-1\Bigr).
\]
Choosing $T^* \le T_2$ small enough to make the right-hand side bounded by $M$ completes the proof.
\end{proof}

%
%
%
%
%
%
%
%
%
%
%
%
\subsubsection{Proof of Theorem \ref{thm:fluid-fluid}} In parallel to the estimates in Section \ref{ssec_thm_ex2}, one can readily check that 
\begin{align*}
    \frac{1}{2}\ddt \|(\rho_L^{n+1}-\rho_L^n)(\cdot,\xi)\|_{L^2}^2 \le C\big(\|(\rho_L^{n+1}-\rho_L^n)(\cdot,\xi)\|_{L^2}^2 + \|(u_L^n - u_L^{n-1})( \cdot, \xi)\|_{H^1}^2\big)
\end{align*}
and
\begin{align*}
    &\frac{1}{2}\ddt \|(u_L^{n+1}- u_L^n)(\cdot,\xi)\|_{H^1}^2  \\
    &\le C \Big(\|(u_L^{n+1} - u_L^n)(\cdot,\xi)\|_{H^1}^2 + \|(u_L^n - u_L^{n-1})(\cdot,\xi)\|_{H^1}^2 \Big)  + C\|(\rho_F^{n+1}-\rho_F^n)(\cdot,\xi)\|_{L^2}^2 + C \|(\rho_L^{n+1} - \rho_L^n)(\cdot,\xi)\|_{L^2}^2.
\end{align*}
In particular, we have
\[
\|(\rho_L^{n+1}-\rho_L^n)(t,\cdot,\xi)\|_{L^2}^2 \leq C\int_0^t \|(u_L^n - u_L^{n-1})(\tau, \cdot, \xi)\|_{H^1}^2\, \dtau
\]
and
\[
\|(u_L^{n+1} - u_L^n)(t,\cdot, \xi)\|_{H^1}^2 \leq C\int_0^t \|(\rho_F^{n+1} - \rho_F^n)(\tau,\cdot)\|_{L^2}^2\,\dtau + C\int_0^t \|(u_L^n - u_L^{n-1})(\tau, \cdot, \xi)\|_{H^1}^2\, \dtau .
\]
This, together with the estimates in Section \ref{ssec_thm_ex2}, yields
\[
    \calH^n(t,\xi) \leq C \int_0^t     \left( \calH^n(\tau,\xi) +   \calH^{n-1}(\tau,\xi)  \right) \,\dtau,
\]
where $C>0$ is independent of $n$ and 
\begin{align*}
    \calH^n(t,\xi) &:= \|(\rho_L^{n+1} - \rho_L^n)(t,\cdot,\xi)\|_{L^2}^2 + \|(\rho_F^{n+1}-\rho_F^n)(t,\cdot)\|_{L^2}^2 \\
    &\quad + \|(u_L^{n+1}- u_L^n)(t,\cdot,\xi)\|_{H^1}^2 + \|(u_F^{n+1} - u_F^n)(t,\cdot)\|_{H^1}^2 .
\end{align*}
We obtain that for all $\xi\in\T^d$, the sequences $(\rho_L^n(\cdot,\xi), u_L^n(\cdot,\xi), \rho_F^n, u_F^n )$ are Cauchy in 
\[
C([0,T^*];L^2(\T^d)) \times C([0,T^*];H^1(\T^d)) \times C([0,T^*];L^2(\T^d)) \times C([0,T^*];H^1(\T^d)).
\]
Thus, we deduce there are limits $\rho_L(\xi), \rho_F, u_L(\xi), u_F$, for which
\begin{equation*}
\begin{split}
 (\rho_L^n(\cdot,\cdot,\xi), u_L^n(\cdot,\cdot,\xi)) &\to (\rho_L(\cdot,\cdot,\xi), u_L(\cdot,\cdot,\xi)) \quad \text{in} \quad C([0,T^*];L^2(\T^d)) \times C([0,T^*]; H^1(\T^d)),\\
(\rho_F^n, u_F^n) &\to (\rho_F, u_F) \quad \text{in} \quad C([0,T^*];L^2(\T^d)) \times  C([0,T^*];H^1(\T^d)).
\end{split}
\end{equation*}
The remaining arguments for establishing existence and uniqueness follow directly from those in Section \ref{ssec_thm_ex2} and are thus omitted. It is further straightforward to check that \eqref{bd_ff} holds, thanks to the uniform bound established in \eqref{uni_app}.

It however remains to prove that, under the assumption of \eqref{eq: init: Hs cont}, there holds
\begin{equation}\label{eq: Hs cont}
\begin{split}
    &\xi\mapsto \rho_L(t,\cdot,\xi) \in L^\infty(0,T^*;\calW^{1,\infty}(\T^d_\xi; H^{s-1}(\T^d_x))), \\
    &\xi\mapsto u_L(t,\cdot,\xi) \in L^\infty(0,T^*; \calW^{1,\infty}(\T^d_\xi; H^s(\T^d_x))).
\end{split}
\end{equation}

We consider $\xi_1,\xi_2\in \T^d$ and define
\begin{equation*}
    \delta \rho_L := \rho_L(\xi_2) - \rho_L(\xi_1), \quad \delta u_L := u_L(\xi_2) - u_L(\xi_1).
\end{equation*}
From the equations \eqref{eq: macmac xi_k}, we find that 
\begin{align*}
    &\p_t \delta \rho_L + \nabla_x \cdot (\delta \rho_L \, u_L(\xi_2)) + \nabla_x\cdot (\rho_L(\xi_1) \, \delta u_L) = 0, \\
    &\p_t \delta u_L + \delta u_L \cdot \nabla_x u_L(\xi_2) + u_L(\xi_1) \cdot \nabla_x \delta u_L = (1-\alpha)(\xi_2 - \xi_1) - \delta u_L.
\end{align*}
A direct computation gives, for all $1\le k \le s$:
\begin{align*}
    \ddt \frac{1}{2} \|\nabla^{k-1} \delta \rho_L\|_{L^2}^2 &=  - \int_{\T^d} \nabla^{k-1} \delta \rho_L \, \nabla^k \delta \rho_L \cdot u_L(\xi_2) \,\dx \\
    &\quad - \int_{\T^d} \nabla^{k-1} \delta\rho_L \big[ \nabla^k (\delta \rho_L \, u_L(\xi_2) ) - \nabla^k \delta\rho_L \cdot u(\xi_2) \big] \,\dx \\
    &\quad - \int_{\T^d} \nabla^{k-1} \delta \rho_L \nabla^{k-1} \nabla\cdot (\rho_L(\xi_1) \,\delta u_L) \,\dx \\
    &=: \sum_{i=1}^3 K_i,
\end{align*} 
where
\begin{equation*}
    \begin{split}
        |K_1| &\le \|\nabla^{k-1} \delta \rho_L\|_{L^2}^2 \|\nabla u_L(\xi_2)\|_{L^\infty} \le C \|\delta \rho_L\|_{H^{s-1}}^2, \\
        |K_2| &\le C \|\nabla^{k-1}\delta\rho_L\|_{L^2} \big( \|\nabla u_L(\xi_2)\|_{L^\infty} \|\nabla^{k-1} \delta \rho_L\|_{L^2} + \|\nabla^k u_L(\xi_2)\|_{L^2} \|\delta \rho_L\|_{L^\infty} \big)  \\
        &\le C \|\delta \rho_L\|_{H^{s-1}}^2, \\
        |K_3| &\le \|\nabla^{k-1} \delta \rho_L\|_{L^2} \|\nabla^k (\rho_L(\xi_1) \,\delta u_L) \|_{L^2} \\
        &\le C \|\delta \rho_L\|_{H^{s-1}} \|\rho_L(\xi_1)\|_{H^s} \|\delta u_L\|_{H^s} \\
        &\le C (\|\delta \rho_L\|_{H^{s-1}}^2 + \|\delta u_L\|_{H^s}^2) .
    \end{split}
\end{equation*}
In the above estimate for $K_2$, we utilized the embedding $\|\delta \rho_L\|_{L^\infty} \lesssim \|\delta \rho_L\|_{H^{s-1}}$, whereas for the estimate of $K_3$ we used the fact that $H^{s}$ is an algebra.

By similar manner, we can obtain the following energy estimate for $\delta u_L$, with $1 \le k \le s+1$:
\begin{equation*}
    \begin{split}
        \ddt \frac{1}{2}\|\nabla^{k-1} \delta u_L\|_{L^2}^2 + \|\nabla^{k-1}\delta u_L\|_{L^2}^2 &= - \int_{\T^d} \nabla^{k-1} \delta u_L \cdot \nabla^{k-1} (\delta u_L \, \nabla u_L(\xi_2)) \,\dx \\
        &\quad - \int_{\T^d} \nabla^{k-1} \delta u_L \, u_L(\xi_1) \, \nabla^{k-1} \nabla \delta u_L \,\dx  \\
        &\quad - \int_{\T^d} \nabla^{k-1} \delta u_L \big( \nabla^{k-1} (u_L(\xi_1) \nabla \delta u_L) - u_L(\xi_1) \nabla^{k-1} \nabla \delta u_L \big) \,\dx \\
        &\quad - (1-\alpha)\mathbf{1}_{k=1} \int_{\T^d} \nabla^{k-1} \delta u_L \,(\xi_2 - \xi_1) \,\dx \\
        &=: \sum_{i=1}^4 L_i,
    \end{split}
\end{equation*}
with
\begin{align*}
    |L_1| &\le \|\nabla^{k-1} \delta u_L\|_{L^2} \|\nabla^{k-1} (\delta u_L \nabla u_L(\xi_2)) \|_{L^2} \\
    &\le \|\delta u_L\|_{H^s} \|\delta u_L\|_{H^s} \|u_L(\xi_2)\|_{H^{s+1}} \\
    &\le C \|\delta u_L\|_{H^s}^2, \\
    |L_2| &\le C \|\nabla u_L(\xi_1)\|_{L^\infty} \|\nabla^{k-1} \delta u_L\|_{L^2}^2, \\
    |L_3| &\le C\|\nabla^{k-1} \delta u_L\|_{L^2} \big( \|\nabla u_L(\xi_1)\|_{L^\infty} \|\nabla^{k-2} \nabla \delta u_L\|_{L^2} + \|\nabla^{k-1} u_L(\xi_1)\|_{L^2} \|\nabla \delta u_L\|_{L^\infty} \big) \\
    &\le C \|\delta u_L\|_{H^s}^2 ,
\end{align*}
the last line owing to the Sobolev embedding $H^s\subset \calW^{1,\infty}$. We finally remark that
\begin{equation*}
    |L_4| = \mathbf{1}_{k=1} \, (1-\alpha) \left|\int_{\T^d} \delta u_L \cdot (\xi_2 - \xi_1) \,\dx \right| \le \mathbf{1}_{k=1} \, (\|\delta u_L\|_{L^2}^2 + |\xi_2 - \xi_1|^2).
\end{equation*}
Collecting the estimates, we find that
\begin{align*}
    \ddt \Big( \|\delta \rho_L\|_{H^{s-1}}^2 + \|\delta u_L\|_{H^{s}}^2 \Big) \lesssim \|\delta \rho_L\|_{H^{s-1}}^2 + \|\delta u_L\|_{H^s}^2 + |\xi_2 - \xi_1|^2,
\end{align*}
and therefore
\begin{equation*}
    \|\delta \rho_L\|_{H^{s-1}}^2 + \|\delta u_L\|_{H^{s}}^2 \lesssim e^{T^*} \Big( \|\delta \rho_L(0)\|_{H^{s-1}}^2 + \|\delta u_L(0)\|_{H^{s}}^2 \Big) + T^* e^{T^*} |\xi_2 - \xi_1|^2.
\end{equation*}
This completes the proof of \eqref{eq: Hs cont} and thus that of the theorem.

\begin{remark}
In the case $\Omega = \R^d$, the presence of the center of mass $\langle x \rangle_{\rho_F}$ of the followers prevents a satisfactory existence proof. Even when $\langle x \rangle_{\rho_F}$ is constant and nonzero, there appears to be no effective way to close estimates for $u_L$ in $L^2(\mathbb{R}^d)$. However, when differentiating the momentum equation, the $\langle x \rangle_{\rho_F}$ term vanishes, which allows us to obtain boundedness of $\nabla u_L$ for a short time, assuming this property holds initially. This argument is similar to that in \cite{CC21}. Provided a strong solution exists, such an estimate for $\nabla u_L$ suffices to deduce the quantitative bound in Theorem \ref{thm:mima-mama} discussed earlier in Remark \ref{rem: en}.
In the special case $\alpha = 0$, the $m_F$ term disappears from the leaders' momentum equation. Then, the argument proceeds exactly as in the previous subsection, leading to a rigorous proof of existence and uniqueness of classical solutions to \eqref{eq: macmac xi_k}--\eqref{eq: macmac followers dirac}.
\end{remark}

%
%
%
%
%

%
%
%
%
%

\section{{Derivation of the feedback control}}\label{appendix:greedycontrol}
In order to derive a feedback control from \eqref{eq:main-cost}, we consider a short-horizon approximation for the optimal control problem \eqref{eq:main-with-u}--\eqref{eq:main-cost}, similarly to the approach proposed in \cite{ACH18,APZ14}. Specifically, we restrict our attention to a reduced time interval $[t, t+h]$ and assume that the control input $\U_i(s)$ remains constant over this interval, i.e., $\U_i(s) \equiv \bar \U_i(t)$ for $s \in [t, t+h]$. This allows us to define a discrete-time cost functional of the form:
\[
J_h(\bar \U) = \frac{1}{N}\sum_{i=1}^N\alpha|x_i(t+h;\bar \U_i)-\langle x \rangle_{\varrho_F^M}(t)|^2 + \beta|x_i(t+h;\bar \U_i)-x_{d_i}|^2+h\gamma |\bar \U_i|^2,
\]
where $x_i(t+h; \bar \U_i)$ is the position of the $i$th leader at time $t+h$, updated via a semi-implicit discretization:
\[
\begin{split}
x_i(t+h;\bar \U_i) &= x_i(t) + hv_i(t+h;\bar \U_i),\\
v_i(t+h;\bar \U_i)&=v_i(t) + h\bar \U_i +\frac{h}{N}\sum_{j=1}^N \nabla W_L (x_j(t)-x_i(t)).
\end{split}
\]
Substituting the velocity into the position update yields:
\begin{equation}\label{eq:x-plus}
x_i(t+h;\bar \U_i) = x_i(t) + hv_i(t)+h^2\bar \U_i +\frac{h^2}{N}\sum_{j=1}^N \nabla W_L (x_j(t)-x_i(t)).
\end{equation}
We now minimize $J_h$ with respect to $\bar \U_i$. Taking the gradient of $J_h$ and setting it to zero yields the optimality condition:
\[
0=\nabla_{\U_i} J_h(\bar \U) = h\alpha(x_i(t+h)-\langle x \rangle_{\varrho_F^M}(t)) + h\beta(x_i(t+h)-x_{d_i}) + \gamma \bar \U_i,
\]
which results in the following expression for $\bar \U_i$
\[
\bar \U_i = -\frac{h}{\gamma + h^3(\alpha +\beta)}\left(\alpha(x_i^+(t)-\langle x \rangle_{\varrho_F^M}(t)) + \beta(x_i^+(t)-x_d)\right),
\]
where $x^+(t)$ is the update \eqref{eq:x-plus} without the control term $h^2\bar \U_i$.

To obtain a feedback control law in continuous time, we perform a scaling of the parameters $\alpha \to \alpha/h$ and $\beta \to \beta/h$, which balances the relative importance of the position terms as $h \to 0$. Taking the formal limit $h \to 0$, we arrive at the feedback control:
\begin{equation}\label{eq:ictrl}
 \U_i(t) = -\frac{1}{\gamma}\left(\alpha(x_i(t) - \langle x \rangle_{\varrho_F^M}(t)) + \beta(x_i(t)-x_d)\right).
\end{equation}
This control depends linearly on the distance between the current position of each leader and both the center of mass of the followers and their individual destinations. The resulting force acts in the direction that simultaneously promotes group cohesion and goal-oriented movement.

We remark that the parameter scaling ensures dimensional consistency: the control $\U_i(t)$ has the dimension of velocity, while the terms $(x_i - \langle x \rangle_{\varrho_F^M})$ and $(x_i - x_{d_i})$ represent positions. The corresponding feedback control used in \eqref{eq:mimi} corresponds to choose $\gamma = 1$, and impose $\alpha + \beta = 1$, treating $\alpha,\beta \in [0,1]$ in the feedback law \eqref{eq:ictrl}.

%
%
%
%
%

%
%
%
%
%

\end{document}